\documentclass{amsart}
\newtheorem{theorem}{Theorem}[section]
\newtheorem{lemma}[theorem]{Lemma}

\theoremstyle{definition}
\newtheorem{definition}[theorem]{Definition}
\newtheorem{example}[theorem]{Example}

\theoremstyle{remark}
\newtheorem{remark}[theorem]{Remark}

\numberwithin{equation}{section}

\usepackage{amsmath,amsthm,verbatim,amssymb,amsfonts,amscd,graphicx,bm}
\usepackage{graphics}
\usepackage{enumitem} 
\usepackage{soul}
\usepackage{caption}
\usepackage{subcaption} 
\usepackage{booktabs}\usepackage{multirow} 
\usepackage{mathtools}
\usepackage{xcolor}
\usepackage{hyperref}
\newcommand{\abs}[1]{\lvert#1\rvert}

\newcommand{\bday}[1]{\langle#1|}
\newcommand{\dday}[1]{|#1\rangle}
\newcommand\bD{{\mathbf{D}}}
\newcommand\sB{{\mathsf{B}}}
\newcommand\bB{{\mathbf{B}}}
\newcommand\cL{{\mathcal{L}}}
\newcommand\cE{{\mathcal{E}}}

\newcommand\cT{{\mathcal{T}}}
\newcommand\cS{{\mathcal{S}}}

\newcommand\cC{{\mathcal{C}}}
\newcommand\cD{{\mathcal{D}}}
\newcommand\cA{{\mathcal{A}}}
\newcommand\cQ{{\mathcal{Q}}}
\newcommand\bj{{\mathbf{j}}}
\newcommand\bu{{\mathbf{u}}}
\newcommand\oh{{\mathbf 0}}

\newcommand\res{{\,\upharpoonright\,}}
\newcommand\nbd{\nobreakdash-\hspace{0pt}}

\DeclareMathOperator{\rank}{rank}
\DeclareMathOperator{\diff}{diff} 
\DeclareMathOperator{\same}{same} 
\DeclareMathOperator{\bigsame}{SAME} 
\DeclareMathOperator{\im}{im} 
\DeclareMathOperator{\spec}{spec}
\DeclareMathOperator{\Ji}{Ji} 
\DeclareMathOperator{\DownJi}{DownJi}

\DeclareMathOperator{\Cond}{Cond} 
\DeclareMathOperator{\seq}{seq} 
\DeclareMathOperator{\env}{env} 
\DeclareMathOperator{\mt}{Maintain}
\DeclareMathOperator{\kl}{Kill} 
\DeclareMathOperator{\tp}{threshold} 
\DeclareMathOperator{\dw}{witness} 
\DeclareMathOperator{\ctr}{ctr} 
\DeclareMathOperator{\visit}{visit} 
\DeclareMathOperator{\enc}{encounter} 
\DeclareMathOperator{\created}{Created}
\DeclareMathOperator{\killed}{Killed} 

\begin{document}
\title{Finite final segments of the d.c.e.\ Turing degrees}

\author[Lempp]{Steffen Lempp}
\address[Lempp]{Department of Mathematics\\
University of Wisconsin\\
480 Lincoln Drive\\
Madison, WI 53706-1325\\
USA}
\email{\href{mailto:lempp@math.wisc.edu}{lempp@math.wisc.edu}}
\urladdr{\url{http://www.math.wisc.edu/~lempp/}}

\author[Yiqun Liu]{Yiqun Liu}
\address[Yiqun Liu]{Office of the President\\
National University of Singapore\\
21 Lower Kent Ridge Road\\
Singapore 119077\\
SINGAPORE}
\email{\href{mailto:liuyq@nus.edu.sg}{liuyq@nus.edu.sg}}

\author[Yong Liu]{Yong Liu}
\address[Yong Liu]{School of Information Engineering\\
Nanjing Xiaozhuang University\\
No. 3601 Hongjing Avenue\\
Nan\-jing 211171\\
CHINA}
\email{\href{mailto:liuyong@njxzc.edu.cn}{liuyong@njxzc.edu.cn}}

\author[Ng]{Keng Meng Ng}
\address[Ng, Wu]{Division of Mathematical Sciences\\
School of Physical and Mathematical Sciences\\
College of Science\\
Nanyang Technological University\\
Singapore 639798\\
SINGAPORE}
\email[Ng]{\href{kmng@ntu.edu.sg}{kmng@ntu.edu.sg}}
\urladdr{\url{http://www.ntu.edu.sg/home/kmng/}}

\author[Wu]{Guohua Wu}
\email[Wu]{\href{guohua@ntu.edu.sg}{guohua@ntu.edu.sg}}
\urladdr{\url{http://www3.ntu.edu.sg/home/guohua/}}

\author[Peng]{Cheng Peng}
\address[Peng]{Institute of Mathematics\\
Hebei University of Technology\\
No.~5340 Xiping Road
Tianjin 300401\\
CHINA}
\email{\href{mailto:pengcheng@hebut.edu.cn}{pengcheng@hebut.edu.cn}}



\thanks{Lempp's research was partially supported by NSF grant DMS-160022
and AMS-Simons Foundation Collaboration Grant 626304. He wishes to thank
National University of Singapore and Nanyang Technological University for
their hospitality during the time much of this research was carried out. Ng
was supported by the Ministry of Education, Singapore, under its Academic
Research Fund Tier~2 (MOE-T2EP20222-0018). Peng's research was partially
supported by NSF of China No.~12271264. Wu was supported by the Ministry of
Education, Singapore, under its Academic Research Fund Tier~1 (RG102/23).}

\begin{abstract}
We prove that every finite distributive lattice is isomorphic to a final
segment of the d.c.e.\ Turing degrees (i.e., the degrees of differences of
computably enumerable sets). As a corollary, we are able to infer the
undecidability of the \(\exists\forall\exists\)-theory of the d.c.e.\ degrees
in the language of partial ordering.
\end{abstract}

\maketitle
\setcounter{tocdepth}{2}

\section{Introduction}
A set~\(A\) is \emph{\(\omega\)\nbd c.e.} if there are computable
functions~\(f\) and~\(g\) such that for each~\(x\), we have \(A(x)=\lim_s
f(x,s)\), \(f(x,0)=0\), and \(\abs{\{s\mid f(x,s+1)\neq f(x,s)\}}\le g(x)\).
\(A\) is \emph{c.e.} if we can choose \(g(x)=1\); \(A\) is \emph{d.c.e.} if
we can choose \(g(x)=2\). A Turing degree is d.c.e.\ if it contains a d.c.e.\
set.

The study of the d.c.e.\ Turing degrees
goes back more than half a century, to the work of Ershov~\cite{Er68a, Er68b,
Er70} in his development of what is now called the Ershov hierarchy.
Cooper~\cite{Co71} first established that there is a properly d.c.e.\ degree,
i.e., a Turing degree containing a d.c.e.\ set but no c.e.\ set. A
breakthrough in the study of the d.c.e.\ degrees was achieved in the
Nondensity Theorem~\cite{CHLLS91} quoted below, which established the
existence of a maximal incomplete d.c.e.\ degree, or, equivalently, that the
two-element linear order can be embedded into the d.c.e.\ degrees as a finite
final segment. This naturally leads to the question of the isomorphism types
of all finite final segments of the d.c.e.\ degrees. If we restrict our
attention to final segments which have a least element, then these must be
finite lattices. In this paper, we will show that all finite distributive
lattices can be realized as final segments of the d.c.e.\ degrees, leading
also to a new and sharper proof of the undecidability of the theory of the
d.c.e.\ degrees (in the language of partial ordering); in particular, we are
able to show that the \(\forall\exists\forall\)\nbd theory of the d.c.e.\
degrees is undecidable.

Our presentation will assume that the reader is familiar with priority
arguments on a tree of strategies, in particular \(\oh'''\)\nbd constructions,
as presented in, e.g., Soare~\cite{So87}. Familiarity with the proof of the
D.C.E.\ Nondensity Theorem would be helpful to understand the present
construction, which builds on this.

We begin by reviewing some definitions and some notation.

\begin{definition}\label{def:embedding}
Let~\(\cL\) be a finite distributive lattice, and denote the upper
semilattice of the d.c.e.\ degrees by~\(\bD_2\). We say \(\bj : \cL \to
\bD_2\) \emph{embeds~\(\cL\) into~\(\bD_2\) as a final segment} if the
following holds:
\begin{enumerate}[label = (\roman*)]
\item\label{it:1}
\(\bj(1) = \oh'\);
\item\label{it:ord}
\(a \le b\) implies \(\bj(a) \le \bj(b)\);
\item\label{it:nonord}
\(a \nleq b\) implies \(\bj(a) \nleq \bj(b)\); and
\item\label{it:onto}
for any d.c.e.\ degree~\(\bu\), there is some \(a \in L\) such that
\(\bj(0) \cup \bu = \bj(a)\).
\end{enumerate}
\end{definition}

The D.C.E.\ Nondensity Theorem can now be stated as follows:

\begin{theorem}[D.C.E.\ Nondensity Theorem (Cooper, Harrington, Lachlan,
Lempp, Soare~\cite{CHLLS91})]\label{thm:maxdre2} There is a maximal
incomplete d.c.e.\ Turing degree~\(\bm{d}\); in particular, the d.c.e.\
Turing degree are not densely ordered.
\end{theorem}

In other words, the 2-element lattice \(\cL = \{0,1\}\) can be embedded into
the d.c.e.\ degrees as a final segment (see Definition~\ref{def:embedding}).
The case when~\(\cL\) is a Boolean algebra was covered by Lempp, Yiqun Liu,
Yong Liu, Ng, Peng, and Wu in~\cite{Li17}. This suggests that
Theorem~\ref{thm:maxdre2} can be generalized to more general finite lattices.
To this end, we prove the following theorem in this paper:

\begin{theorem}\label{thm:bool}
If~\(\cL\) is a finite distributive lattice, then~\(\cL\) can be embedded
into d.c.e.\ degrees as a final segment.
\end{theorem}

We conjecture that this theorem can be extended to a much wider class of
finite lattices, in particular to the join-semidistributive or even the
so-called interval dismantlable lattices introduced by Adaricheva, Hyndman,
Lempp, and Nation~\cite{AHLN18}; at this point, extending it to all finite
lattices appears out of reach but not impossible.

We note an immediate important consequence of Theorem~\ref{thm:bool}, a
sharper undecidability result for the first-order theory of the d.c.e.\
degrees in the language of partial ordering; this theory had previously been
proven to be undecidable by Cai/Shore/Slaman \cite{CSS12}; a closer analysis
of their proof yields the undecidability of the
\(\forall\exists\forall\exists\)-theory. Our result strengthens this:

\begin{theorem}\label{thm:undec}
The \(\exists\forall\exists\)-theory of the d.c.e.\ degrees in the language of
partial ordering is undecidable.
\end{theorem}

\begin{proof}
The proof is almost exactly the same as the proof for the Turing degrees by
Lachlan \cite[Section~3]{La68} (see Lerman \cite[Theorem VI.4.6]{Le83} for a
textbook exposition): Let~\(S\) be the set of the \(\le\)-sentences true of all
distributive lattices, and~\(F\) the set of the \(\le\)-sentences true of all
finite distributive lattices; then by Ershov/Taitslin \cite{ET63}, there is
no computable set~\(R\) with \(F \subseteq R \subseteq S\). Let \(\theta(x)\)
be a
\(\le\)-formula stating that the interval \([x,1]\) is a distributive lattice;
then by Theorem~\ref{thm:bool}, the set~\(H\) of \(\le\)-sentences of the form
\(\forall x\, (\theta(x) \rightarrow \varphi(x))\) is a set of
\(\exists\forall\exists\)-sentences and satisfies \(F \subseteq H \subseteq S\).
\end{proof}

We now recall the definitions of Boolean algebras and distributive lattices
in Section~\ref{sec:lattice}, where we also discuss other useful properties.

\section{Basics on Lattice Theory}\label{sec:lattice}

In this section, we define and cover the relevant basic notions and
properties of lattices. We follow~\cite{Gr11} for the basic definitions. We
restrict ourselves to \emph{finite} lattices with least element~0 and
greatest element~1.

A lattice~\(\cL\) can be thought of as a poset \((L, \le)\), where any finite
subset has an infimum and a supremum, or as an algebraic structure \((L,
\vee, \wedge)\), where \(x \le y\) iff \(x \wedge y = x\) iff \(x \vee y =
y\).

A lattice~\(\cL\) is finite if~\(|L|\) is finite. A lattice~\(\cL\) is
\emph{distributive} if it satisfies the following for all \(a, b, c \in L\):
\begin{align*}
a \wedge(b \vee c) &= (a \wedge b)\vee (a \wedge c),\\
a \vee(b \wedge c) &= (a \vee b)\wedge (a \vee c).
\end{align*}

For \(a \le b\), we call the set \([a, b] = \{x \in L \mid a \le x \le b\}\)
an \emph{interval}. In particular, an interval is always a sublattice
of~\(\cL\) (note that we do not require \(0, 1 \in [a,b]\), i.e., the least
and greatest element of \([a,b]\) and of~\(\cL\) need not be the same,
respectively).

We write \(a \prec b\) if \(a<b\) and there is no element~\(c\) such that
\(a<c<b\). We call~\(a\) an \emph{atom} if \(0 \prec a\). We call \(a \in L\)
\emph{join-irreducible} in~\(\cL\) if \(a \neq 0\) and \(a = b \vee c\)
implies \(a=b\) or \(a=c\). If~\(a\) is join-irreducible, then we let~\(a_*\)
denote the unique element such that \(a_*\prec a\). Let \(\Ji(\cL)\) be the
set of all join-irreducibles of~\(\cL\). Let \(\spec(a) = \{b \in \Ji(\cL)
\mid b \le a\}\).

Let \(\DownJi(\cL)\) be the collection of sets \(A \subseteq \Ji(\cL)\) such
that if \(x, y \in \Ji(\cL)\) with \(x<y\) and \(y \in A\), then \(x \in A\).
Observe that \((\DownJi(\cL), \cap, \cup)\) is a distributive lattice with
least element~\(\varnothing\) and greatest element \(\Ji(\cL)\). The
following theorem is worth mentioning.

\begin{theorem}[see~\cite{Gr11}]
Let~\(\cL\) be a finite distributive lattice. Then the map
\[
\spec: a \mapsto \spec(a)
\]
is an isomorphism between~\(\cL\) and \(\DownJi(\cL)\).
\end{theorem}

In particular, a useful fact to keep in mind is that \(\spec(a \vee b) =
\spec(a)\cup \spec(b)\). As one of the equivalent definitions of finite
Boolean algebras, we say that a finite lattice~\(\cL\) is a \emph{Boolean
algebra} if~\(\cL\) is a finite distributive lattice such that for any \(c
\in \Ji(\cL)\), \(0 \prec c\).

We continue with definitions. An element~\(a\) is a \emph{complement}
of~\(b\) if \(a \vee b = 1\) and \(a \wedge b = 0\). (The complement of~\(a\)
need not always exist, even in finite distributive lattices.)  An
element~\(a^*\) is a \emph{pseudocomplement} of~\(a\) if for all~$x$, \(a
\wedge x = 0\) iff \(x \le a^*\). Unlike the complement of~\(a\), if~\(a^*\)
exists, it is unique; but in general, \(a^{**}\neq a\). If for all \(a \in
L\),~\(a^*\) exists, we call~\(\cL\) \emph{pseudocomplemented}.

The \emph{pseudocomplement of~\(a\) relative to~\(b\)} is the (unique)
element~\(a*b\) such that for all~$x$, \(a \wedge x \le b\) iff \(x \le
a*b\). If for all \(a, b \in L\),~\(a*b\) exists, we call~\(\cL\)
\emph{relatively pseudocomplemented}.

Observe that if~\(\cL\) is a finite distributive lattice, then~\(\cL\) is
relatively pseudocomplemented. To see this, let \(S = \{x \in L \mid a \wedge
x \le b\}\); this set is nonempty since \(b \in S\), and by
distributivity,~$S$ is closed under join and thus has a greatest element,
namely,~\(a*b\). Note that \(a^* = a*0\), and hence a finite distributive
lattice is also pseudocomplemented. It is not difficult to show that
if~\(\cL\) is relatively pseudocomplemented, then~\(\cL\) is distributive.

A subset \(C = \{a_0, a_1,\ldots , a_n\} \subseteq L\) is a \emph{chain of
length \(n+1\)} if \(a_0 \prec a_1 \prec \cdots \prec a_n\).

\begin{theorem}[see~\cite{Gr11}]
Let~\(\cL\) be finite distributive lattice. Then every maximal chain
in~\(\cL\) is of length \(|\Ji(\cL)|+1\).
\end{theorem}

Thus, every chain from~\(0\) to~\(1\) is of the same length. For a
non-example, any maximal chain in the five-element modular nondistributive
lattice~\(M_3\) is of length~\(3\), but \(|\Ji(M_3)|+1 = 4\). (Here, \(M_3 =
\{0, a_1,a_2,a_3, 1 \}\) with \(0 < a_i < 1\) and \(a_i \nleq a_j\) for
distinct \(i, j \in \{1, 2, 3\}\).)

Moreover, we can define the \emph{rank} of \(a \in L\) by setting \(\rank(a)
= n\) if some chain from~\(0\) to~\(a\) has length \(n+1\) but no chain has
length \(n+2\). In particular, if~\(\cL\) is finite distributive, then
\(\rank(a) = |\spec(a)|\). An element~\(a\) is an atom iff \(\rank(a) = 1\).

We list some useful facts.

\begin{lemma}\label{lem:useful}
Let~\(\cL\) be a finite distributive lattice and \(a, b, c \in L\).
\begin{enumerate}
\item\label{it:useful1}
\(a \le b*a\).
\item
\(a = \bigvee \spec(a)\).
\item
If \(a \le b \le c\), then \({(b^*)}^{[a,c]} = (b*a)\wedge
c\).\footnote{Here, \({(b^*)}^{[a,c]}\) is the pseudocomplement of~\(b\)
computed in \([a,c]\).}
\item\label{it:useful2}
If \(c \le a\), then \(a \wedge (a*c) = c\).
\item\label{it:useful-mon}
If \(a, b \in \Ji(\cL)\) and \(a \le b\), then \(a*a_* \le b*b_*\) and \(b
\nleq a*a_*\).
\end{enumerate}
\end{lemma}

\begin{proof}
(1) Since \(b \wedge a \le a\), we have \(a \le b*a\) by definition.

(2) Clearly, \(\bigvee \spec(a) \le a\). Suppose that \(b = \bigvee \spec(a) <
a\) and consider \(A = \{x \le a \mid x \nleq b\}\), so \(a \in A\). Any \(x
\in A\) is join-reducible, hence \(x = y \vee z\) for some \(y, z < x\),
where at least one of them must be in~\(A\). Therefore, inductively, we see
that~\(A\) is infinite, a contradiction.

(3) Since \((b*a)\wedge c \in [a,c]\), we work in \([a,c]\) and see that for
any \(x \in [a,c]\), \(b \wedge x = a\) iff \(x \le (b*a)\wedge c\), which
follows directly from the definition of~\(b*a\) and the fact that \(x \in
[a,c]\).

(4) Since \(c \le a \wedge (a*c) \le c\).

(5) Observe that \(b \wedge(a*a_*)\le b\). Suppose that \(b \wedge(a*a_*) =
b\), then \(b \le a*a_*\), which implies \(a = a \wedge b \le a_*\), a
contradiction. Hence \(b \wedge(a*a_*)\le b_*\) and so \(a*a_* \le b*b_*\).
Suppose that \(b \le a*a_*\), then \(a = a \wedge b \le a_*\), again a
contradiction.
\end{proof}

\section{The Requirements and Conflicts}\label{sec:req-confl}

We introduce a preliminary version of our requirements and show that they
suffice for our theorem in Section~\ref{sub:req}. Then we rephrase the
\(S\)\nbd requirements in Section~\ref{sub:S-req} and simplify the \(R\)\nbd
requirements in Section~\ref{sub:R-req}. After achieving this final (more
useful) version of our requirements, we give three examples in
Section~\ref{sub:example}. Finally, we discuss the conflicts between
requirements in Section~\ref{sub:conflicts}

\subsection{The setup}\label{sub:setup}

Recall Definition~\ref{def:embedding}: Given a finite distributive
lattice~\(\cL\), our job is to exhibit an embedding \(\bj : \cL \to \bD_2\)
onto a final segment of the d.c.e.\ degrees. We will define a map \(j : \cL
\to D_2\) (into the set~\(D_2\) of d.c.e.\ sets) and then let \(\bj(a) =
\deg_T(j(a))\).

For each element \(a \in L\), we will construct a d.c.e.\ set~\(A\) and
map~\(a\)
to \(j(a) = A\). The set assigned to~\(0\) will be called~\(E\). Now
our~\(j\) will
be the following map:
\[
j: a \mapsto E \oplus (\bigoplus \{B \mid b \in \spec(a)\})
\]

Of course, whether~\(\bj\) embeds~\(\cL\) into~\(\bD_2\) as a final segment
depends on our choice of the set~\(E\) and of the sets~\(A\) for each \(a \in
\Ji(\cL)\). We will construct our sets meeting certain requirements.

We start with a global requirement
\[
G: K = \Theta^{j(1)},
\]
where~\(\Theta\) is a functional we build. This ensures
Definition~\ref{def:embedding}\ref{it:1} that \(j(1) \equiv_T K\).

Definition~\ref{def:embedding}\ref{it:ord} is automatic, because if \(a \le
b\), then \(\spec(a) \subseteq \spec(b)\), and so by the definition of~\(j\),
\(j(a) \le_T j(b)\). It is worth noting that \(\spec(a \vee b) = \spec(a)
\cup \spec(b)\), so we have \(j(a \vee b) \equiv_T j(a) \oplus j(b)\).
Therefore, \(\im(\bj)\) is a finite upper semilattice with least and greatest
element. (Since~\(\cL\) is finite and has a least element, and since~\(\bj\)
will be onto a final segment of~\(\bD_2\), \(\im(\bj)\) will also
automatically be a lattice.)

Next, we explain how to ensure Definition~\ref{def:embedding}\ref{it:nonord}
and~\ref{it:onto}.

\subsection{Initial version of our requirements}\label{sub:req}

In order to facilitate notation, we will agree on the following: For an
interval \(S = [a,b] \subseteq L\), we define
\[
    j[S] = \{U \in D_2 \mid j(a) \le_T U \le_T j(b)\}.
\]
\(\bj[S]\) is the degree version of \(j[S]\). In particular, if \(S = \{a\}\)
is a singleton, then \(\bj[S] = \{\deg_T(j(a))\} = \bj(a)\).

For all d.c.e.\ sets~\(U\), we have a set of requirements \(S_\sigma(U)\),
indexed by a finite set of nodes \(\sigma \in 2^{<\omega}\). For each partial
computable functional~\(\Psi\), we also have a set of
requirements~\(R_\sigma(\Psi)\), indexed by the same finite set of nodes
\(\sigma \in 2^{<\omega}\). For now, our requirements are defined recursively
as follows.

Let \(L_\lambda = L = [0,1]\). Suppose that we have defined \(L_\sigma =
[p_\sigma, q_\sigma]\). If \(p_\sigma = q_\sigma\), then we stop. Otherwise,
we choose \(p_{\sigma0}\) such that \(p_\sigma \prec p_{\sigma0} \le
q_\sigma\). We now let \(q_{\sigma0} = q_\sigma\), \(p_{\sigma1} =
p_\sigma\), and \(q_{\sigma1} = {(p_{\sigma0}^*)}^{L_\sigma}\). Set
\(L_{\sigma0} = [p_{\sigma0},q_{\sigma0}]\) and \(L_{\sigma1} =
[p_{\sigma1},q_{\sigma1}]\) and continue.

\begin{definition}\label{def:T_L}
We let \(T_\cL = \{\sigma \mid L_\sigma \text{ is defined}\}\) and \(T'_\cL =
\{\sigma \mid \sigma \text{ not a leaf of \(T_\cL\)}\}\).
\end{definition}

We then state our local requirements for now in preliminary form as follows,
for each \(\sigma \in T'_\cL\), d.c.e.\ set~\(U\) and partial computable
functional~\(\Psi\):
\begin{align*}
S_\sigma(U)&: j(p_\sigma)\oplus U \in j[L_\sigma] \to j(p_{\sigma0}) =
   \Gamma_\sigma^{j(p_\sigma)\oplus U} \lor
   U = \Delta_\sigma^{j(q_{\sigma1})}\\
R_\sigma(\Psi)&: j(p_{\sigma0}) \neq \Psi^{j(q_{\sigma1})}
\end{align*}

This finishes the definition of our requirements. Thus we have, for each
d.c.e.\ set~\(U\) and each \(\sigma \in T'_\cL\), a
requirement~\(S_\sigma(U)\). For each partial computable functional~\(\Psi\)
and each \(\sigma \in T'_\cL\), we have a requirement \(R_\sigma(\Psi)\).
When the oracle~$U$ is clear from the context, we say that~\(\Gamma_\sigma\)
is \emph{total and correct} if \(j(p_{\sigma 0}) =
\Gamma_\sigma^{j(p_\sigma)\oplus U}\), and that~\(\Delta_\sigma\) is
\emph{total and correct} if \(U = \Delta_\sigma^{j(q_{\sigma 1})}\).

In the rest of the section, we prove the following lemma.

\begin{lemma}\label{lem:req_suffice}
If all requirements are met, then \(\bj : \cL \to \bD_2\) embeds~\(\cL\)
into~\(\bD_2\) as a final segment.
\end{lemma}

We assume that all requirements are met throughout the rest of the section in
a sequence of lemmas.

We call \(S_\sigma(U)\) \emph{active} if \(j(p_\sigma)\oplus U \in
j[L_\sigma]\). Otherwise \(S_\sigma(U)\) is satisfied trivially.

\begin{lemma}\label{lem:S-active}
Let~\(S_\sigma(U)\) be active.
\begin{enumerate}
\item
If~\(\Gamma_\sigma\) is total, then \(j(p_\sigma)\oplus U \equiv_T
j(p_{\sigma0})\oplus U \in j[L_{\sigma0}]\). So if \(|L_{\sigma0}|>1\),
then \(S_{\sigma0}(U)\) is active.
\item
If~\(\Delta_\sigma\) is total, then \(j(p_\sigma)\oplus U \equiv_T
j(p_{\sigma1})\oplus U \in j[L_{\sigma1}]\). So if \(|L_{\sigma1}|>1\),
then \(S_{\sigma1}(U)\) is active.
\end{enumerate}
\end{lemma}

\begin{proof}
(1) We have \(j(p_{\sigma0}) \le_T j(p_{\sigma0}) \oplus U \le_T
j(p_{\sigma}) \oplus U \le_T j(q_\sigma) = j(q_{\sigma 0})\), where the
second reducibility follows from the fact that~\(\Gamma_\sigma\) is total,
and the third reducibility follows from the fact that~\(S_\sigma(U)\) is
active. But \(p_\sigma \le p_{\sigma0}\) implies \(j(p_\sigma) \oplus U \le_T
j(p_{\sigma0}) \oplus U\), so the second reduction is in fact an equivalence.

(2)~\(\Delta_\sigma\) states that \(U \le_T j(q_{\sigma1})\). Since
\(p_{\sigma1}= p_\sigma \le q_{\sigma1}\), we have \(j(p_{\sigma1})\le_T
j(p_{\sigma1})\oplus U \le_T j(q_{\sigma1})\).
\end{proof}

\begin{lemma}
For \(\sigma \in T'_\cL\), we have \(L_\sigma = L_{\sigma0} \sqcup
L_{\sigma1}\) and \(\bj[L_\sigma] = \bj[L_{\sigma0}] \sqcup
\bj[L_{\sigma1}]\).
\end{lemma}

\begin{proof}
For the first equation, \(L_{\sigma0} \cup L_{\sigma1} \subseteq L_\sigma\)
is obvious.

Given \(a \in L_\sigma\), if \(p_{\sigma0} \le a\), then \(a \in
L_{\sigma0}\). Otherwise, \(p_{\sigma0}\wedge a = p_\sigma\), so \(a \le
q_{\sigma1}\) and thus \(a \in L_{\sigma1}\). Hence
\(L_\sigma=L_{\sigma0}\cup L_{\sigma1}\).

Now suppose that \(b \in L_{\sigma0} \cap L_{\sigma1}\). Then \(p_{\sigma0}
\le b \le q_{\sigma1}\), contradicting \(q_{\sigma1} =
{(p_{\sigma0}^*)}^{L_\sigma}\). Hence \(L_\sigma = L_{\sigma0} \sqcup
L_{\sigma1}\).

For the second equation, \(\bj[L_{\sigma0}] \cup \bj[L_{\sigma1}] \subseteq
\bj[L_\sigma]\) is trivial.

Given \(\deg_T(E \oplus U) \in \bj[L_\sigma]\), we have that the degree of
\(E \oplus U \equiv_T j(p_\sigma) \oplus U\) is in \(\bj[L_\sigma]\).
So~\(S_\sigma(U)\) is active. If~\(\Gamma_\sigma\) is total, then
\[
E \oplus U \equiv_T j(p_\sigma) \oplus U \equiv_T j(p_{\sigma0}) \oplus U
  \in j[L_{\sigma0}].
\]
If~\(\Delta_\sigma\) is total, then
\[
E \oplus U \equiv_T j(p_\sigma) \oplus U \equiv_T j(p_{\sigma1}) \oplus U
  \in j[L_{\sigma1}].
\]
Therefore \(\bj[L_\sigma] = \bj[L_{\sigma0}] \cup \bj[L_{\sigma1}]\).

Now suppose that \(\deg_T(E \oplus U) \in \bj[L_{\sigma0}] \cap
\bj[L_{\sigma1}]\). Then we would have
\[
j(p_{\sigma0}) \le_T E \oplus U \le_T j(q_{\sigma1}),
\]
contradicting the \(R_\sigma(\Psi)\)\nbd requirements.
\end{proof}

\begin{lemma}
\(\bj\) is injective.
\end{lemma}

\begin{proof}
Suppose that \(a \neq b\). Let~\(\sigma\) be the longest such that \(a, b \in
L_\sigma\). So \(|L_\sigma| \ge 2\), and thus \(L_\sigma = L_{\sigma0} \sqcup
L_{\sigma1}\). Without loss of generality, we assume \(a \in L_{\sigma0}\)
and \(b \in L_{\sigma1}\). Hence \(\bj(a) \in \bj[L_{\sigma0}]\) and \(\bj(b)
\in \bj[L_{\sigma1}]\) and so \(j(a) \not\equiv_T j(b)\).
\end{proof}

\begin{lemma}
If \(a, b \in L\), then \(a \nleq b\) implies \(j(a) \nleq_T j(b)\).
\end{lemma}

\begin{proof}
We suppose towards a contradiction that \(j(a) \le_T j(b)\). Now let \(c = b
\vee a\). Note that \(b<c\) since otherwise we would have \(a \le b\). So
\(j(b) <_T j(c)\) since~\(\bj\) is injective and order-preserving. But then
\[
j(c)\equiv_T j(b \vee a)\equiv_T j(b)\oplus j(a)\equiv_T j(b) <_T j(c),
\]
a contradiction.
\end{proof}

Thus we have Definition~\ref{def:embedding}\ref{it:nonord}.

\begin{definition}
Call~\(L_\sigma\) \emph{\(U\)\nbd active} if \(j(0)\oplus U \in
j[L_\sigma]\).
\end{definition}

\begin{lemma}
\hspace{1em}
\begin{enumerate}
\item
If~\(L_\sigma\) is \(U\)\nbd active, then \(j(0)\oplus U \equiv_T
j(p_\sigma)\oplus U\). So if \(|L_\sigma|\ge 2\), then~\(S_\sigma(U)\) is
active.
\item
Suppose that~\(S_\sigma(U)\) is active. If~\(\Gamma_\sigma\) is total,
then~\(L_{\sigma0}\) is \(U\)\nbd active. If~\(\Delta_\sigma\) is total,
then~\(L_{\sigma1}\) is \(U\)\nbd active.
\end{enumerate}
\end{lemma}

\begin{proof}
(1) We have \(j(p_\sigma)\le_T j(0)\oplus U \le_T j(p_\sigma)\oplus U\),
where the first reducibility follows from the fact that~\(L_\sigma\) is
\(U\)\nbd active.

(2) Apply Lemma~\ref{lem:S-active} inductively.
\end{proof}

It is easy to see \(C'_\cL(U) \coloneqq \{\sigma \in T'_\cL \mid S_\sigma(U)
\text{ is active}\}\) is a maximal chain in~\(T'_\cL\). We also have that
\(C_\cL(U) \coloneqq \{\sigma \in T_\cL \mid L_\sigma \text{ is \(U\)\nbd
active}\}\) is a maximal chain in~\(T_\cL\). Clearly,~\(C'_\cL(U)\) is
\(C_\cL(U)\) without its longest node.

\begin{lemma}
For each d.c.e.\ set~\(U\), there exists \(a \in L\) such that \(j(0) \oplus
U \equiv_T j(a)\).
\end{lemma}

\begin{proof}
Let~\(\sigma\) be the longest string in~\(C_\cL(U)\), then \(L_\sigma =
\{a_\sigma \}\) and \(j(0) \oplus U \equiv_T j(a_\sigma)\).
\end{proof}

So we have Definition~\ref{def:embedding}\ref{it:onto}, and hence
Lemma~\ref{lem:req_suffice} has been proved.

\subsection{\texorpdfstring{The \(S\)-requirements rephrased}
{The S-requirements rephrased}}\label{sub:S-req}

We need to rephrase our requirements to better fit our construction. Our
initial rephrasing of the \(S\)\nbd requirements is summarized in the
following

\begin{remark}\label{rem:p_sigma}
\hspace{1em}
\begin{enumerate}
\item\label{it:jpsigma0}
If~\(S_\sigma(U)\) is active, then \(j(0) \oplus U \equiv_T j(p_\sigma)
\oplus U\).

Since \(j(0)=E\), we can replace the disjunct \(j(p_{\sigma0}) =
\Gamma^{j(p_\sigma)\oplus U}\) in the conclusion of the \(S_\sigma(U)\)\nbd
requirement by
\[
j(p_{\sigma0})=\Gamma^{E \oplus U}.
\]
\item\label{it:psig,a0}
Since \(p_\sigma \prec p_{\sigma0}\), there is a unique \(c_\sigma \in
\Ji(\cL)\) such that
\[
\spec(p_{\sigma0}) = \spec(p_\sigma)\cup \{c_\sigma \}
\]
and
\[
p_{\sigma0} = p_\sigma \vee c_\sigma.
\]
Hence
\[
j(p_{\sigma0}) \equiv_T j(p_\sigma)\oplus C_\sigma.
\]
So if~\(S_\sigma(U)\) is active, we can rephrase the \(S_\sigma(U)\)\nbd
requirement as
\[
S_\sigma(U): C_\sigma = \Gamma^{E \oplus U} \lor
             U = \Delta_\sigma^{j(q_{\sigma1})}
\]
\item\label{it:psigma0qsigma1}
Since \(j(p_{\sigma0}) \equiv_T j(p_\sigma)\oplus C_\sigma\) and
\(j(p_\sigma)\le_T j(q_{\sigma1})\), we have
\[
j(p_{\sigma0})\nleq_T j(q_{\sigma1}) \text{ iff }
  C_\sigma \nleq_T j(q_{\sigma1}).
\]
Hence we can rewrite the requirement~\(R_\sigma(\Psi)\) as
\[
R_\sigma(\Psi) : C_\sigma \neq \Psi^{j(q_{\sigma1})}.
\]
\end{enumerate}
\end{remark}

\begin{definition}\label{def:Fsigma}
    For \(\sigma \in T_L\), we define~\(F_\sigma(U)\) recursively as follows,
    \begin{enumerate}
    \item \(F_\lambda(U) = \varnothing\),
    \item \(F_{\sigma0}(U) = F_\sigma(U) \cup \{\Gamma_\sigma \}\),
    \item \(F_{\sigma1}(U) = F_\sigma(U) \cup \{\Delta_\sigma \}\),
    \end{enumerate}
    where~\(\Gamma_\sigma\) stands for the requirement
    \(C_\sigma=\Gamma_\sigma^{E \oplus U}\), and~\(\Delta_\sigma\) stands for
    the requirement \(U = \Delta_\sigma^{j(q_{\sigma1})}\).
\end{definition}

We say~\(F_\sigma(U)\) is \emph{satisfied} if all functionals in
\(F_\sigma(U)\) are total and correct. Given two
functionals~\(\Lambda_\sigma\) and~\(\Lambda_\tau\) in~\(F_\eta(U)\), we
say~\(\Lambda_\sigma\) is \emph{higher than}~\(\Lambda_\tau\)
(and~\(\Lambda_\tau\) is \emph{lower than}~\(\Lambda_\sigma\)) if \(\sigma
\subset \tau\). (So~\(T_\cL\) is ``growing downward''.)

\begin{lemma}\label{lem:U-active}
\(L_\sigma\) is \(U\)\nbd active iff~\(F_\sigma(U)\) is satisfied.
\end{lemma}

\begin{proof}
We proceed by induction on \(\sigma \in T_\cL\).~\(L_\lambda\) is obviously
\(U\)\nbd active and \(F_\lambda(U)=\varnothing\).

(1) \(L_{\sigma0}\) is \(U\)\nbd active iff~\(L_\sigma\) is \(U\)\nbd active
and~\(\Gamma_\sigma\) is total iff~\(F_\sigma(U)\) is satisfied
and~\(\Gamma_\sigma\) is total iff~\(F_{\sigma0}(U)\) is satisfied.

(2) \(L_{\sigma1}\) is \(U\)\nbd active iff~\(L_\sigma\) is \(U\)\nbd active
and~\(\Delta_\sigma\) is total iff~\(F_\sigma(U)\) is satisfied
and~\(\Delta_\sigma\) is total iff~\(F_{\sigma1}(U)\) is satisfied.
\end{proof}

\begin{definition}
Let~\([T_\cL]\) be the set of all leaves of the finite tree~\(T_\cL\)
(Definition~\ref{def:T_L}). We lexicographically order~\([T_\cL]\), denoted
by~\(<\).
\end{definition}

Now we can write our \(S(U)\)\nbd requirements succinctly as follows:
\[
S(U): (\exists \eta\in [T_\cL])\; F_\eta(U)
\]
Note that we stated the \(S(U)\)\nbd requirement above in a more compact
form; it is equivalent to the previous list of requirements for various
\(\sigma \in T'_\cL\) by induction on~\(|\sigma|\).

\(j(q_{\sigma 1})\), as it appeared in~\(\Delta_\sigma\), will be simplified
in Section~\ref{sub:R-req} and Section~\ref{sub:conflicts}.

\subsection{\texorpdfstring{The \(R\)-requirements rephrased}%
{The R-requirements rephrased}}\label{sub:R-req}

We now want to identify~\(q_{\sigma1}\) more carefully for each \(\sigma \in
T'_\cL\).

\begin{lemma}\label{lem:*}
If \(a \le p\) and \(a \prec c\) and \(c \nleq p\), then \((c \vee p)*c =
p*a\)
\end{lemma}

\begin{proof}
Let \(x = (c \vee p)*c\) and \(y = p * a\). Note right away, by
Lemma~\ref{lem:useful}\eqref{it:useful1}, that \(x \ge c\) and \(y \ge a\).

(1) To show that \(x \le y\), it suffices to show that \(p \wedge x \le a\).
Since \(p \wedge x \neq c\) (since otherwise \(c \le p\)), it suffices to
show \(p \wedge x \le c\) (since then \(a \prec c\)). Now
\[
c = (c \vee p) \wedge x = (c \wedge x) \vee (p \wedge x) = c \vee (p \wedge
x),
\]
where the first equality follows from
Lemma~\ref{lem:useful}\eqref{it:useful2}.

Hence \(p \wedge x \le c\) as desired.

(2) To show that \(y \le x\), we need to show that \((c \vee p)\wedge y \le
c\). But
\[
(c \vee p)\wedge y = (c \wedge y) \vee (p \wedge y) = (c \wedge y)\vee a =
  (c \vee a) \wedge (y \vee a) = c \wedge y \le c,
\]
where the second equality follows from
Lemma~\ref{lem:useful}\eqref{it:useful2}.
\end{proof}

\begin{lemma}\label{lem:**}
If \(a \le p\), \(a \prec c\), \(p \prec p \vee c\), and~\(c\) is a
join-irreducible, then \((p \vee c)*p = c*a\).
\end{lemma}

\begin{proof}
Pick \(a = a_0 \prec a_1 \prec \cdots \prec a_n = p\). Let \(b_i = a_i \vee
c\). Then we claim:
\begin{enumerate}
\item \(c = b_0 \prec b_1 \prec \cdots \prec b_n = p \vee c\),
\item \(a_i \prec b_i\),
\item \(b_i = a_i \vee b_{i-1}\) for \(i \ge 1\).
\end{enumerate}
By an easy induction argument, observe that \(\spec(b_i) = \spec(a_i) \cup
\{c\}\) for all \(i \le n\). This immediately gives~(1), (2) and~(3).

Now, applying the Lemma~\ref{lem:*} repeatedly, we obtain
\begin{align*}
(p \vee c)* p &= b_n*a_n\\
  &= (a_n \vee b_{n-1})*a_n\\
  &=b_{n-1}*a_{n-1} = \cdots  = c*a.
\end{align*}
\end{proof}

Now we can compute~\(q_{\sigma1}\) explicitly.

Now we have
\begin{align*}
q_{\sigma1} &= {(p_{\sigma0}^*)}^{L_\sigma}\\
  & = (p_{\sigma0}*p_\sigma)\wedge q_\sigma \\
  & = ((p_\sigma \vee c_\sigma)*p_\sigma) \wedge q_\sigma \\
  & = (c_\sigma*c_{\sigma,*})\wedge q_\sigma.
\end{align*}
Here, the last equality uses Lemma~\ref{lem:**} and the fact that
\(c_{\sigma,*} \le p_\sigma\).

\begin{definition}\label{def:special}
    For \(\sigma \in T_L\),~\(\sigma\) is \emph{special} if \(L_\sigma =
    [p_\sigma, 1]\) (i.e., if \(q_\sigma=1\)).
\end{definition}

Obviously,~\(\sigma\) is special iff~\(\sigma0\) is special (since being
special is just another way of saying that~\(\sigma\) is a string where every
bit is~\(0\)).

Note that if~\(\sigma\) is special, then \(q_{\sigma 1} =
c_\sigma*c_{\sigma,*}\), which is immediate from \(q_\sigma=1\) for
special~\(\sigma\), and which implies the following

\begin{lemma}
If \(c_\tau = c_\sigma\) where~\(\sigma\) is special, then \(q_{\tau 1}\le
q_{\sigma 1}\). \qed{}
\end{lemma}

We next show the following

\begin{lemma}\label{lem:spec*}
\(\Ji(\cL) = \{c_\sigma \mid \sigma \in T'_\cL \text{ is special}\}\).
\end{lemma}

\begin{proof}
For \(\sigma \in T_\cL\), if~\(\sigma\) is special, then~\(\sigma 0\) is
special, and inductively we have
\[
p_{\sigma 0} = p_\sigma \vee c_\sigma = c_\lambda \vee \cdots \vee c_\sigma.
\]

Let~\(\sigma\) be the longest special~\(\sigma\) in~\(T'_\cL\), then
\[
p_{\sigma 0} = 1 = c_\lambda \vee \cdots \vee c_\sigma.
\]
Therefore \(\Ji(\cL) = \spec(1) = \{c_\sigma \mid \sigma \in T'_\cL \text{ is
special}\}\).
\end{proof}

Now consider the requirement \(R_\tau(\Psi) : C_\tau \neq \Psi^{j(q_{\tau
1})}\). By Lemma~\ref{lem:spec*}, there exists a special~\(\sigma\) such that
\(c_\tau = c_\sigma\). Also note that \(q_{\tau 1} \le q_{\sigma 1}\);
therefore we only need to keep the requirements
\[
R_\sigma(\Psi) : C_\sigma \neq \Psi^{j(q_{\sigma1})}
\]
for the special nodes \(\sigma \in T'_\cL\). But if~\(\sigma\) is special,
then we have \(q_{\sigma 1} = c_\sigma*c_{\sigma,*}\), so we can rewrite our
\(R\)\nbd requirements as follows:
\[
R_c(\Psi):C \neq \Psi^{j(c*c_*)}
\]
where \(c \in \Ji(\cL)\). (Note here that we are switching from the
notation~\(R_{c_\sigma}\) to the notation~\(R_c\) for brevity.)

The final version of our requirements is thus the following:
\begin{align*}
G&: K = \Theta^{j(1)}\\
S(U)&: (\exists \eta\in [T_\cL])\; F_\eta(U)\\
R_c(\Psi)&:C \neq \Psi^{j(c*c_*)} \text{ (for each \(c \in \Ji(\cL)\))}
\end{align*}
where~\(U\) ranges over all d.c.e.\ sets and~\(\Psi\) ranges over all Turing
functionals, and where~\(F_\eta(U)\) was defined in Section~\ref{sub:S-req}.

\subsection{Three examples}\label{sub:example}

We give three examples (see Figure~\ref{fig:lattices}) and their requirements
in this section. In each lattice~\(\cL\), join-irreducible elements are
marked by~\(\bullet\) and also labeled according to~\(T_\cL\), the other
elements are marked by~\(\circ\) with the least element labeled by~\(0\). The
\(S\)\nbd requirements for each lattice are generated as in
Figure~\ref{fig:lattices-S}.

\begin{figure}[ht]\centering
    \begin{subfigure}[b]{0.25\textwidth}
        {\hspace{1cm}\includegraphics{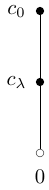}}
        \caption{3-element chain}\label{fig:lattice3}
    \end{subfigure}
    \begin{subfigure}[b]{0.3\textwidth}
        \includegraphics{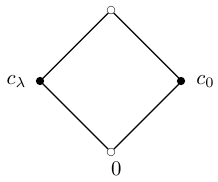}
        \caption{diamond lattice}\label{fig:lattice4}
    \end{subfigure}
    \begin{subfigure}[b]{0.4\textwidth}
        \includegraphics{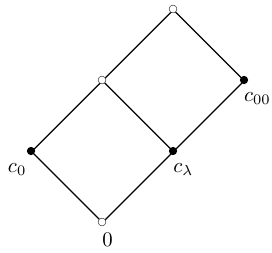}
        \caption{6-element lattice}\label{fig:lattice6}
    \end{subfigure}
    \caption{Pictures of lattices}\label{fig:lattices}
\end{figure}

We remark that in Figures~\ref{fig:lattice4} and~\ref{fig:lattice4-S}, we
write \(C_0 = \Gamma_1^{E \oplus U}\) instead of \(C_1 = \Gamma_1^{E \oplus
U}\) because \(c_1 = c_0\). In the same manner, in Figures~\ref{fig:lattice6}
and~\ref{fig:lattice6-S}, we replace~\(C_1\) and~\(C_{01}\) by~\(C_0\)
and~\(C_{00}\), respectively.

\begin{figure}[ht]\centering
    \begin{subfigure}[b]{0.9\textwidth}
        \centering
        \includegraphics{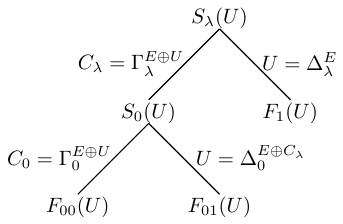}
        \caption{S-requirements for the 3-element
        chain}\label{fig:lattice3-S}
    \end{subfigure}

    \begin{subfigure}[b]{0.9\textwidth}
        \centering
        \includegraphics{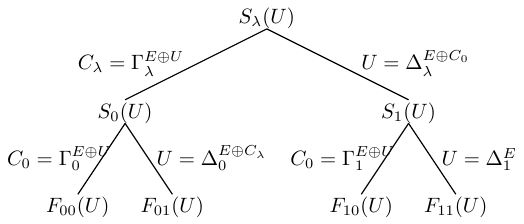}
        \caption{S-requirements for the diamond}\label{fig:lattice4-S}
    \end{subfigure}

    \begin{subfigure}[b]{0.9\textwidth}
        \centering
        \includegraphics[width=1\textwidth]{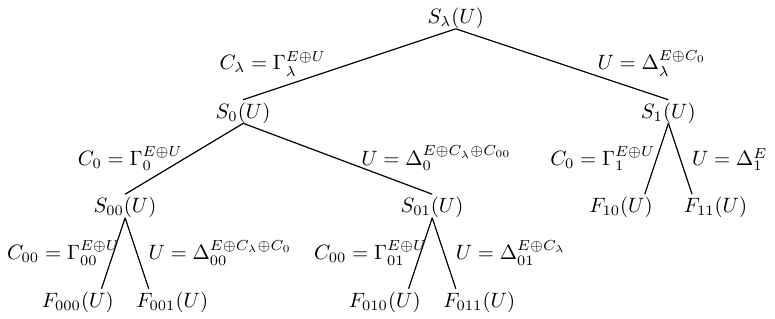}
        \caption{S-requirements for the 6-element
        lattice}\label{fig:lattice6-S}
    \end{subfigure}
    \caption{S-requirements for each lattice}\label{fig:lattices-S}
\end{figure}

Using Figure~\ref{fig:lattices-S}, we give the requirements for each lattice
in Figure~\ref{fig:lattices}.

\begin{enumerate}[label= (\alph*)]
    \item The requirements for the three-element chain
        (Figure~\ref{fig:lattice3}): \begin{align*}
        G&: K = \Theta^{E \oplus C_\lambda \oplus C_0},\\
        S(U)&: F_{00}(U) \lor F_{01}(U)\lor F_{1}(U),\\
        R_{c_\lambda}(\Psi)&: C_\lambda \neq \Psi^{E},\\
        R_{c_0}(\Psi)&: C_0 \neq \Psi^{E \oplus C_\lambda}.
    \end{align*}
    \item The requirements for the diamond (Figure~\ref{fig:lattice4}):
        \begin{align*}
        G&: K = \Theta^{E \oplus C_\lambda \oplus C_0},\\
        S(U)&: F_{00}(U) \lor F_{01}(U)\lor F_{10}(U)\lor F_{11}(U),\\
        R_{c_\lambda}(\Psi)&: C_\lambda \neq \Psi^{E \oplus C_0},\\
        R_{c_0}(\Psi)&: C_0 \neq \Psi^{E \oplus C_\lambda}.
    \end{align*}
    \item The requirements for the six-element lattice
        (Figure~\ref{fig:lattice6}): \begin{align*}
        G&: K = \Theta^{E \oplus C_\lambda \oplus C_0 \oplus C_{00}},\\
        S(U)&: F_{000}(U) \lor F_{001}(U)\lor \cdots \lor F_{11}(U),\\
        R_{c_\lambda}(\Psi)&: C_\lambda \neq \Psi^{E \oplus C_0},\\
        R_{c_0}(\Psi)&: C_0 \neq \Psi^{E \oplus C_\lambda \oplus C_{00}},\\
        R_{c_{00}}(\Psi)&: C_{00}\neq \Psi^{E \oplus C_\lambda \oplus C_0}.
    \end{align*}
\end{enumerate}

\subsection{A first discussion of the potential conflicts}%
\label{sub:conflicts}

Considering the three examples in Figure~\ref{fig:lattices} and their
\(S\)\nbd requirements in Figure~\ref{fig:lattices-S}, we observe the
following properties:
\begin{enumerate}
\item
If \(c_\sigma<c_\tau\), then~\(C_\sigma\) appears above~\(C_\tau\). For
example, in Figure~\ref{fig:lattice6-S} we have~\(C_\lambda\) appears
above~\(C_{00}\) but not necessarily above~\(C_{0}\). (See
Lemma~\ref{lem:spec1}.)
\item
The oracle of each~\(\Delta_\sigma\)
is the join of the set~\(E\) and the sets~\(C_\tau\) with $\tau0 \subseteq
\sigma$ or $\sigma1 \subseteq \tau$. (See
Lemmas~\ref{lem:tau0},~\ref{lem:sigma1}, and~\ref{lem:deltaoracle1}.)
\end{enumerate}
These properties are crucial, so we will formally prove them now.

\begin{lemma}\label{lem:spec0}
For \(\sigma \in T_L\), \(\spec(p_\sigma) = \{c_\tau \mid \tau0 \subseteq
\sigma \}\).
\end{lemma}

\begin{proof}
The lemma holds for~\(\lambda\) since \(p_\lambda=0\) and so
\(\spec(p_\lambda) = \varnothing\).

Suppose that it holds for~\(\sigma\); then it holds for~\(\sigma 0\)
and~\(\sigma 1\) because
\begin{align*}
\spec(p_{\sigma 0}) &=\spec(p_\sigma)\cup \{c_\sigma \} =
  \{c_\tau \mid \tau0 \subseteq \sigma 0 \}, \text{ and}\\
\spec(p_{\sigma 1}) &=\spec(p_\sigma) =
  \{c_\tau \mid \tau0 \subseteq \sigma \} = \{c_\tau \mid \tau0
   \subseteq \sigma 1 \}.  \qedhere
\end{align*}
\end{proof}

\begin{lemma}\label{lem:spec1}
If~\(c\) is join-irreducible in~\(\cL\) and \(c<c_\sigma\), then there is
some~\(\tau\) such that \(\tau0 \subseteq \sigma\) and \(c = c_\tau\).
\end{lemma}

\begin{proof}
Recall that \(\spec(p_{\sigma0}) = \spec(p_\sigma) \cup \{c_\sigma\}\). If
\(c<c_\sigma\), then \(c\in \spec(p_{\sigma 0})\setminus \{c_\sigma \} =
\spec(p_\sigma)\), so \(c \le p_\sigma\). Now apply Lemma~\ref{lem:spec0}.
\end{proof}

\begin{lemma}\label{lem:tau0}
Let \(\sigma, \tau \in T_\cL\). If \(\tau0 \subseteq \sigma\), then \(c_\tau
\le q_{\sigma 1}\).
\end{lemma}

\begin{proof}
By Lemma~\ref{lem:spec0}, we have \(c_\tau \le p_\sigma \le q_{\sigma 1}\).
\end{proof}

\begin{lemma}\label{lem:sigma1}
Let \(\sigma,\tau \in T_\cL\). If \(\sigma1 \subseteq \tau\), then \(c_\tau
\le q_{\sigma 1}\).
\end{lemma}

\begin{proof}
Since~\(L_\tau\) is a sublattice of~\(L_{\sigma1}\), we have \(c_\tau \le
p_{\tau 0} \le q_\tau \le q_{\sigma 1}\).
\end{proof}

\begin{lemma}\label{lem:deltaoracle1}
Let \(\sigma \in T_\cL\) and \(\eta = \sigma10 \cdots 0 \in [T_L]\). Then
\[
\spec(q_{\sigma 1}) = \{c_\tau \mid \tau 0 \subseteq \eta \} = \{c_\tau \mid
\tau0 \subseteq \sigma \} \cup \{c_\tau \mid \sigma1 \subseteq \tau\}.
\]
\end{lemma}

\begin{proof}
Observe that
\[
    \{c_\tau \mid \tau 0 \subseteq \eta \} \subseteq \{c_\tau \mid \tau0
    \subseteq \sigma \} \cup \{c_\tau \mid \sigma1 \subseteq \tau\}
    \subseteq \spec(q_{\sigma 1}),
\]
where the second inclusion follows from Lemmas~\ref{lem:tau0}
and~\ref{lem:sigma1}. So we only need to show \(\spec(q_{\sigma 1}) \subseteq
\{c_\tau \mid \tau 0 \subseteq \eta \}\). Let \(L_{\sigma 1} = [p_{\sigma 1},
q_{\sigma 1}]\). Then \(L_\eta = [q_{\sigma 1}, q_{\sigma 1}]\). Hence,
\(p_\eta = q_{\sigma 1}\), and so \(\spec(q_{\sigma 1}) = \spec(p_\eta) =
\{c_\tau \mid \tau 0 \subseteq \eta \}\) by Lemma~\ref{lem:spec0}.
\end{proof}

From Lemma~\ref{lem:deltaoracle1}, we have the following
\begin{lemma}\label{lem:deltaoracle2}
Let \(\sigma,\tau \in T_\cL\). If \(\sigma 1 \subseteq \tau\), then
\(\spec(q_{\tau 1})\subseteq \spec(q_{\sigma 1})\). \qed
\end{lemma}

Next, we consider \(R\)\nbd requirements and analyze what happens when an
\(R\)\nbd requirement tries to diagonalize. In order to give some intuition,
we will have to talk about \emph{killing} or \emph{correcting} a functional,
the \emph{use} of a computation, and a \emph{witness} of a requirement in the
usual sense, but the formal definitions of these are postponed until
Section~\ref{sec:1U}.

We will first illustrate these using the example of the six-element lattice
in Figure~\ref{fig:lattice6}, the \(S\)\nbd requirements shown in
Figure~\ref{fig:lattice6-S}, and a particular \(R\)\nbd requirement.

Suppose that~\(R_{c_{00}}\) has a witness~\(x\) and a computation with
use~\(\psi(x)\).

Case~0: \(F_{000}(U) = \{\Gamma_\lambda, \Gamma_0, \Gamma_{00}\}\).
Then~\(R_{c_{00}}\) has no conflicts with~\(\Gamma_\lambda\) or~\(\Gamma_0\)
since~\(R_{c_{00}}\) wants to preserve~\(C_\lambda\) and~\(C_0\),
so~\(R_{c_{00}}\) will not trigger any \(\Gamma_\lambda\)- or
\(\Gamma_0\)\nbd correction. But~\(R_{c_{00}}\) has a conflict
with~\(\Gamma_{00}\) since when~\(R_{c_{00}}\) enumerates~\(x\)
into~\(C_{00}\), then \(\Gamma_{00}^{E \oplus U}(x)\), intending to
compute~\(C_{00}\), may require correction by a small number entering~\(E\),
possibly injuring the computation of~\(R_{c_{00}}\).

Case~1: \(F_{001}(U) = \{\Gamma_\lambda, \Gamma_0, \Delta_{00}\}\). Then, as
in Case~0,~\(R_{c_{00}}\) has no conflicts with~\(\Gamma_\lambda\)
or~\(\Gamma_0\). But~\(R_{c_{00}}\) has a conflict with~\(\Delta_{00}\) since
\(q_{001}\le c_{00}*c_{00,*}\) (i.e., the sets appearing in the oracle
of~\(\Delta_{00}\) appear also in the oracle of~\(R_{c_{00}}\)), and any
correction made by~\(\Delta_{00}\) via a number $\le \psi(x)$ will
injure~\(R_{c_{00}}\).

Case~2: \(F_{010}(U) = \{\Gamma_\lambda, \Delta_0, \Gamma_{01}\}\). Then, as
in Case~0,~\(R_{c_{00}}\) has no conflict with~\(\Gamma_\lambda\).
Also,~\(R_{c_{00}}\) has no conflict with~\(\Delta_0\) since~\(C_{00}\) can
be used to correct~\(\Delta_0\), and~\(R_{c_{00}}\) itself wants to
change~\(C_{00}\) as well. Finally,~\(R_{c_{00}}\) has a conflict
with~\(\Gamma_{01}\) because \(c_{01} = c_{00}\).

Case~3: \(F_{011}(U) = \{\Gamma_\lambda, \Delta_0,\Delta_{01} \}\). Then, as
in Cases~0 and~2, respectively,~\(R_{c_{00}}\) has no conflict
with~\(\Gamma_\lambda\) or~\(\Delta_0\). Analogously to Case~1
(with~\(\Delta_{00}\)),~\(R_{c_{00}}\) has a conflict with~\(\Delta_{01}\)
since \(c_{01} = c_{00}\) and \(q_{011} \le c_{01}*c_{01,*} =
c_{00}*c_{00,*}\), where the~\(\le\) follows from the calculation above
Definition~\ref{def:special}.

Case~4: \(F_{10}(U) = \{\Delta_\lambda, \Gamma_1 \}\). Then~\(R_{c_{00}}\)
has a conflict with~\(\Delta_\lambda\) since the oracle of~\(\Delta_\lambda\)
also appears in the oracle of~\(R_{c_{00}}\). (To be a little more general,
we can show that \(c_\lambda<c_{00}\) implies that \(q_1\le c_{00}*c_{00,*}\)
by Lemma~\ref{lem:useful}\eqref{it:useful-mon} with \(a=c_\lambda\) and the
calculation above Definition~\ref{def:special}.) But~\(R_{c_{00}}\) has no
conflict with~\(\Gamma_1\) since~\(R_{c_{00}}\) wants to preserve \(C_1 =
C_0\), so~\(R_{c_{00}}\) will not trigger any \(\Gamma_1\)\nbd correction.

Case~5: \(F_{11}(U) = \{\Delta_\lambda, \Delta_1\}\). As in
Case~4,~\(R_{c_{00}}\) has a conflict with~\(\Delta_\lambda\); note
that~\(R_{c_{00}}\) also has a conflict with~\(\Delta_1\) since~\(\Delta_1\)
is lower than~\(\Delta_\lambda\) and by Lemma~\ref{lem:deltaoracle2}. But in
this case,~\(R_{c_{00}}\) only takes care of~\(\Delta_\lambda\) since
if~\(R_{c_{00}}\) cannot ensure the correctness of~\(\Delta_\lambda\),
then~\(\Delta_1\) won't matter.

In summary,
\begin{itemize}
    \item \(R_{c_{00}}\) has a conflict with~\(\Gamma_\sigma\) iff \(c_\sigma
        = c_{00}\) iff \(\sigma = 00\) or \(\sigma = 01\).
    \item In Case~2, since~\(R_{c_{00}}\) has a conflict with~\(\Gamma_{01}
        \in F_{010}(U)\), it has no conflict with~\(\Delta_0\); this is
        because \(c_{00}=c_{01} \le q_{01}\) and~\(\Delta_0\) has
        oracle~\(j(q_{01})\).
    \item In Cases~1,~3,~4, and~5,~\(R_{c_{00}}\) has no conflict with
        any~\(\Gamma\), but~\(R_{c_{00}}\) has conflicts
        with~\(\Delta_{00}\), \(\Delta_{01}\), and both~\(\Delta_\lambda\)
        and~\(\Delta_1\), respectively. Note here also that \(c_{00},
        c_{01}, c_\lambda \le c_{00}\).
\end{itemize}

We now want to show how to generalize these properties to arbitrary finite
distributive lattices. We first formulate the first and third property as a
definition.

\begin{definition}[conflicts]\label{def:conflicts}
\hspace{1em}
\begin{enumerate}
    \item \(R_c\) has \emph{a conflict} with~\(\Gamma_\sigma\) iff \(c_\sigma
        = c\).
    \item \(R_c\) has \emph{a conflict} with~\(\Delta_\tau\) iff \(q_{\tau 1}
        \le c*c_*\).
\end{enumerate}
\end{definition}

By Lemma~\ref{lem:spec1}, for any join-irreducible element~\(c\) in~\(\cL\),
if there is some \(\Gamma_\tau \in F_\eta(U)\) such that \(c \le c_\tau\)
(and so~\(R_c\)'s diagonalization might be destroyed by \(\Gamma_\tau\)\nbd
correction), then there is some \(\Gamma_\sigma \in F_\eta(U)\) with
\(c_\sigma = c\). Therefore, if for any \(\Gamma \in F_\eta(U)\) computing a
set~\(C_\tau\) with \(c \le c_\tau\), then~\(R_c\) has a conflict with some
\(\Gamma_\sigma \in F_\eta(U)\) with \(\sigma \subseteq \tau\); so it
suffices to consider this~\(\Gamma_\sigma\) in the definition of conflict.

Given~\(F_\eta(U)\), suppose that~\(R_c\) has a conflict with \(\Delta_\sigma
\in F_\eta(U)\); then for all \(\Delta_\tau \in F_\eta(U)\) which are lower
than~\(\Delta_\sigma\), \(R_c\) also has a conflict with~\(\Delta_\tau\)
because \(\spec(q_{\tau 1}) \subseteq \spec(q_{\sigma 1})\) by
Lemma~\ref{lem:deltaoracle2}. Therefore, if~\(R_c\) has a conflict with any
\(\Delta \in F_\eta(U)\), we need to only consider the highest
such~\(\Delta_\sigma\).

The following lemma is crucial to our argument.

\begin{lemma}\label{lem:conflicts}
For \(c \in \Ji(\cL)\) and \(\eta \in [T_L]\), we have:
\begin{enumerate}
    \item\label{it:R-plus} If~\(R_c\) has a conflict with some (necessarily
        unique) \(\Gamma_\tau \in F_\eta(U)\), then for all \(\Delta_\sigma
        \in F_\eta(U)\), \(c \le q_{\sigma 1}\). (Therefore,
        such~\(\Delta_\sigma\) can be corrected via the set~\(C\).)
    \item\label{it:R-minus}
         Otherwise, there is some \(\Delta_\sigma \in F_\eta(U)\) with
         which~\(R_c\) has a conflict. For the highest \(\Delta_\sigma \in
         F_\eta(U)\) with which~\(R_c\) has a conflict, \(c_\sigma \le c\).
         For all \(\Delta_\tau \in F_\eta(U)\) with which~\(R_c\) has no
         conflict, we have \(c \le q_{\tau 1}\).
\end{enumerate}
\end{lemma}

\begin{proof}
    \begin{enumerate}
        \item Recall the definition of~\(F_\eta(U)\), we have
            \(\tau0 \subseteq \eta\) and \(\sigma1 \subseteq \eta\).
            Therefore, we either have \(\sigma1 \subseteq \tau 0\) or
            \(\tau0 \subseteq \sigma\). In either case, by
            Lemma~\ref{lem:deltaoracle1}, we have \(c = c_\tau \le
            q_{\sigma 1}\).
        \item We will proceed by induction on \(\sigma \subseteq \eta\) and
            show:
            \begin{enumerate}
                \item If \(\sigma 1 \subseteq \eta\) and \(q_{\sigma 1}\le
                    c*c_*\), then \(c_\sigma \le c\). We stop the
                    induction.
                \item If \(\sigma 1 \subseteq \eta\) and \(q_{\sigma 1}\nleq
                    c*c_*\), then \(c \le q_{\sigma 1}\) and continue
                    with~\(\sigma 1\).
                \item If \(\sigma 0 \subseteq \eta\), then \(c_{\sigma }\ne
                    c\), \(c \le q_{\sigma 0}\) and continue with~\(\sigma
                    0\).
             \end{enumerate}

        Case~(a). Suppose towards a contradiction that \(c_\sigma \nleq c\),
        so \(c_\sigma \land c \le c_{\sigma,*}\). Therefore \(c \le
        c_\sigma*c_{\sigma,*}\). Because of the induction hypothesis in
        Cases~(b) and~(c), we have \(c \le q_\sigma\). Together we have
        \[
            c \le (c_\sigma*c_{\sigma,*})\wedge
            q_\sigma = q_{\sigma 1} \le c*c_*,
        \]
        where the second equality uses the calculation above
        Definition~\ref{def:special}, and the last inequality is the
        assumption of Case~(a). Therefore, we obtain \(c \le c*c_*\),
        contradicting Lemma~\ref{lem:useful}\eqref{it:useful-mon}.

        Case~(b). Since \(q_{\sigma 1}\nleq c*c_*\), we have \(c \wedge
        q_{\sigma 1} \nleq c_*\). But \(c \wedge q_{\sigma 1} \le c\), so
        the only possibility is that \(c \wedge q_{\sigma 1} = c\), and
        hence \(c \le q_{\sigma 1}\).

        Case~(c). \(c \le q_{\sigma} = q_{\sigma 0}\) where the first
        inequality is the induction hypothesis in Cases~(b)
        and~(c) and the last equality is by
        the definition of \(q_{\sigma 0}\). \(c_\sigma \neq c\) is our
        assumption for~(2) that~\(R_c\) has no conflicts with any
        \(\Gamma\in F_\eta(U)\).

        Now suppose that we never reach Case~(a), and so we have \(c \le
        q_\eta = p_\eta\). By Lemma~\ref{lem:spec0}, we have \(c = c_\tau\)
        for some \(\tau 0 \subseteq \eta\), so~\(R_c\) would have a
        conflict with \(\Gamma_\tau \in F_\eta(U)\), a contradiction. Thus
        the induction will end with Case~(a), and thus~(2) is proved.
        \qedhere
    \end{enumerate}
\end{proof}

Suppose that~\(R_c\) has a witness~\(x\), a computation with use \(\psi(x)\),
and a conflict with \(\Gamma \in F_\eta(U)\). The strategy that~\(R_c\) takes
would be the following. It will actively kill \(\Gamma\) by enumerating
\(\gamma(x)\) (possibly \(\gamma(x) \le \psi(x)\)) into~\(E\) and request
that~\(\gamma(x)\) be redefined as fresh (so the new \(\gamma(x) > {}\) the
current~\(\psi(x)\) next time). What~\(R_c\) hopes for is a point \(z \le
\gamma(x)\) such that \(U(z)\) changes, and in this case, we are allowed to
take the old~\(\gamma(x)\) out so that the computation of~\(R_c\) is restored
and~\(\Gamma\) can be corrected using the latest \(\gamma(x)\), not injuring
\(\psi(x)\). Note that after~\(R_c\) kills~\(\Gamma_\sigma\), we will attempt
to make \(F_{\sigma100 \cdots 0}(U)\) satisfied.

Suppose that~\(R_c\) finds no conflict with any \(\Gamma \in F_\eta(U)\).
Then let~\(\Delta_\sigma\) be the highest one with which~\(R_c\) has a
conflict.~\(R_c\) would like to know who is responsible for building
this~\(\Delta_\sigma\). By the preceding paragraph,~\(R_{c_\sigma}\) must be
responsible for this, where we also have \(c_\sigma \le c\).

This motivates the following definition.

\begin{definition}\label{def:Rc}
Recall that for each \(c \in \Ji(\cL)\), we have \(R_c\)\nbd requirements.
For each \(R_c\)\nbd requirement, based on the conflicts explained above, we
define the nondecreasing map \(R_c: [T_\cL] \to [T_\cL]\)
(Definition~\ref{def:T_L}) by
\[
R_c(\eta) =
\begin{cases}
\sigma10 \ldots 0 & \text{if~\(R_c\) has a conflict with
                         \(\Gamma_\sigma \in F_\eta(U)\)},\\
\eta              & \text{otherwise.}
\end{cases}
\]
\end{definition}

Examples are given in Table~\ref{tab:R}. In this table, we list \([T_L]\)
lexicographically as \(\eta_0<\eta_1<\cdots<\eta_{\abs{\cL}}\). We have, for
example, \((R_{c_0}\circ R_{c_\lambda}\circ R_{c_{00}}\circ R_{c_0}\circ
R_{c_{00}})(\eta_0) = \eta_5\) for the six-element lattice.

\begin{table}[ht]\centering
    \begin{tabular}{l l*{6}{c}}
        \toprule
         & & \(\eta_0\) & \(\eta_1\) & \(\eta_2\) &\(\eta_3\) &\(\eta_4\)
           &\(\eta_5\) \\
        \midrule
        \multirow{2}{*}{3-element chain}
         &\(R_{c_\lambda}\) & \(\eta_2\) & \(\eta_2\) & \(\eta_2\) &  &  & \\
         &\(R_{c_0}\) & \(\eta_1\) & \(\eta_1\) & \(\eta_2\) &  &  &  \\
        \midrule
        \multirow{2}{*}{diamond}
         &\(R_{c_\lambda}\) & \(\eta_2\) & \(\eta_2\) & \(\eta_2\) &\(\eta_3\)
           &  &  \\
         &\(R_{c_0}\) & \(\eta_1\) & \(\eta_1\) & \(\eta_3\) &\(\eta_3\)
           &  &  \\
        \midrule
        \multirow{3}{*}{6-element lattice}
         &\(R_{c_\lambda}\) & \(\eta_4\) & \(\eta_4\)
           & \(\eta_4\) &\(\eta_4\) &\(\eta_4\) &\(\eta_5\) \\
         &\(R_{c_0}\) & \(\eta_2\) & \(\eta_2\) & \(\eta_2\) &\(\eta_3\)
           &\(\eta_5\) &\(\eta_5\) \\
         &\(R_{c_{00}}\) & \(\eta_1\) & \(\eta_1\) & \(\eta_3\) &\(\eta_3\)
           &\(\eta_4\) &\(\eta_5\) \\
        \bottomrule
      \end{tabular}
      \caption{The \(R\)\nbd maps for our three examples}\label{tab:R}
    \end{table}

This completes the discussion of the lattice-theoretic aspects of our
construction.

\section{\texorpdfstring{A Single \(U\)-set}{A Single U-set}}\label{sec:1U}

\subsection{Introduction}
Let~\(\cL\) be the finite distributive lattice with least element~\(0\) and
greatest element~\(1\). Recall that our requirements take the following form:
\begin{align*}
 G&: K=\Theta^{j(1)}\\
 S(U)&: \exists \eta\in [T_\cL], F_\eta(U) \\
 R_c(\Psi_e)&: C\neq \Psi_e^{j(c*c_*)}
\end{align*}
where~\(U\) ranges over all d.c.e.\ sets,~\(\Psi_e\) is the \(e\)\nbd th
Turing functional and~\(e\) ranges over~\(\omega\),~\(c\) ranges over
\(\Ji(\cL)\), and~\(G\) is a single global requirement.

The goal of this section is to present the conflicts of the \(G\)\nbd
requirement, one single \(S(U)\)\nbd requirement (for an arbitrary fixed
d.c.e.\ set~\(U\)), and all \(R\)\nbd requirements. We note one unusual
feature of our construction: Unlike in other constructions, we will have to
try to meet the same requirement repeatedly without any apparent gain until
we succeed, simply to have a sufficient number of strategies in the right
arrangement.

We will denote the number of potential tries by~\(m\) for now.
Note that the optimal value of~\(m\) depends on the lattice~\(\cL\) only. A
careful analysis into the lattice structure of~\(\cL\) may give us the
optimal value. \(m=1\) is optimal for Boolean algebras, \(m=2\) is optimal
for the \(3\)\nbd element chain (Figure~\ref{fig:lattice4}), but setting
\(m=\abs{\cL}+1\) is always sufficient for our construction.

This section will introduce the computability-theoretic aspects of our
construction (using a general finite distributive lattice instead of
examples) in the simplest combinatorial setting, preparing us for the more
challenging full setup with all sets~\(U_j\).

\subsection{The priority tree}\label{sec:1Utree}

For each \(c \in \Ji(\cL)\), recall the map \(R_c: [T_\cL] \to [T_\cL]\) from
Definition~\ref{def:Rc}; for now, the only property of~\(R_c\) that we need
is that it is a nondecreasing map. We also fix a computable list \(\{R^e\}_{e
\in \omega}\) of all \(R\)\nbd requirements \(\{R_c(\Psi_j)\}_{j\in\omega,
c\in\Ji(\cL)}\).

The functionals in~\(F_\xi(U)\) for each \(\xi\in [T_\cL]\) will be referred
to as \emph{\(U\)\nbd functionals}. A node~\(\alpha\) on the tree will be
called an \emph{\(R\)\nbd node} if it is assigned to an \(R\)\nbd
requirement; and an \emph{\(S\)\nbd node} if it is assigned to an \(S\)\nbd
requirement with pair \((a,\xi) \in \{0,\ldots , m-1\} \times [T_\cL]\). We
view an \(S\)\nbd node as the \(a\)\nbd th copy of one of the \(S_U\)\nbd
strategies. We order \(\{0,\ldots, m-1\} \times [T_\cL]\) lexicographically.
The intuition behind the notation \((a,\xi)\) is that we work our way through
all the necessary \(S(U)\)\nbd strategies until there is no longer a
\(\Gamma\)\nbd functional to attack (with \(a=0\), i.e., the first time).
Then we start this whole process again with \(a=1\), with \(a=2\), etc.,
until we reach \(a=m-1\). At this point, we will be sure that we have a
sufficient number of \(S_U\)\nbd strategies so that we can deal with any
possible \(U\)\nbd changes back and forth, as we will have to prove in the
end.

An \(S\)\nbd node assigned to \((a,\xi)\) has only one outcome,~0. An
\(R\)\nbd node~\(\alpha\) has a~\(w\)\nbd outcome, a~\(d\)\nbd outcome,
possibly one of two types of \(U\)\nbd outcomes, and finally a~\(\ctr\)\nbd
outcome. The intuition (which will become clear later) for the outcomes is
the following: Initially,~\(\alpha\) keeps visiting its \(w\)\nbd outcome
while~\(\alpha\) is looking for some computation to converge. Next, having
found a computation,~\(\alpha\) is now ready to visit the next outcome, the
\(U\)\nbd outcome, if any, meaning that~\(\alpha\) is dealing with
\(S_U\)\nbd functionals. In case that no \(U\)\nbd outcome is available
for~\(\alpha\), the next outcome that~\(\alpha\) visits is the \(\ctr\)\nbd
outcome, meaning that~\(\alpha\) has gathered enough information and becomes
a controller. If~\(\alpha\) is successful in its diagonalization, then the
\(d\)\nbd outcome is said to be active, and~\(\alpha\) visits the \(d\)\nbd
outcome. If the \(d\)\nbd outcome is inactive, then~\(\alpha\) will visit its
outcomes in sequence: It always starts with the \(w\)\nbd outcome, then
possibly a \(U\)\nbd outcome (if any), and then the \(\ctr\)\nbd outcome, at
which point the \(\ctr\)\nbd outcome becomes active.

\begin{definition}\label{def:1Utree}
We define the priority tree~\(\cT\) by recursion: We assign the root
node~\(\lambda\) to \((0,\iota)\) (\(\iota\) is the string \(00\cdots 0\) of
proper length) and call it an \(S\)\nbd node.

Suppose that~\(\alpha\) is an \(S\)\nbd node assigned to \((a,\xi)\). We
determine the least~\(e\) such that there is no \(R^e\)\nbd node \(\beta
\subset \alpha\) with \(\beta^\frown w \subseteq \alpha\) or \(\beta^\frown d
\subseteq \alpha\), and we assign \(\alpha^\frown 0\) to~\(R_e\).

Suppose that~\(\alpha\) is an \(R_c\)\nbd node for some \(c\in \Ji(\cL)\).
Let \(\beta \subset \alpha\) be the longest \(S\)\nbd node, which is assigned
to \((a,\xi)\), say. We add outcomes to~\(\alpha\) in the following sequence:
\begin{enumerate}
\item We add a \(U\)\nbd outcome as its \emph{first outcome}.
\begin{enumerate}
\item If \(\xi<R_c(\xi)\), then we say the \(U\)\nbd outcome is a
\emph{Type~I} outcome, and we assign \(\alpha^\frown U\) to \((a,
R_c(\xi))\). The \emph{next} outcome is placed just to the left of this
\(U\)\nbd outcome.
\item If \(\xi=R_c(\xi)\), then we say the \(U\)\nbd outcome is a
\emph{Type~II} outcome. If \(a<m-1\), then we say that this \(U\)\nbd
outcome is \emph{GREEN}, and we assign \(\alpha^\frown U\) to
\((a+1,\iota)\). If \(a = m-1\), then we say that this \(U\)\nbd outcome
is~\(\emph{RED}\), and we do not assign \(\alpha^\frown U\) to any
requirement (so it is a terminal  node on the tree). In either case, the
\emph{next} outcome is placed just to the right of this \(U\)\nbd
outcome. (We caution the reader here that the definition of GREEN and RED
is \emph{not} static. A GREEN outcome can become RED and vice versa. The
information below the RED outcome will be frozen for a while, waiting for
another piece of information to wake up, at which time the RED turns
GREEN again. We will never visit a RED outcome directly.)
\end{enumerate}
\item We add a \(\ctr\)\nbd outcome as the \emph{second outcome}. We do not
assign \(\alpha^\frown \ctr\) to any requirement (so it is a terminal node on
the tree).
\item Finally, we add~\(w\)- and \(d\)\nbd outcomes to the right of all
existing outcomes and assign both \(\alpha^\frown w\) and \(\alpha^\frown
d\) to \((a,\xi)\). (Note that this will introduce ``dummy'' \(S_U\)\nbd
strategies that are not strictly needed; however, this will no longer be
possible when we have infinitely many sets~\(U_j\)).
\end{enumerate}

The order of the outcomes is \(\ctr<U<w<d\) if the \(U\)\nbd outcome is
Type~I, and \(U<\ctr<w<d\) if the \(U\)\nbd outcome is Type~II\@. For
\(\alpha,\beta \in \cT\),~\(\alpha\) has \emph{higher priority}
than~\(\beta\) (denoted by \(\alpha <_P \beta\)) if~\(\alpha\) is to the left
of \(\beta\) or~\(\alpha\) is a proper initial segment of~\(\beta\).

This finishes the definition of the priority tree \(\cT\).
\end{definition}

We need some additional auxiliary notions:
\begin{itemize}
\item
The~\(\ctr\)- and \(d\)\nbd outcomes can be \emph{active} or
\emph{inactive}.
\item
\(\alpha\) is a \emph{controller} iff \(\alpha^\frown \ctr\) is active.
\item
Only the \(Type~II\) outcome can be GREEN or RED\@. A GREEN outcome can
become RED and vice versa.
\item
\(\alpha^-\) is the longest node such that \(\alpha^-\subsetneq \alpha\).
\item
For an \(S\)\nbd node \(\alpha\) assigned to \((a,\xi)\), we set
\(\seq(\alpha)= (a,\xi)\), \(\seq_0(\alpha)=a\), and
\(\seq_1(\alpha)=\xi\). (Note that if~\(\alpha\) is an \(R\)\nbd node,
then~\(\alpha^-\) will always be an \(S\)\nbd node in our priority tree.)
\item
For an \(R\)\nbd node \(\alpha\) with \(\seq_0(\alpha^-)=a\), we
say~\(\alpha\) is \emph{dealing with the~\(a\)\nbd th copy of the \(U\)\nbd
functionals}. (The \(b\)\nbd th copy of the \(U\)\nbd functionals is
irrelevant to~\(\alpha\) if \(b \neq a\).)
\item
For an \(R\)\nbd node, the \emph{next outcome} of a particular outcome is
well-defined (except for~\(\ctr\)-, \(w\)- and \(d\)\nbd outcomes): It is
the next outcome added to this \(R\)\nbd node after the particular outcome
is added.
\item
For each \(R\)\nbd node \(\beta\), we have a threshold point \(\tp(\beta)\)
associated to~\(\beta\), and the diagonalizing witness associated to the
\(w\)\nbd outcome of \(\beta\), denoted by \(\dw(\beta)\).
\end{itemize}

An example of the priority tree for the three-element lattice (see
Figures~\ref{fig:lattice3} and~\ref{fig:lattice3-S}) can be found in
Figure~\ref{fig:1Utree} with \(m=2\). \([T_\cL]\) is listed lexicographically
as \(\eta_0<\eta_1<\eta_2\). In Figure~\ref{fig:1Utree}, the \(S\)\nbd node
assigned to \((a,\eta_j)\) is denoted by~\(aj\) for short.

Some of the~\(w\)- and \(d\)\nbd outcomes are hidden, and so are the labels
of \(U\)\nbd outcomes. A Type~II \(U\)\nbd outcome is denoted by a thick
line. The \(U\)\nbd outcomes of~\(R^3\), \(R^4\), and~\(R^7\) are RED\@. A
terminal node is denoted by a~\(\bullet\). A \(\ctr\)\nbd outcome is denoted
by a slim line and a~\(\bullet\).

\begin{figure}[ht]\centering
    \includegraphics{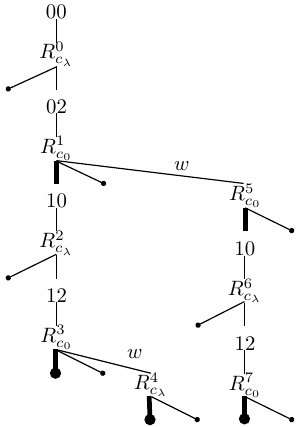}
    \caption{The priority tree for 3-element chain}\label{fig:1Utree}
\end{figure}

\begin{remark}
\begin{enumerate}
\item One may not want to add \(\ctr\)\nbd outcomes in the first place as we
do not even assign \(\alpha^\frown \ctr\) to any requirements. However,
adding \(\ctr\)\nbd outcomes reveals its priority clearly. Suppose that
\(\alpha^\frown \ctr\) becomes active. If \(\ctr<U\), then nodes below
\(\alpha^\frown U\) will be initialized. If \(U<\ctr\), nodes below
\(\alpha^\frown U\) \emph{must not} be initialized. If~\(\alpha\)
and~\(\beta\) are two controllers, then~\(\alpha\) has higher priority
than~\(\beta\) if \(\alpha^\frown \ctr <_P \beta^\frown \ctr\).
\item The reason why we add a RED \(U\)\nbd outcome to~\(\alpha\) even
though we do not even assign \(\alpha^\frown U\) to any requirement is to
make the construction more uniform. For example, in
Figure~\ref{fig:1Utree}, when the \(U\)\nbd outcome of~\(R^1\) becomes RED,
the strategy of~\(R^1\) will be the same as~\(R^3\).
\item For those who are familiar with the proof of D.C.E.\ Nondensity
Theorem~\cite{CHLLS91}, our notion of controller is dynamic and generalizes
theirs.
\item The letter~\(U\) in \emph{\(U\)\nbd outcome} really has the same
meaning as the letter~\(U\) in the \(S_U\)\nbd node just above~\(\alpha\).
In Section~\ref{sec:2U} when we have indexed sets~\(U_j\), we will have
\(U_j\)\nbd outcomes, or simply \(j\)\nbd outcomes.
\end{enumerate}
\end{remark}

The following lemma states that all requirements are represented by some node
along any infinite path through~\(\cT\).

\begin{lemma}\label{lem:1U full requirement}
Let~\(p\) be an infinite path through~\(\cT\). Then
\begin{enumerate}
\item there is an \(S\)\nbd node~\(\alpha\) such that for each
\(S\)\nbd node~\(\beta\) with \(\alpha \subset \beta \subset p\), we have
\(\seq(\alpha) = \seq(\beta)\), and
\item for each~\(e\), there is an \(R^e\)\nbd node~\(\alpha\) such that
either \(\alpha^\frown d \subset p\) or \(\alpha^\frown w \subset p\).
\end{enumerate}
\end{lemma}

\begin{proof}
(1) Consider \(\seq(\alpha)\) for all \(S\)\nbd nodes \(\alpha \subset p\),
which is nondecreasing for \(\alpha \subset p\); now note that \(m \times
[T_\cL]\) is finite.

(2) We proceed by induction on~\(e\). Suppose that~\(\alpha\) is assigned
to~\(R^e\). Suppose that for all \(R^e\)\nbd nodes~\(\beta\) with \(\alpha
\subseteq \beta \subset p\), we have \(\beta^\frown U \subset p\), then the
value of \(\seq(\gamma)\) for all \(S\)\nbd nodes~\(\gamma\) with \(\alpha
\subset \gamma \subset p\) would be strictly increasing. But \(m \times
[T_\cL]\) is finite.
\end{proof}

\subsubsection{\texorpdfstring{Functionals manipulated at \(S\)-nodes and at
\(R\)-nodes}{Functionals manipulated at S-node and R-node}}%
\label{sec:functionals at S and R}

An \(S\)\nbd node~\(\beta\) assigned to \((b,\xi)\) intends to ensure that
the \(b\)\nbd th copies of all functionals in~\(F_\xi(U)\) are correct and
total. However, not all functionals in~\(F_\xi(U)\) are maintained
at~\(\beta\). \(\beta\) only builds and maintains each \(\Gamma\)\nbd
functional in \(\mt(\beta,U)\), defined as follows.

\begin{definition}\label{def:maintain1U}
Let \(\beta\) be an \(S\)\nbd node with \(\seq(\beta)=(b,\xi)\) and \(\alpha
= (\beta^-)^-\), if it exists, with \(\seq(\alpha)=(a,\eta)\).
\begin{itemize}
\item
If \(\alpha\) does not exist, then \(\mt(\beta,U) = F_\xi(U)\). (In fact,
\(\xi=00\cdots 0\) and \(b=0\).)
\item
If \(a=b\) and \(\eta=\xi\), then \(\mt(\beta,U) =\varnothing\).
\item
If \(a=b\) and \(\eta<\xi\), then for some \(\sigma\) we have \(\sigma
0\subseteq \eta\) and \(\sigma 100\cdots 0 = \xi\), and we let
\(\mt(\beta,U)= \{\Gamma_\tau\mid \sigma 1\subseteq \tau\subseteq \xi\}\).
\item
If \(a+1=b\), then \(\mt(\beta,U) = F_\xi(U)\). (In fact, \(\xi=00\cdots
0\).)
\end{itemize}
\end{definition}

An \(R\)\nbd node \(\beta\) with \(\seq(\beta^-)=(a,\eta)\) ensures that some
of the functionals in \(F_\eta(U)\) are properly killed (see
Section~\ref{sec:functionals} for details), and that possibly one
\(\Delta\)\nbd functional is correct and total, depending on which outcome
\(\beta\) is visiting.

\begin{definition}\label{def:kill}
Let \(\beta\) be an \(R\)\nbd node with \(\seq(\beta^-)=(a,\eta)\), and
suppose, if \(\beta^\frown U\) is not a terminal node, that
\(\seq(\beta^\frown U)=(b,\xi)\).
\begin{itemize}
\item
If \(a=b\), then \(\kl(\beta,U)=F_\eta(U)\setminus F_\xi(U)\) and
\(\mt(\beta,U)=\{\Delta\}\) for the unique \(\Delta\in F_\xi(U)\setminus
F_\eta(U)\).
\item
If \(a<b\), then \(\kl(\beta,U)=F_\eta(U)\) and \(\mt(\beta,U)
=\varnothing\).
\end{itemize}
If \(\beta^\frown U\) is a terminal node, then
\(\kl(\beta,U)=\mt(\beta,U)=\varnothing\). We also define
\(\kl(\beta,d)=\mt(\beta,d)=\kl(\beta,w)=\mt(\beta,w)=\varnothing\).
\end{definition}

In summary, an \(R\)\nbd node \(\beta\) visiting its \(U\)\nbd outcome has to
kill each functional in \(\kl(\beta,U)\), and build and maintain the
\(\Delta\)\nbd functional, if any, in \(\mt(\beta,U)\).

\subsubsection{\texorpdfstring{\(\beta^*\) and \(\beta^\sharp\)}{beta star %
and beta sharp}}\label{sec:star and sharp}

Now we are ready to formulate the crucial definition for our construction.
Consider an \(R_c\)\nbd node~\(\beta\). Suppose that \(\seq(\beta^-) =
(a,\eta)\) and that \(\eta = R_c(\eta)\). By
Lemma~\ref{lem:conflicts}\eqref{it:R-minus}, we know that there is a highest
\(\Delta_\sigma \in F_\eta(U)\) with which~\(R_c\) has a conflict.
Let~\(\alpha\) be the \(R_{c_\sigma}\)\nbd node such that \(\alpha^\frown U
\subseteq \beta^-\) and~\(\alpha\) builds (the~\(a\)\nbd th copy of)
this~\(\Delta_\sigma\) (along this \(U\)\nbd outcome). We then define
\[
\beta^* = \alpha.
\]

Two properties will be used throughout the paper: The first one is that
\(c_\sigma \le c\); The second one is that the set~\(C\) does not appear in
the oracle of \(R_{c_{\sigma}}\)
(Lemma~\ref{lem:useful}~\ref{it:useful-mon}).

Suppose that \(\seq(\beta^-) = (a,\eta)\) with \(a \neq 0\). Let~\(\alpha\)
be the \(R_d\)\nbd node with \(\alpha \subset \alpha^\frown U \subseteq
\beta\), where the \(U\)\nbd outcome is a Type~II outcome and where
\(seq(\alpha^\frown U) = (a,\iota )\in m \times [T_\cL]\) (\(\iota = 00\cdots
0\)). We define
\[
\beta^\sharp = \alpha.
\]
If \(a = 0\) then we do not define~\(\beta^\sharp\).

\begin{remark}\label{rmk:starsharp}
\hspace{1em}
\begin{enumerate}
\item
In fact, we should have written~\(\beta^{*U}\) instead of~\(\beta^*\)
because we are discussing \(U\)\nbd functionals. Once we consider several
sets~\(U_j\), we will write~\(\beta^{*j}\) instead. If the index is clear
from the context, we simply write~\(\beta^*\). We proceed analogously
for~\(\beta^\sharp\).
\item
If~\(\beta^\sharp\) is defined, then so is~\({(\beta^\sharp)}^*\).
\item
The intuition is the following: Let~\(\beta\), an \(R_c\)\nbd node, be a
controller without a Type~I \(U\)\nbd outcome. Then~\(\beta^*\), an
\(R_d\)\nbd node, is defined and \(\beta^*\subsetneq \beta\). Suppose
that~\(\beta\) and~\(\beta^*\) are \(R_c\)\nbd strategies for the
same~\(c\). Informally, if there is a \(U\)\nbd change, we make
\({\beta^*}^\frown d\) active; otherwise, \(\beta^\frown d\) is active.
Suppose that they are \(R\)\nbd strategies with different~\(c\); then we
will wait for the stage such that \(\alpha = {(\beta^*)}^\sharp\) becomes a
controller and check if~\(\alpha\) and~\(\alpha^*\) are \(R_c\)\nbd
strategies for the same~\(c\). When~\({(\beta^*)}^\sharp\) is not defined
and the process cannot be continued, then some \(R\)\nbd node with a Type~I
\(U\)\nbd outcome is ready to become a controller (and this will be
discussed later).
\item
As an illustration, in Figure~\ref{fig:1Utree}, we have
\({(R_{c_\lambda}^4)}^* = {(R_{c_0}^3)}^* = R_{c_\lambda}^2\),
\({(R_{c_\lambda}^2)}^\sharp = R_{c_0}^1\), \({(R_{c_0}^1)}^*=
R_{c_\lambda}^0\), and~\({(R_{c_\lambda}^0)}^\sharp\) is not defined.
\end{enumerate}
\end{remark}

\subsection{Preliminaries}\label{sec:functionals}

We cover some standard notions and shorthand notations in this section. Given
a set~\(X\) of natural numbers, we usually think of it as an infinite binary
string. \(X\res l = \sigma\) if the length of \(\sigma\) is \(l\) and for
each \(n<l\), \(X(n)=\sigma(n)\). As usual, \(X_0 \oplus \cdots \oplus
X_{k-1} (kn+i)=X_i(n)\) for each \(i<k\). However, by a slight abuse of
notation, we let \((X_0 \oplus \cdots \oplus X_{k-1})\res l = X_0\res l
\oplus \cdots \oplus X_k\res l\), and the length of \((X_0 \oplus \cdots
\oplus X_{k-1})\res l\) is defined to be \(l\).

For a d.c.e.\ set \(U=W_i\setminus W_j\) for some c.e.\ sets~\(W_i\)
and~\(W_j\), we write~\(U_s\) for \(W_{i,s}\setminus W_{j,s}\).

We will use \(\Gamma, \Delta, \Phi, \Psi\) to denote Turing functionals. A
Turing functional~\(\Phi\) is a c.e.\ set consisting of ``axioms'' of the
form \((x,i,\sigma)\), where \(x \in \omega\) is the input, \(i \in \{0,1
\}\) the output, and \(\sigma \in 2^{<\omega}\) the oracle use; so
\((x,i,\sigma) \in \Phi\) denotes that \(\Phi^\sigma(x)=i\). Furthermore, if
\((x,i,\sigma),(x,j,\tau)\in \Phi\) for comparable~\(\sigma\) and~\(\tau\),
we require \(i = j\) and \(\tau = \sigma\). For \(X \subseteq \omega\), we
write \(\Phi^{X}(x)\downarrow = i\) if \((x,i,X \res l) \in \Phi\) for
some~\(l\) (the \emph{use function} \(\varphi(x)\) is defined to be the least
such~\(l\)); \(\Phi^{X}(x)\uparrow\) if for each~\(i\) and~\(l\) we have
\((x,i,X \res l) \notin \Phi\). \(x\) is a \emph{divergent point} of
\(\Phi^X\) if \(\Phi^{X}(x)\uparrow\). If \(\Phi_s\) is a c.e.\ enumeration
of \(\Phi\), where each \(\Phi_s\) is a finite subset of \(\Phi\), these
notions apply accordingly to~\(\Phi_s\).

For a Turing functional \(\Phi\) that is constructed by us stage by stage,
suppose that our intention is ensure \(\Phi^{X_0 \oplus \cdots \oplus
X_{k-1}}=Y\) where \(X_i\) or \(Y\) could be either a set with given fixed
enumeration or a set that is to be constructed by us. At stage \(s\), we say
that \(\Phi_s(n)\) (omitting the oracles and the set~\(Y\)) is \emph{correct}
if \(\Phi_s^{X_{0,s} \oplus \cdots \oplus X_{k-1,s}}(n)\downarrow=Y_s(n)\);
\emph{incorrect} if \(\Phi_s^{X_{0,s} \oplus \cdots \oplus X_{k-1,s}}(n)
\downarrow\neq Y_s(n)\); and \emph{undefined} if \(\Phi_s^{X_{0,s} \oplus
\cdots \oplus X_{k-1,s}}(n)\uparrow\). Suppose that \(\Phi_s^{X_{0,s} \oplus
\cdots \oplus X_{k-1,s}}(n)\) is undefined; then \emph{defining \(\Phi(n)\)
with use \(u\)} means that we enumerate \((n,Y_s(n),(X_{0,s} \oplus \cdots
\oplus X_{k-1,s}) \res u)\in \Phi_s\) so that \(\Phi_s^{X_{0,s} \oplus \cdots
\oplus X_{k-1,s}}(n) \) becomes correct. Note that whether \(\Phi_s(n)\) is
correct, incorrect, or undefined depends on a particular substage of
stage~\(s\), but this is usually clear from the context.

\subsection{Use blocks}\label{sec:use block}

Use blocks are the main source of both verbal and mathematical complexity of
our construction.

Consider a \(\Gamma\)\nbd functional that belongs to \(\mt(\alpha,U)\) for
some \(S\)\nbd node~\(\alpha\) (Definition~\ref{def:maintain1U}) and intends
to ensure \(\Gamma^{E\oplus U}=C\) for some \(c\in \Ji(\cL)\).
Suppose~\(\alpha\) is being visited at stage~\(s\), we define \(\Gamma(n)\)
\emph{with use block} \(\sB=[u-l,u)\)
means that we define \(\Gamma(n)\) with use~\(u\) and reserve the \emph{use
block}~\(\sB\) for future use. We also say that~\(\sB\) is \emph{defined for}
\(\Gamma(n)\) \emph{by}~\(\alpha\) at stage~\(s\); \(\sB\) \emph{belongs
to}~\(\Gamma\);~\(\sB\) is \emph{for} \(\Gamma(n)\);~\(\sB\) is
\emph{maintained by}~\(\alpha\). If the use block~\(\sB\) is a \emph{fresh}
use block, we define \(\bB_{\bday{s}}(\gamma,n)=\sB\) and
\(\created(\sB)=s\).

\begin{remark}
It will be seen (Section~\ref{sec:1U construction}: \(\visit(\alpha)\) for
\(S\)\nbd node) that \(\bB_{\bday{s}}(\gamma,n)\) is well defined because we
will not define \(\Gamma(n)\) twice at a single stage.
\end{remark}

The use block \(\sB=[u-l,u)\) is viewed as a \emph{potential subset}
of~\(E\). Enumerating (or extracting) a point~\(k\) with \(u-l\le k<u\)
in~\(\sB\) means letting \(E(k)=1\) (\(E(k)=0\), respectively). Similar to
the use function, the \emph{use block function} \(\bB_s(\gamma,n)\) is
defined to be the use block~\(\sB\) with \(\gamma_s(n)=\max \sB+1\) if
\(\gamma_s(n)\downarrow\); undefined if \(\gamma_s(n)\uparrow\). Different
from the notion \(\bB_{\bday{s}}(\gamma,n)\), the notion \(\bB_s(\gamma,n)\)
depends on a particular substage of stage~\(s\); conventionally, it is
usually evaluated when~\(\alpha\) is being visited and will be clear from the
context.

A use block~\(\sB\) defined at stage \(t<s\)
(\(\sB=\bB_{\bday{t}}(\gamma,n)\) for some~\(n\)) is \emph{available for
correcting~\(n\)} at stage~\(s\) if \(\sB=\bB_s(\gamma,n)\). The general idea
is the following: If \(\Gamma_s^{E\oplus U}(n)=j\neq C_s(n)\) and~\(\sB\) is
the use block that is available for correcting~\(n\), then we enumerate an
unused point into~\(\sB\) and so immediately \(\Gamma_s^ {E\oplus
U}(n)\uparrow\), then we can redefine \(\Gamma_s^{E\oplus U}(n)=1-j\) with
the \emph{same} use block~\(\sB\).

A use block~\(\sB\) for~\(\Gamma(n)\) can be \emph{killed}
(Section~\ref{sec:R and killing}) by some node~\(\beta\) at stage~\(s\), and
we define \(\killed(\sB)=s\) and write
\(\sB=\bB_{\dday{s}}^\beta(\gamma,n)\). A killed use block can still be
available for correcting~\(n\) in the future in which case we have to show
certain bad things will not happen to it. A use block which cannot be
available for correcting~\(n\) is good in the sense that it will not add any
complexity to our construction and such use block will be called
\emph{permanently killed}. Since a permanently killed use block will never
concern us, we are not using additional notation. For convenience, a
permanently killed use block is also said to be killed.

Let \(\sB_0=\bB_{\bday{s_0}}(\gamma,n)\) and
\(\sB_1=\bB_{\bday{s_1}}(\gamma,n)\) be two use blocks with \(s_0<s_1\) such
that at stage \(s_1\), \(\sB_0\) is killed and a point~\(x\) is in the use
block.
Suppose that~\(x\) is extracted at \(s_2>s_1\), then~\(\sB_1\) will never be
available for correcting~\(n\) in the future as~\(x\) will never be
enumerated back into~\(\sB_0\). In such case,~\(\sB_1\) is called permanently
killed. This phenomenon to~\(\sB_1\) will be handled \emph{tacitly}. The
other situation when we permanently kill a use block is
in~(\ref{it:correcting strategy 2b}) in Section~\ref{sec:R and delta}.

\begin{remark}
To call a use block killed or permanently killed is to request a certain
functional, say,~\(\Gamma\), to (re)define \(\Gamma(n)\) with a \emph{fresh}
use block when necessary.
\end{remark}

\(\bB_s(\gamma,n)\) is an interval and \(\gamma_s(n)\) is a natural number.
However, they are closely related. In a slight abuse of notation, if \(\bB_s
(\gamma,n)=[u-l,u)\) and \(u=\gamma_s(n)\), we write \(y<\gamma_s(n)<z\) if
\(y<u-l\) and \(u \le z\). A \emph{fresh} use block \(\sB=[u-l,u)\) is one
with \(u-l\) fresh and \(l\) sufficiently large. As we are either defining a
functional with the same use block or a fresh one, it turns out that all
these use blocks are pairwise disjoint and we can also leave sufficiently
large spaces between adjacent use blocks for diagonalizing witnesses picked
by \(R\) \nbd nodes or points enumerated by \(G\)\nbd requirements.

\begin{remark}
We need the size of \(\sB\) to be sufficiently large so that there is always
an unused element whenever we need one. This phenomenon occurs also in the
construction of a maximal incomplete d.c.e.\ degree. See Lemma~\ref{lem:block
size}.
\end{remark}

Next, consider a \(\Delta\)\nbd functional that belongs to \(\mt(\beta,U)\)
for some \(R\)\nbd node~\(\beta\) and intends to ensure \(\Delta^{E\oplus
C_0\oplus \cdots C_{k-1}}=U\), possibly without any \(c_i\in \Ji(\cL)\). It
is defined in essentially the same way except that~\(\Delta\) has additional
oracle sets built by us besides~\(E\). The use block~\(\sB\) is therefore a
potential subset of \(E\oplus C_0\oplus \cdots C_{k-1}\). We say that~\(\sB\)
\emph{crosses over}~\(E\) and~\(C_i\) for each \(i<k\). To enumerate (or
extract) a point~\(k\) into (or from)~\(\sB\) via \(X\in
\{E,C_0,\ldots,C_{k-1}\}\) is to let \(X(k)=1\) (or \(X(k)=0\),
respectively); by default, \(X=E\) if it is not explicitly mentioned.
Likewise, we have \(\bB_s(\delta,n)\), \(\bB_{\bday{s}}(\delta,n)\), and
\(\bB_{\dday{s}}^\eta(\delta,n)\) defined.

Suppose that \(\sB\) crosses over \(X\). The use block~\(\sB\) is
\emph{\(X\)\nbd restrained} if \(X\res B\) is restrained, in which case we
are not allowed to enumerate a point into or extract a point from~\(\sB\)
via~\(X\). \(\sB\) is \emph{restrained} if~\(\sB\) is \(X\)\nbd restrained
for some~\(X\).~\(\sB\) is \emph{\(X\)\nbd free} if~\(\sB\) is not \(X\)\nbd
restrained. As we will see, if~\(\sB\) is \(E\)\nbd restrained, then there
will be a set~\(C\) for some \(c\in \Ji(\cL)\) such that~\(\sB\) is \(C\)\nbd
free.

If it is not available for correcting~\(n\) at stage~\(s\), the use block
does not come into play at stage~\(s\) and we therefore do not worry about
it. However, if it is available for correcting~\(n\) at stage~\(s\), the use
block can be killed or \(E\)\nbd restrained in which case we have to be
cautious.

As a summary of the notations and also as a preview of what can happen to a
use block in the construction, let us consider a use block~\(\sB\).
\begin{enumerate}
\item
If~\(\sB\) is permanently killed, it will never be available for
correcting.
\item
If~\(\sB\) is available for correcting, and not restrained, we can do
whatever we want to the use block~\(\sB\) to make the correction.
\item
If~\(\sB\) is available for correcting \(\Gamma(n)\) and \(E\)\nbd
restrained, we will show that \(\Gamma(n)\) is in fact correct and needs no
additional correction.
\item
If~\(\sB\) is available for correcting \(\Delta^{E\oplus
C_1\oplus\dots\oplus C_{k-1}}(n)\) and \(E\)\nbd restrained, we will show
that there will be some \(i<k\) such that~\(\sB\) is \(C_i\)\nbd free so
that we can use~\(C_i\) to correct \(\Delta(n)\).
\end{enumerate}

We remark that the set~\(C\) for each \(c\in \Ji(\cL)\) will be built as a
c.e.\ set. Therefore if we enumerate a point into a use block via~\(C\), we
will not extract it. In fact,~(\ref{it:correcting strategy 2b}) in
Section~\ref{sec:R and delta} takes advantage of this.

\subsection{\texorpdfstring{The \(S\)-strategy}{The S-strategy}}%
\label{sec:S strategy}

Let~\(\beta\) be an \(S\)\nbd node with \(\seq(\beta)=(b,\xi)\). The idea of
the \(S\)\nbd strategy is straightforward: It builds and keeps each
\(\Gamma\)\nbd functional that belongs to \(\mt(\beta,U)\)
(Definition~\ref{def:maintain1U}) correct and total. At stage~\(s\), for each
\(\Gamma^{E\oplus U}=C\) that belong to \(\mt(\beta,U)\) and for each \(x\le
s\), \(\beta\) keeps \(\Gamma^{E\oplus U}(x)\) defined and correct according
to the following

\paragraph*{\emph{Correcting Strategy:}}
\begin{enumerate}
\item
Suppose \(\Gamma_s^{E\oplus U}(x)\downarrow=C_s(x)\).~\(\beta\) does
nothing.
\item
Suppose \(\Gamma_s^{E\oplus U}(x)\downarrow\neq C_s(x)\) with use block
\(\sB=\bB_s(\gamma,x)\).
\begin{enumerate}
\item
If~\(\sB\) is killed and not \(E\)\nbd restrained, then~\(\beta\)
enumerates an unused point, referred to as a \emph{killing point},
into~\(\sB\) via~\(E\). Then we go to~(3) immediately.
\item
If~\(\sB\) is not killed and not \(E\)\nbd restrained, \(\beta\)
enumerates an unused point, referred to as a \emph{correcting point},
into \(\sB\) via \(E\). Then we redefine \(\Gamma_s^{E\oplus U}(x) =
C_s(x)\) with the \emph{same} use block \(\sB\).
\end{enumerate}
\item
Suppose \(\Gamma_s^{E\oplus U}(x)\uparrow\). If each
\(\bB_{\bday{t}}(\gamma,x)\) with \(t<s\) has been \emph{killed} (see
Section~\ref{sec:R and killing} below), then~\(\beta\) will pick a fresh use
block \(\sB'\) and define \(\Gamma_s^{E\oplus U}(x) = C_s(x)\) with use block
\(\sB'\) (hence \(\sB'=\bB_{\bday{s}}(\gamma,x)\)); if otherwise, we define
\(\Gamma_s^{E\oplus U}(x)=C_s(x)\) with the use block that is not killed
(there will be at most one such use block).
\end{enumerate}

\begin{remark}
\begin{itemize}
\item
We will show that if~\(\sB\) is \(E\)\nbd restrained, then we will not have
Case~(2) in the correcting strategy.
\item
A correcting strategy never extracts a point from a use block.
\end{itemize}
\end{remark}

\subsection{\texorpdfstring{The \(R\)-strategy and the \(\Delta\)-functional}%
{The R-strategy and the Delta-functional}}\label{sec:R and delta}

Let~\(\beta\) be an \(R\)\nbd node. If~\(\beta\) decides to visit its
\(U\)\nbd outcome, it needs to build and keep the \(\Delta\)\nbd functional,
if any, that belongs to \(\mt(\beta,U)\) correct in essentially the same way
as an \(S\)\nbd node.

Let~\(\Delta\), if any, belong to \(\mt(\beta,U)\). Without loss of
generality, we assume that the \(\Delta\)\nbd functional is \(\Delta^{E\oplus
C_0\oplus \cdots \oplus C_{k-1}}=U\) for some \(c_i\in \Ji(\cL)\). For each
\(x\le s\),~\(\beta\) keeps \(\Delta^{E\oplus C_0\oplus \cdots \oplus
C_{k-1}}(x)\) defined and correct according to the following

\paragraph*{\emph{Correcting Strategy:}}
\begin{enumerate}
\item
Suppose that \(\Delta_s^{E\oplus C_0\oplus \cdots \oplus
C_{k-1}}(x)\downarrow=U_s(x)\).~\(\beta\) does nothing.
\item
Suppose that \(\Delta_s^{E\oplus C_0\oplus \cdots \oplus
C_{k-1}}(x)\downarrow\neq U_s(x)\) with use block \(\sB=\bB_s(\delta,x)\).
\begin{enumerate}
\item
If~\(\sB\) is killed and not \(E\)\nbd restrained, then~\(\beta\)
enumerates an unused point, referred to as a \emph{killing point},
into~\(\sB\) via~\(E\). Then we go to~(3) immediately.
\item\label{it:correcting strategy 2b}
If~\(\sB\) is killed and \(E\)\nbd restrained, we let~\(C_i\) for some
\(i<k\) be a set such that~\(\sB\) is \(C_i\)\nbd free (we will show that
such~\(C_i\) exists) and then~\(\beta\) enumerates an unused point,
referred to as a \emph{killing point}, into~\(\sB\) via~\(C_i\).~\(\sB\)
is then \emph{permanently killed} (as~\(C_i\) will be a c.e.\ set). Then
we go to~(3) immediately.
\item
If~\(\sB\) is not killed and not restrained, then~\(\beta\) enumerates an
unused point, referred to as a \emph{correcting point}, into~\(\sB\)
via~\(E\). Then we define \(\Delta_s^{E\oplus C_0\oplus \cdots \oplus
C_{k-1}}(x)=U_s(x)\) with the same use block~\(\sB\).
\item
If~\(\sB\) is not killed but \(E\)\nbd restrained, we let~\(C_i\) for
some \(i<k\) be a set such that~\(\sB\) is \(C_i\)\nbd free (we will show
such~\(C_i\) exists), and then~\(\beta\) enumerates an unused point,
referred to as a \emph{correcting point}, into~\(\sB\) via~\(C_i\). Then
we define \(\Delta_s^{E\oplus C_0\oplus \cdots \oplus C_k}(x)=U_s(x)\)
with the same use block~\(\sB\).
\end{enumerate}
\item
Suppose that \(\Delta_s^{E\oplus C_0\oplus \cdots \oplus C_k}(x)\uparrow\).
If for each \(t<s\), \(\bB_{\bday{t}}(\delta,x)\) is killed, then~\(\beta\)
will choose a fresh use block~\(\sB'\) and define \(\Delta_s^{E\oplus
C_0\oplus \cdots \oplus C_k}(x) = U_s(x)\) with use block~\(\sB'\) (hence
\(\sB'=\bB_{\bday{s}}(\delta,n)\)); otherwise, we will define
\(\Delta_s^{E\oplus C_0\oplus \cdots \oplus C_k}(x) = U_s(x)\) with the use
block that is not killed (there will be at most one such use block).
\end{enumerate}

At stage~\(s\), if the \(U\)\nbd outcome is visited, then~\(\beta\) will
follow the above instructions for each \(x\le s\) and we have
\(\Delta_s^{E\oplus C_0\oplus \cdots \oplus C_k}(x)\) defined and correct.

Note that~\(\beta\) only enumerates correcting points into the use block
whenever a correction is needed. Thus, if \(U(x)\) changes twice in a row, we
will have two correcting points in the use block.

\subsection{\texorpdfstring{The \(R\)-strategy and the killing-strategy}%
{The R-strategy and killing}}\label{sec:R and killing}

Let~\(\beta\) be an \(R\)\nbd node.~\(\beta\) will first pick a
\emph{threshold point} denoted by~\(\tp(\beta)\). If~\(\beta\) decides to
visit its \(U\)\nbd outcome, it needs to \emph{kill} each functional that
belongs to \(\kl(\beta,U)\), if any. The intention is that the use of each of
these functionals at the point~\(\tp(\beta)\) should go to~\(\infty\)
as~\(s\) goes to~\(\infty\). To be precise:

Let \(\Gamma^{E\oplus U}=C\) (a \(\Delta\)\nbd functional is dealt with
similarly) belong to \(\kl(\beta,U)\) and \(v=\tp(\beta)\). We suppose that
\(\seq(\beta^-) = (b,\xi)\) and note that at stage~\(s\), when~\(\beta\) is
visited, (the \(b\)\nbd th copy of) each functional in \(F_\xi(U)\) is
defined and correct by the correcting strategies. Hence
\(\bB_s(\gamma,x)\downarrow\) for each \(x\le s\) after~\(\beta^-\) finishes
its job. Then~\(\beta\) executes the following

\paragraph*{\emph{Killing strategy}:}

For each~\(x\) with \(v\le x\le s\), let \(\sB_x = \bB_s(\gamma,x)\) be the
use block. We enumerate an unused point, referred to as a \emph{killing
point}, into~\(\sB_x\) and declare that~\(\sB_x\) is \emph{killed} (hence
\(\sB_x=\bB_{\dday{s}}^\beta(\gamma,x)\)).

One easily sees that if~\(\sB_x\) contains a killing point, then~\(\sB_x\)
will never be available for correcting~\(x\) (Section~\ref{sec:use block}).

\subsection{\texorpdfstring{The \(R\)-strategy and its computation with
slowdown condition}%
{The R-strategy and its computation with slowdown condition}}%
\label{sec:R and computation}

Let~\(\beta\) be an \(R_c\)\nbd node, where \(c\in \Ji(\cL)\) and
\(\spec(c*c_*)= \{c_1,\ldots, c_k \}\). Suppose that~\(\beta\) is assigned to
the requirement \(R_c(\Psi): C\neq \Psi^{E\oplus C_1\oplus \cdots\oplus
C_k}\). Let \(v=\tp(\beta)\) and \(x>v\) be a diagonalizing witness.

Suppose that at stage~\(s\), we have
\[
    \Psi_s^{E\oplus C_1\oplus \cdots C_k}(x)\downarrow = 0
\]
with use \(\psi_s(x)\). Let \(\sigma = (E \oplus C_1 \oplus \cdots \oplus
C_k)[s]\res \psi_s(x)\). (Again, we will tacitly actually define \(\sigma =
(E[s]\res \psi_s(x)) \oplus (C_1[s]\res \psi_s(x)) \oplus \cdots \oplus
(C_k[s]\res \psi_s(x))\).) Now, abusing notation, we will let
\(y=\psi_s(x)-1\) refer not only to the number \(\psi_s(x)-1\), but also to
the string~\(\sigma\). We say~\(y\) is \emph{the computation for~\(\beta\) at
stage~\(s\)}, while the diagonalizing witness~\(x\) is understood from
context. At stage \(t>s\),~\(y\) is \emph{restored} if \((E \oplus C_1 \oplus
\cdots \oplus C_k)[t]\res y+1 = \sigma\); \emph{injured} if otherwise.
Usually, we will focus on part of the whole computation. Given a use block
\(\sB\), \(y\res B\) is \emph{restored} if for each \(x\in B\), \((E \oplus
C_1 \oplus \cdots \oplus C_k)[t]\res B = \sigma\res B\); \emph{injured} if
otherwise. If~\(y\) is restored eventually and \(C(x)=1\), then
\(\Psi_s^{E\oplus C_1\oplus \cdots C_k}(x)\downarrow = 0\neq C(x)\) and the
requirement \(R_c(\Psi)\) is satisfied.

For a technical reason, we introduce a slowdown condition when finding a
computation. This delay feature is also used in the proof of the original
D.C.E.\ Nondensity Theorem.

\begin{definition}\label{def:relevant}
For any node~\(\alpha\), a set~\(X\) is \emph{relevant} to~\(\alpha\)
if~\(X\) is either \(K, E\), or~\(C\) for some \(c \in \Ji(\cL)\), or~\(X =
U\) if an \(S(U)\)\nbd requirement is assigned to some node \(\gamma
\subseteq \alpha\). (Of course, in the current section, we have only one
set~\(U\).)
\end{definition}

Clearly, there are only finitely many sets that are relevant to a fixed node.

\begin{definition}
Let~\(X\) be a set,~\(y\) be a number, and \(t\le s\) be two stages. We
define
\begin{itemize}
\item
\(\same(X,y,t,s)\) iff \(X_t \res (y+1) = X_s \res (y+1)\).
\item
\(\diff(X,y,t,s)\) iff \(\lnot\same(X,y,t,s)\).
\item
\(\bigsame(X,y,t,s)\) iff for all~\(s'\) with \(t \le s'\le s\),
we have \(\same(X,y,s',s)\).
\end{itemize}
\end{definition}

\(\bigsame()\) checks if a set is stable,~ \(\same()\) and~\(\diff()\) will
tell us which computations will be restored (as will be seen later).

\begin{definition}[computation with slowdown condition]\label{def:sd}
Let~\(\beta\) be an \(R\)\nbd node,~\(s\) the current stage, and \(s^*<s\) be
the last \(\beta\)\nbd stage. Let~\(y\) be the computation for~\(\beta\) at
stage~\(s\).~\(y\) is the computation \emph{with slowdown conditions} if the
following is satisfied:
\begin{enumerate}
\item\label{it:sd 1}
For each~\(X\) that is relevant to~\(\beta\), we have
\[
\bigsame(X,y,s^*,s).
\]
(If~\(s^*\) is not defined, then \(\bigsame(X,y,s^*,s)\) is defined to be
false.)
\item\label{it:sd 2}
If~\(n\) is a point enumerated into some use block \(\sB=[u-l,u)\) by a
node \(\alpha\subsetneq \beta\) at the same stage, then the
computation~\(y\) should also satisfy \(y<u-l\).
\item\label{it:sd 3}
\(y<s^*\).
\end{enumerate}
\end{definition}

If~\(\beta\) does not find a computation with slowdown condition,
then~\(\beta\) simply visits its \(w\)\nbd outcome. Clearly, imposing the
slowdown condition only delays~\(\beta\) for finitely many stages if the
computation actually converges.

In the rest of the paper, a computation always refers to a computation with
slowdown condition.

\subsection{\texorpdfstring{The \(R\)-strategy and the \(\varnothing\)-data}%
{The R-strategy and the emptyset-data}}\label{sec:R and data}

Let~\(\beta\) be an \(R\)\nbd node. If a computation is not found
by~\(\beta\), we should visit the \(w\)\nbd outcome of~\(\beta\). If a
computation~\(y\) is found, then we might be ready to visit the \(U\)\nbd
outcome and make some progress. As we will recursively collect a bunch of
computations found at various \(R\)\nbd nodes as our \emph{data}, we put the
single computation~\(y\) into the same package as our base step.

\begin{definition}[\(\varnothing\)-data]\label{def:0 data}
Let~\(\beta\) be an \(R\)\nbd node. Suppose that a computation~\(y\) is found
at stage~\(s\). Let the \emph{\(\varnothing\)\nbd data} of~\(\beta\), denoted
by \(\cE_s^\varnothing(\beta)\), consist of the following:
\begin{enumerate}
\item
a set of nodes \(\cE_s^\varnothing(\beta)=\{\beta\}\) (slightly abusing
notation),
\item
the computation~\(y\) for~\(\beta\).
\end{enumerate}
If there is no confusion, we might drop the subscript~\(s\) of
\(\cE_s^\varnothing(\beta)\).
\end{definition}

Before we have a long discussion on the \(U\)\nbd outcome, let us have a
quick overview of the strategy of an \(R\)\nbd node~\(\beta\). If the
\(d\)\nbd outcome is \emph{activated}, then~\(\beta\) visits this \(d\)\nbd
outcome, claiming that the \(R\)\nbd requirement is satisfied by doing
nothing. In all other cases, after~\(\tp(\beta)\) is defined and a
diagonalizing witness \(x>\tp(\beta)\) is picked,~\(\beta\) tries to obtain
the \(\varnothing\)\nbd data. If it fails to obtain the \(\varnothing\)\nbd
data, then~\(\beta\) visits the \(w\)\nbd outcome, claiming (eventually) that
a disagreement has been found and the \(R\)\nbd requirement is therefore
satisfied. If it obtains the \(\varnothing\)\nbd data, then~\(\beta\)
\emph{encounters each of the other outcomes in order}
(Definition~\ref{def:1Utree}) and decides what to do next and which one of
the outcomes to visit. In the current section~\ref{sec:1U}, the first outcome
is a \(U\)\nbd outcome, and the second outcome is a \(\ctr\)\nbd outcome.

\subsection{\texorpdfstring{The \(R\)-strategy and the \(U\)-outcome}%
{The R-strategy and the U-outcome}}\label{sec:R and U}

Suppose that the current stage is~\(s\) and~\(\beta\) is an \(R\)\nbd node
with \(\varnothing\)\nbd data \(\cE_s^\varnothing(\beta)\). Now we encounter
the first outcome of~\(\beta\), which is always a \(U\)\nbd outcome. The
action we take depends on whether this \(U\)\nbd outcome is Type~I, GREEN, or
RED\@. Let us assume that \(\seq(\beta^-)=(a,\eta)\in m \times [T_\cL]\) and
\(\seq(\beta^\frown U) = (b, \xi)\), if the latter is defined. Let
\(v=\tp(\beta)\).

If the \(U\)\nbd outcome is Type~I, then we visit it. By visiting this
outcome, we kill each functional that belongs to \(\kl(\beta,U)\) (see
Section~\ref{sec:R and killing} for the killing strategy), and we define and
keep the \(\Delta\)\nbd functional that belongs to \(\mt(\beta,U)\) correct
(see Section~\ref{sec:R and delta} for the correcting strategy).
\(\cE_s^\varnothing(\beta)\) is \emph{not} discarded.

If the \(U\)\nbd outcome is GREEN, then we visit it. We also kill each
functional that belongs to \(\kl(\beta,U)\) by the killing strategy. In this
case, there is no \(\Delta\)\nbd functional to define, and
\(\cE_s^\varnothing(\beta)\) \emph{is} discarded.

If the \(U\)\nbd outcome is RED, then we do \emph{not} visit it. Recall from
Section~\ref{sec:star and sharp} that~\(\beta^*\) is defined. Note also
that~\(s\) is a \(\beta^*\)\nbd stage visiting the \(U\)\nbd outcome
of~\(\beta^*\) with its \(\varnothing\)\nbd data
\(\cE_s^\varnothing(\beta^*)\). We now combine \(\cE_s^\varnothing(\beta^*)\)
and \(\cE_s^\varnothing(\beta)\) and add some information as follows.

\begin{definition}[\(U\)-data]\label{def:1U data}
At stage~\(s\), suppose that \(\cE_s^\varnothing(\beta)\) is obtained. If the
first \(U\)\nbd outcome for \(\beta\) is RED, we let~\(\beta^*\) be defined
as in Section~\ref{sec:star and sharp}. Let~\(y_{\beta^*}\) and~\(y_\beta\)
be the computations for~\(\beta^*\) and~\(\beta\), respectively. Let the
\emph{\(U\)\nbd data} \(\cE_s^U(\beta)\) consist of the following:
\begin{enumerate}
\item
a set of nodes
\(\cE_s^U(\beta)=\cE_s^\varnothing(\beta)\cup\cE_s^\varnothing(\beta^*)\),
\item
a \(U\)\nbd reference stage~\(s\) for each \(\xi\in \cE_s^U(\beta)\),
\item\label{it:1U data1 3}
for each \(\xi\in \cE_s^\varnothing(\beta)\), a~\emph{\(U\)\nbd condition}
\(\same(U,y_\xi,s,t)\) with \emph{\(U\)\nbd reference length} \(y_\xi\) and
variable~\(t\), and
\item\label{it:1U data1 4}
for each \(\xi\in \cE_s^\varnothing(\beta^*)\), a~\emph{\(U\)\nbd
condition} \(\diff(U,y_\beta,s,t)\) with \emph{\(U\)\nbd reference length}
\(y_\beta\) and variable~\(t\).
\end{enumerate}
We denote the \emph{\(U\)\nbd condition} for each \(\xi\in \cE_s^U(\beta)\)
by \(\Cond^U(\xi,t)\) where~\(t\) is a variable. If there is no confusion, we
might drop the subscript~\(s\) of~\(\cE_s^U(\beta)\).
\end{definition}
The reference stages for \(\cE_s^U(\beta)\), \(\cE_s^\varnothing(\beta)\),
and \(\cE_s^\varnothing(\beta^*)\) happen to be the same for now. We are very
careful about the reference length in the above definition to reflect the
dependence of each parameter: the reference length in~\eqref{it:1U data1 3}
will follow the same idea when we discuss multiple \(S(U)\)\nbd requirements
in Section~\ref{sec:2U}, while the reference length in~\eqref{it:1U data1 4}
still needs to be modified.

To demonstrate this situation, let us look at an example and see
how~\(\cE_s^U(\beta)\) can be helpful. This example reminds the reader of the
essential idea in the proof of the original D.C.E.\ Nondensity Theorem, where
we embed the \(2\)\nbd element chain, but in a more general setting.

\begin{example}\label{eg:1U good controller}
Let us consider the case in Figure~\ref{fig:1Utree}. Recall from Figure~\ref
{fig:lattice3-S} that \(F_{00}(U)\) consists of \(C_\lambda=\Gamma_\lambda^
{E\oplus U}\) and \(C_0=\Gamma_0^{E\oplus U}\), \(F_{01}(U)\) consists of
\(C_\lambda=\Gamma_\lambda^{E\oplus U}\) and \(U=\Delta_0^{E\oplus
C_\lambda}\), and \(F_{1}(U)\) consists of \(U=\Delta_\lambda^E\). For the
easy case, let us ignore~\(R_{c_0}^3\) and consider \(\beta=R_{c_\lambda}^4\)
and \(\beta^*=R_{c_\lambda}^2\). Let~\(y_2\) and~\(y_4\) denote the
computation for~\(R_{c_\lambda}^2\) and~\(R_{c_\lambda}^4\), respectively,
and \(\cE_s^U(R_{c_\lambda}^4)\) be the \(U\)\nbd data currently obtained
when~\(R_{c_\lambda}^4\) encounters the RED \(U\)\nbd outcome at stage~\(s\).
As we can assume that~\(y_2\) is larger than it actually is as long as our
construction allows restoring those extra digits, we assume that \(y_4<y_2\).
Note that~\(y_2\) is injured at the end of stage~\(s\) by the killing
strategy of \(R_{c_\lambda}^2\). It is also natural for us to enumerate the
diagonalizing witnesses~\(x_4\) and~\(x_2\) with \(x_4<x_2\) for
\(R_{c_\lambda}^4\) and \(R_{c_\lambda}^2\), respectively,
into~\(C_\lambda\). Then the \(S\)\nbd node assigned~\(10\) potentially has
to make corrections to \(\Gamma_\lambda^{E\oplus U}(x_4)\) and
\(\Gamma_\lambda^ {E\oplus U}(x_2)\) in the future, which potentially
injures~\(y_2\). Let \(s^*<s\) be the last stage when \(R_{c_\lambda}^4\) is
visited. With slowdown condition, we have \(\bigsame(U,y_4,s^*,s)\). Let
\(w=\tp(R_{c_\lambda}^2)\).

With the \(U\)\nbd data and the above observations, we now discuss under
which conditions we can restore~\(y_2\) or~\(y_4\).

Suppose that \(t>s\) is a stage when we have \(\diff(U,y_4,s,s_0)\) and are
visiting the \(S\)\nbd node assigned~\(10\). We intend to explain why~\(y_2\)
can be restored while this \(S\)\nbd node can keep~\(\Gamma_\lambda\)
and~\(\Gamma_0\) correct. We consider a use block \(\sB<y_2\) that looks
different before and after the moment when~\(y_2\) is found. The goal is to
show that \(\sB\) is either not available for correcting, or does not need
any correction. There are three kinds of such use blocks,
\(\bB_{\dday{s}}^{R^2}(\gamma_\lambda,n)\) for some \(n\ge w\),
\(\bB_{\dday{s}}^{R^2}(\gamma_0,n)\) for some \(n\ge w\), and
\(\bB_{\bday{s'}}(\delta_\lambda,n)\) for some~\(n\) and some \(s'\le s\),
where the first two are similar.

For a use block \(\sB=\bB_{\dday{s}}^{R^2}(\gamma_\lambda,n)\) for some
\(n\ge w\), we note that \(\created(\sB)>s^*\) because otherwise it would
have been killed at~\(s^*\). By \(\bigsame(U,y_4,s^*,s)\) and
\(\diff(U,y_4,s,t)\), we realize that even if we restore \(y_2\res \sB\) at
stage~\(t\),~\(\sB\) is not available for correcting~\(n\). Therefore
restoring~\(y_2\) and keeping~\(\Gamma_\lambda(n)\) correct at the \(S\)\nbd
node~\(10\) creates no conflicts. The same argument applies to a use block
\(\bB_{\dday{s}}^{R^2}(\gamma_0,n)\) for each \(n\ge w\). For the use block
\(\sB'=\bB_{\bday{s'}}(\delta_\lambda,n)\) for some \(n\) and some \(s'\le
s\), we simply restore \(y\res \sB'\) because if we have restored~\(y_2\), we
are going to claim satisfaction of~\(R_{c_\lambda}^2\), and we are not going
to visit the \(U\)\nbd outcome of~\(R_{c_\lambda}^2\) anymore and hence we do
not have to keep~\(\Delta_\lambda\) correct. By the way, if we have
\(\diff(U,y_4,s,t)\), each use block~\(<y_2\) will be \(E\)\nbd restrained at
stage~\(t\).

Suppose that at stage \(t>s\), we have \(\same(U,y_4,s,t)\) and are visiting
\({R_{c_\lambda}^2}^\frown U\). We intend to explain why~\(y_4\) can be
restored. We will only consider the use block \(\sB<y_4\) that looks
different before and after the moment when~\(y_4\) is found. First of all,
since \(y_4<B_{\dday{s}}^{R^2}(\gamma_\lambda,w)\) and
\(y_4<B_{\dday{s}}^{R^2}(\gamma_0,w)\) by the Slowdown Condition
(Definition~\ref{def:sd}~(\ref{it:sd 2})),
we will only consider a use block \(\sB=\bB_{\bday{s'}}(\delta,n)<y_4\) for
some \(s'\le s^*\) and some \(n<y_4\). \(\Delta_\lambda(n)\) is correct at
stage~\(s^*\) as otherwise we would enumerate a correcting point into~\(\sB\)
and hence \(y_4<\sB\) by the Slowdown Condition. As we have
\(\bigsame(U,y_4,s^*,s)\), no point will be enumerated into~\(\sB\) during
each stage~\(s''\) with \(s^*\le s''\le s\). For each \(t'>s\), if we have
\(\diff(U,y_4,s,t')\), we restore~\(y_2\) and do not maintain this use
block~\(\sB\); if \(\same(U,y_4,s,t')\), then \(\Delta_\lambda(n)\) is
correct as it was at stage~\(s^*\). In other words, we will never enumerate a
point into~\(\sB\) and~\(y_4\) will never be injured. Therefore, it is safe
for us to restore~\(y_4\) and activate the \(d\)\nbd outcome of
\(R_{c_\lambda}^4\), claiming that \(R_{c_\lambda}^4\) is satisfied. By the
way, if we have \(\same(U,y_4,s,t)\), this use block~\(\sB\) will be
\(E\)\nbd restrained at stage~\(t\).

For each \(t>s\), we have either \(\diff(U,y_4,s,t)\) or
\(\same(U,y_4,s,t)\). This gives us a \emph{decision map}
\(\cD_t(R_{c_\lambda}^4)=\xi\) when \(\Cond^U(\xi,t)\)
(Definition~\ref{def:1U data}) holds. According to the decision map, we
decide which computation is to be restored at each stage.
\end{example}

From this Example~\ref{eg:1U good controller}, we see that
\(\cE_s^U(R_{c_\lambda}^4)\) contains all information to decide
whether~\(y_2\) or~\(y_4\) will be restored at each stage \(t>s\), and they
can really be restored while functionals that belong to \(\mt(01,U)\) or
\(\mt(R_{c_\lambda}^2,U)\) can be kept correct depending on which computation
is restored. This motivates that we should stop collecting more data and get
ready to make some progress. This is exactly what the \(\ctr\)\nbd outcome
suggests: After encountering a RED outcome and obtaining \(\cE_s^U(\beta)\),
we encounter the second outcome of \(\beta\), which is a \(\ctr\)\nbd
outcome.~\(\beta\) is ready to become a \emph{controller}.

\subsection{\texorpdfstring{The \(R\)-strategy and the controller, part~1}%
{The R-strategy and the controller, part~1}}\label{sec:R and controller 1}

As a \(\ctr\)\nbd outcome is never the first outcome, it will be clear from
our construction that whenever an \(R\)\nbd node~\(\beta\) encounters the
\(\ctr\)\nbd outcome, it must have obtained the \(U\)\nbd data
\(\cE^U(\beta)\) (Definition~\ref{def:1U data}).

\begin{definition}[controller]\label{def:1U controller}
At stage~\(s\), suppose that~\(\beta\) is an \(R\)\nbd node encountering the
\(\ctr\)\nbd outcome with \(U\)\nbd data \(\cE^U(\beta) =
\cE^\varnothing(\beta)\cup\cE^\varnothing(\alpha) \) for some~\(\alpha\) (it
will be shown that~\(\beta\)~and \(\alpha\) are related in a certain way).
Let \(\cE^{\ctr}(\beta)=\cE^U(\beta)\) (a modification will be needed in
Section~\ref{sec:2U}). Suppose \(\seq(\beta^-)=(b,\xi)\). We say that
\(\beta\) becomes a \(U^b\)\nbd \emph{controller} (or \emph{controller} for
simplicity). The controller~\(\beta\) inherits the priority from the terminal
node \(\beta^\frown \ctr\) on the priority tree~\(\cT\). Suppose
that~\(\beta\) and~\(\alpha\) are~\(R_c\)- and \(R_d\)\nbd nodes,
respectively, for some \(c,d\in \Ji(\cL)\) with \(d\le c\) (we will not need
to consider the case when \(d>c\)).
\begin{enumerate}
\item\label{it:1U controller1}
If \(d=c\), then we say that the~\(\beta\) has \emph{no \(U^b\)\nbd
problem}.
\item\label{it:1U controller2}
If \(d<c\) (see Section~\ref{sec:star and sharp}), then we say
that~\(\alpha\) is the \emph{\(U^b\)\nbd problem} (or \(U\)\nbd problem for
short) for~\(\beta\). (See Example~\ref{eg:1U bad controller} below.)
\end{enumerate}
In both cases, we let \(\Cond_\beta^U(\xi,t)\) be \(\Cond^U(\xi,t)\) for each
\(\xi\in \cE^U(\beta)\). Let \(s_\beta^{\ctr}\) denote current stage~\(s\).

While~\(\beta\) is not initialized, \(\hat{C}\res s_\beta^{\ctr}\) is
restrained for each \(\hat{c}\in \Ji(\cL)\) with \(\hat{c}\neq c\).
\end{definition}

\begin{definition}[decision map]\label{def:1U decision}
Let~\(\beta\) become a controller with \(\cE^{\ctr}(\beta)\). For each \(s\ge
s_\beta^{\ctr}\), the \emph{decision map} is defined by setting
\(\cD_s(\beta)=\xi\) for the longest \(\xi\in \cE^{\ctr}(\beta)\) with
\(\Cond_\beta^U(\xi,t)\). The controller~\(\beta\) \emph{changes its
decision} (at stage~\(s\)) if \(\cD_s(\beta)\neq \cD_{s-1}(\beta)\).
When~\(s\) is clear from context, we write \(\cD(\beta)=\xi\) for short.
\end{definition}

If \(\cD_s(\beta)=\xi\), we would like to show that~\(y_\xi\) can be
restored, and we also put a restraint on \(E\res y_\xi\) at stage~\(s\).

\begin{definition}[noise]\label{def:noise}
Let~\(\beta\) be a controller with \(\cE^{\ctr}(\beta)\). At the beginning of
stage \(s> s_\beta^{\ctr}\), if there is some set~\(X\) that is relevant
to~\(\beta\) (Definition~\ref{def:relevant}) such that
\[
\diff(X,s_\beta^{\ctr},s-1,s),
\]
then~\(\beta\) \emph{sees some noise} at stage~\(s\).
\end{definition}

While it is not initialized, a controller~\(\beta\) can see at most finitely
much noise and hence changes its decision at most finitely many times.
Therefore, if~\(\beta\) sees some noise, we can safely initialize all nodes
to the right of \(\beta^\frown \ctr\). In this Section~\ref{sec:1U},
``longest'' in Definition~\ref{def:1U decision} does not matter as there will
be a unique choice; it will matter in Section~\ref{sec:2U}. From
\(\Cond^U(\beta,t)\) and \(\Cond^U(\beta^*,t)\) defined in
Definition~\ref{def:1U data} and from \(\cE^{\ctr}(\beta)=\cE_s^U(\beta)\),
we see that at each \(s\ge s_\beta^{\ctr}\), \(\cD_s(\beta)\) is always
defined.

If~\(\beta\) becomes a controller at stage~\(s\), we stop the current stage.
Example~\ref{eg:1U good controller} is Case~\eqref{it:1U controller1} in
Definition~\ref{def:1U controller}. We now discuss Case~\eqref{it:1U
controller2} in the next example.

\begin{example}\label{eg:1U bad controller}
Different than in Example~\ref{eg:1U good controller}, we now
consider~\(R_{c_0}^3\) instead of~\(R_{c_\lambda}^4\). Suppose
that~\(R_{c_0}^3\) is encountering its \(\ctr\)\nbd outcome with data
\(\cE^U(\beta) = \cE^\varnothing(R_{c_0}^3)\cup
\cE^\varnothing(R_{c_\lambda}^2)\) (we drop the subscript~\(s\) if there is
no confusion). We let \(\cE^{\ctr}(\beta)=\cE^U(\beta)\) but realizing that
\(R_{c_\lambda}^2\) is a \(U^1\)\nbd problem for \(R_{c_0}^3\)
(Definition~\ref{def:1U controller}\ref{it:1U controller2}). Let
\(s_*=s_{R_{c_0}^3}^{\ctr}\).

Why is~\(R_{c_\lambda}^2\) called a problem? Because its diagonalizing
witness~\(x_2\) should be enumerated into~\(C_\lambda\) while~\(C_\lambda\)
also belongs to the oracle of~\(R_{c_0}^3\). This is the potential conflict.
(A plausible attempt is to assume \(y_3<x_2\) by patiently waiting so that
enumerating~\(x_2\) does not injure~\(y_3\). However, this does not work if
we have multiple \(S\)\nbd requirements.) The solution we take here is that
we enumerate only~\(x_3\) into~\(C_0\) and keep~\(x_2\) out of~\(C_\lambda\).
Note that enumerating~\(x_3\) into~\(C_0\) does not injure~\(y_2\)
because~\(C_0\) does not belong to the oracle of~\(R_{c_\lambda}^2\) (this is
not a coincidence -- see Section~\ref{sec:star and sharp}).

To emphasize the fact to keep~\(x_2\) out of~\(C_\lambda\), the
controller~\(R_{c_0}^3\) puts a restraint on \(C_\lambda\res s_*\) (also for
the sake of the \(G\)\nbd requirement discussed later). The restraint is not
canceled unless the controller \(R_{c_0}^3\) is initialized.

Suppose \(\same(U,y_3,s_*,s)\), then we have \(\cD(R_{c_0}^3)=R_{c_0}^3\) and
\(E\res y_3\) is restrained. Similar to Example~\ref{eg:1U good controller},
we can restore~\(y_3\) and activate the \(d\)\nbd outcome of~\(R_{c_0}^3\) in
this situation.

Suppose \(\diff(U,y_3,s_*,s)\), we now have a different situation than in
Example~\ref{eg:1U good controller}. First of all,~\(x_2\) is not enumerated,
so we are not attempting to activate the \(d\)\nbd outcome
of~\(R_{c_\lambda}^2\). However, we still want to restore~\(y_2\). In fact,
this can still be done and it is done in a very crude way: we simply
restore~\(y_2\) and ignore the impact on each use block \(\sB<y_2\). Then we
turn the GREEN \(U\)\nbd outcome of \(R_{c_0}^1\) (\(=
(R_{c_\lambda}^2)^{\sharp}\)) into a RED outcome. By doing so, we will not
visit the \(S\)\nbd node assigned~\(10\) and hence we do not need to keep
those functionals correct and ignoring each use block \(\sB<y_2\) is
legitimate. In this case, we also declare that \(E\res y_2\) is restrained
although it is not necessary. If~\(R_{c_0}^3\) never changes its decision,
then our construction will run until~\(R_{c_0}^1\) becomes a controller at
\(s_{**}=s_{R_{c_0}^1}^{\ctr}\).

At each stage \(s>s_{**}\), we will encounter the following cases.
If~\(R_{c_0}^3\) sees some noise (Definition~\ref{def:noise}), \(R_{c_0}^1\)
is initialized (\(R_{c_0}^1\) has lower priority than~\(R_{c_0}^3\) as the
node \((R_{c_0}^1)^\frown \ctr\) is to the right of \((R_{c_0}^3)^\frown
\ctr\)). If \(\cD(R_{c_0}^1)=R_{c_0}^1\), we can have~\(R_{c_0}^1\) restored
and activate the \(d\)\nbd outcome of~\(R_{c_0}^1\). However, if
\(\cD(R_{c_0}^1)=R_{c_\lambda}^0\), we cannot mimic what~\(R_{c_0}^3\) did
because~\((R_{c_\lambda}^0)^{\sharp}\) is not defined this time. We will
discuss this situation in the next example.
\end{example}

Before leaving this section, we give some definitions that will be used
throughout the rest of the paper. Recall from Section~\ref{sec:R and
computation} that a computation~\(y\) is restored at (a substage of)
stage~\(t\) if \((E\oplus C_1\oplus \cdots C_k)[t]\res y+1=y\); and is
injured otherwise. By \(U\)\nbd restoring~\(y_\beta\) at the beginning of
stage~\(s\), we mean that we set \((E_s\oplus C_{1,s} \oplus \cdots \oplus
C_{k,s})\res \sB = y_\beta\res \sB\) (viewing~\(y_\beta\) as a string) for
each use block \(\sB<y_\beta\) that belongs to a \(U\)\nbd functional. In
this section, each use block belongs to some \(U\)\nbd functional, and
therefore \(U\)\nbd restoring~\(y_\beta\) is the same as
restoring~\(y_\beta\).

\begin{definition}[\(U\)-restorable]\label{def:U restorable}
Let~\(\beta\) be given and~\(y_\beta\) be its computation. We say
that~\(y_\beta\) is \emph{\((U,C)\)\nbd restorable} at stage~\(s\) if, after
we \(U\)\nbd restore~\(y_\beta\) and set \(E\res y_\beta\) to be restrained
at the beginning of stage~\(s\), for each \(\alpha\subseteq \beta\) and for
each use block \(\sB<y_\beta\) (hence~\(\sB\) is \(E\)\nbd restrained)
maintained by~\(\alpha\), we have the following:
\begin{enumerate}
\item\label{it:U restorable 1}
If~\(\sB\) is a use block for \(\Gamma^{E\oplus U}(n)=D\), where
\(n>\tp(\beta)\), \(\Gamma\in\mt(\alpha,U)\) (so~\(\alpha\) is an \(S\)\nbd
node), and~\(D\) is not necessarily different from~\(C\), then either
\begin{enumerate}
\item\label{it:U restorable 1a}
\(\sB\) is available for correcting~\(n\) and \(\Gamma^{E\oplus
U}(n)=D_s(n)\) (i.e., \(\Gamma(n)\) is correct and hence needs no
correction), or
\item\label{it:U restorable 1b}
\(\sB\) is not available for correcting~\(n\).
\end{enumerate}
\item\label{it:U restorable 2}
If~\(\sB\) is a use block for \(\Delta^{E\oplus \dots}(n)=U(n)\) for some
\(n>\tp(\beta)\) and \(\Delta\in\mt(\alpha,U)\) (so~\(\alpha\) is an
\(R\)\nbd node), then either
\begin{enumerate}
\item\label{it:U restorable 2a}
\(\sB\) is available for correcting~\(n\) and \(\Delta^{E\oplus
\dots}(n)=U_s(n)\) (i.e., \(\Delta(n)\) is correct and hence needs no
correction), or
\item\label{it:U restorable 2b}
\(\sB\) crosses over~\(C\).
\end{enumerate}
\end{enumerate}

Suppose that~\(y_\beta\) is \((U,C)\)\nbd restorable and~\(\beta\) is an
\(R_c\)\nbd node. Then~\(y_\beta\) is \emph{weakly \(U\)\nbd restorable} if
either
\begin{itemize}
\item
the witness~\(x_\beta\) is not enumerated into~\(C\), or
\item
for some use block~\(\sB\) that~(\ref{it:U restorable 2a}) fails (hence
(\ref{it:U restorable 2b}) holds), we have that~\(\sB\) is \(C\)\nbd
restrained.
\end{itemize}
(In fact, as we will see later, if~\(x_\beta\) is allowed to be enumerated
into~\(C\), then~\(\sB\) should be \(C\)\nbd free and vice versa.)
\(y_\beta\) is \emph{\(U\)\nbd restorable} in other cases.
\end{definition}

If~\(\beta\) is an \(R_c\)\nbd node and \(U\)\nbd restorable at stage~\(s\),
then enumerating a point into~\(C\) will not injure~\(y_\beta\). Therefore we
can enumerate its diagonalizing witness~\(x_\beta\) into~\(C\) and the use
block in Case~(\ref{it:U restorable 2b}) in Definition~\ref{def:U restorable}
can help correcting~\(\Delta(n)\). Hence,~\(y_\beta\) can be restored at the
beginning of~\(s\) and will not be injured by the end of stage~\(s\).

In Example~\ref{eg:1U good controller}, we showed that if
\(\diff(U,y_4,s,t)\), then~\(y_2\) is \(U\)\nbd restorable at~\(t\); if
\(\same(U,y_4,s,t)\), then~\( y_4\) is \(U\)\nbd restorable at~\(t\). In
Example~\ref{eg:1U bad controller}, if \(\same(U,y_3,s_*,s)\), then~\(y_3\)
is \(U\)\nbd restorable at~\(s\); if \(\diff(U,y_3,s_*,s)\), then~\(y_2\) is
\((U,C_\lambda)\)\nbd restorable (in fact, no use block \(\sB<y_2\)
satisfies~\eqref{it:U restorable 2} in Definition~\ref{def:U restorable}) and
weakly \(U\)\nbd restorable at~\(s\) as the witness~\(x_2\) is not
enumerated.

\subsection{\texorpdfstring{The \(R\)-strategy and the controller, %
part~2}{The R-strategy and the controller, part 2}}%
\label{sec:R and controller 2}

In Example~\ref{eg:1U bad controller}, it was shown that each
controller~\(\beta\) has some action to take unless it is a \(U^0\)\nbd
controller and~\(\cD(\beta)\) is a \(U^0\)\nbd problem
(Definition~\ref{def:1U controller}\eqref{it:1U controller2}). In this
section, we discuss this situation.

\begin{example}\label{eg:1U way out}
We continue with Example~\ref{eg:1U bad controller}. At
stage~\(s_{**}\),~\(R_{c_0}^1\) becomes a \(U^0\)\nbd controller
and~\(R_{c_\lambda}^0\) is a \(U^0\)\nbd problem (Definition~\ref{def:1U
controller}). Suppose that at \(s>s_{**}\) we have
\(\cD_s(R_{c_0}^1)=R_{c_\lambda}^0\). Notice that we also have
\(\cD_s(R_{c_0}^3)=R_{c_\lambda}^2\), a \(U^1\)\nbd problem for the
\(U^1\)\nbd controller~\(R_{c_0}^3\). We have \emph{two} \(U\)\nbd problems,
and they are both \(R_{c_\lambda}\)\nbd nodes. This is not a coincidence.

Recall that this example is about embedding the \(3\)\nbd element chain
(Figure~\ref{fig:lattice4}), and if~\(\alpha\) is a \(U\)\nbd problem,
then~\(\alpha\) must be an \(R_{c_\lambda}\)\nbd node. By setting \(m=2\)
(Figure~\ref{fig:1Utree}), if we run into a \(U^0\)\nbd problem, we must also
have a \(U^1\)\nbd problem, both of which are \(R_{c_\lambda}\)\nbd nodes.
Note that~\(y_2\) and~\(y_0\) are both weakly \(U\)\nbd restorable, meaning
that if the restraint on~\(C_\lambda\) is dropped then their witnesses can be
enumerated and therefore they should become \(U\)\nbd restorable under
certain conditions. This is our plan.

At stage~\(s\), when \(\cD(R_{c_0}^1)=R_{c_\lambda}^0\), we do the following
two things and stop stage~\(s\):
\begin{enumerate}
\item
We obtain the \(U\)\nbd data \(\cE_s^U(R_{c_\lambda}^2) =
\cE_{s_*}^\varnothing(R_{c_\lambda}^2)\cup
\cE_{s_{**}}^\varnothing(R_{c_\lambda}^0)\). We add the new \(U\)\nbd
conditions for both~\(R_{c_\lambda}^2\) and~\(R_{c_\lambda}^0\) as follows:
\begin{itemize}
\item \(\Cond^U(R_{c_\lambda}^0,t)\) is \(\diff(U,y_2,s,t)\), and
\item \(\Cond^U(R_{c_\lambda}^2,t)\) is \(\same(U,y_2,s,t)\).
\end{itemize}
We also add the \(U\)\nbd reference stage~\(s\) for each \(\xi\in
\cE^U(R_{c_\lambda}^2)\). (As the notation suggests, this will be the data
that~\(R_{c_\lambda}^2\) needs in order to encounter the second outcome
of~\(R_{c_\lambda}^2\).)
\item
We establish a \emph{link}, connecting the root of the tree and the
\(\ctr\)\nbd outcome of~\(R_{c_\lambda}^2\). That is, whenever the root is
visited, we \emph{directly} encounter the \(\ctr\)\nbd outcome
of~\(R_{c_\lambda}^2\) (and~\(R_{c_\lambda}^2\) will become a controller as
we will see).
\end{enumerate}
At the beginning of stage \(s+1\), the next stage, if one of the
controllers~\(R_{c_0}^3\) and~\(R_{c_0}^1\) sees some noise, we discard the
\(U\)\nbd data \(\cE^U(R_{c_\lambda}^2)\) obtained at stage~\(s\) and we also
destroy the link, and then we wait for another stage when
\(\cD(R_{c_0}^1)=R_{c_\lambda}^0\); otherwise, we \emph{immediately} travel
the link, and~\(R_{c_\lambda}^2\) encounters the \(\ctr\)\nbd outcome with
the data \(\cE_s^U(R_{c_\lambda}^2)\). Before we let~\(R_{c_\lambda}^2\)
obtain \(\cE^{\ctr}(R_{c_\lambda}^2)\), let us first analyze the data
\(\cE_s^U(R_{c_\lambda}^2)\) and see why~\(R_{c_\lambda}^2\) is ready to
become a controller.

Firstly, recall that \(\Cond_{R_{c_0}^3}^U(R_{c_\lambda}^2,t)\) is
\(\diff(U,y_3,s_*,t)\) and that \(\Cond_{R_{c_0}^1}^U(R_{c_\lambda}^0,t)\) is
\(\diff(U,y_1,s_{**},t)\), where~\(s_*\) denotes the \(U\)\nbd reference
stage for~\(R_{c_\lambda}^2\) stored in \(\cE_{s_*}^{\ctr}(R_{c_0}^3)\),
and~\(s_{**}\) is the \(U\)\nbd reference stage for~\(R_{c_\lambda}^0\)
stored in \(\cE_{s_{**}}^{\ctr}(R_{c_0}^1)\). Let \(s_0>s_*\) be the last
stage when~\(R_{c_0}^3\) sees some noise (Definition~\ref{def:noise}), then
we have \(\bigsame(U,s_*,s_0,s)\). Note that we can assume, without loss of
generality (by assuming that each~\(y_i\) is as large as possible), that
\[
y_3<y_2<s_*<s_0<y_1<y_0<s_{**}<s.
\]
Notice that we also have \(\cD_{s_0}(R_{c_0}^3) = R_{c_\lambda}^2\), or
equivalently, \(\diff(U,y_3,s_*,s_0)\).

For a stage \(t>s\), from \(\bigsame(U,s_*,s_0,s)\) and
\(\diff(U,y_3,s_*,s_0)\) and \(y_3<y_2<s_*\), we deduce that
\[
\same(U,y_2,s,t)\Rightarrow \diff(U,y_3,s_{*},t),
\]
which implies that~\(y_2\) is weakly \(U\)\nbd restorable at stage~\(t\);
from \(\bigsame(U,s_*,s_0,s)\) (particularly from \(\same(U,y_2,s_{**},s)\))
and \(y_2<y_1\), we deduce that
\[
\diff(U,y_2,s,t)\Rightarrow \diff(U,y_1,s_{**},t),
\]
which implies that~\(y_0\) is weakly \(U\)\nbd restorable at stage~\(t\).

Now, as we are encountering and then visiting the \(\ctr\)\nbd outcome, which
is to the left of both of the controllers~\(R_{c_0}^3\) and~\(R_{c_0}^1\), we
are safe to initialize both of them. More importantly, the restraint
on~\(C_\lambda\) is dropped. Therefore when~\(R_{c_\lambda}^2\) becomes a
controller at \(s_{R_{c_\lambda}^2}^{\ctr} = s+1\), we are allowed to
enumerate~\(x_2\) and~\(x_0\), the diagonalizing witnesses for~\(y_2\)
and~\(y_0\), respectively, into~\(C_\lambda\). Being weakly \(U\)\nbd
restorable is now being \(U\)\nbd restorable. By the way, we also put a
restraint on \(C_0\res s_{R_{c_\lambda}^2} \).
\end{example}

We remark that in the above example, visiting the second outcome
of~\(R_{c_\lambda}^2\) requires a lot of work --- we have to obtain
\(\cE^U(R_{c_\lambda}^2)\) in a very time-consuming way. In this example, the
next outcome is a \(\ctr\)\nbd outcome, and we are lucky
that~\(R_{c_\lambda}^2\) can immediately become a controller and the
\(U\)\nbd data \(\cE_s^U(R_{c_\lambda}^2)\) is not wasted. However, this will
not generally be true in Section~\ref{sec:2U}. The \(U\)\nbd data can be
wasted (in the same manner that \(\varnothing\)\nbd data can be wasted) and
both of the previous controllers are initialized. It seems that we have
gained nothing, but encountering the second outcome one more time is a bit of
progress.

Now we make the procedure of obtaining \(\cE^U(R_{c_\lambda}^2)\) as in
Example~\ref{eg:1U way out} formal. We set \(m=\abs{\Ji(\cL)}+1\) to keep the
argument simple.

\begin{definition}[strong \(U\)\nbd data]\label{def:1U strong data}
Suppose \(m=\abs{\Ji(\cL)}+1\). Suppose that each~\(\alpha_i\), \(i<m\), is a
\(U^i\)\nbd problem for the controller~\(\beta_i\), where \(\alpha_0
\subseteq \cdots \subseteq \alpha_{m-1}\). For each~\(\alpha_i\), its
computation is~\(y_{\alpha_i}\), and \(\Cond_{\beta_i}^U(\alpha_i,s)\) is
\(\diff(U,z_{\alpha_i},s_i,s)\), where~\(z_{\alpha_i}\) is the \(U\)\nbd
reference length (in fact, \(z_{\alpha_i}=y_{\beta_i}\) in this section)
and~\(s_i\) is the \(U\)\nbd reference stage (in fact,
\(s_i=s_{\beta_i}^{\ctr}\) in this section). Let~\(s\) be the current stage when
\(\cD_{s}(\beta_0)=\alpha_0\neq\cD_{s-1}(\beta_0)\). By the Pigeonhole
Principle, we have for some \(0\le i<j<m\) and some \(c\in \Ji(\cL)\) such
that both~\(\alpha_i\) and~\(\alpha_j\) are \(R_c\)\nbd nodes.

The \(U\)\nbd data \(\cE_s^U(\alpha_j)\) consists of the following:
\begin{enumerate}
\item
a set of nodes \(\cE_s^U(\alpha_j) = \cE^\varnothing(\alpha_j)\cup
\cE^\varnothing(\alpha_i)\), where \(\cE^\varnothing(\alpha_i)\) belongs to
\(\cE^{\ctr}(\beta_i)\) and \(\cE^\varnothing(\alpha_j)\) belongs to
\(\cE^{\ctr}(\beta_j)\) (the subscripts of \(\cE^\varnothing(\alpha_i)\)
and \(\cE^\varnothing(\alpha_j)\) can be deduced from
\(\cE^{\ctr}(\beta_i)\) and \(\cE^{\ctr}(\beta_j)\), respectively, and
hence can be omitted),
\item
for each \(\xi\in \cE^\varnothing(\alpha_j)\), a \(U\)\nbd condition
\(\Cond^U(\xi,t)=\same(U,y_\xi,s,t)\) with \emph{reference
length}~\(y_\xi\), \emph{reference stage}~\(s\), and variable~\(t\),
\item
for each \(\xi\in \cE^\varnothing(\alpha_i)\), a \(U\)\nbd condition
\(\Cond^U(\xi,t)=\diff(U,y_{\alpha_j},s,t)\) with \emph{reference
length}~\(y_{\alpha_j}\), \emph{reference stage}~\(s\), and variable~\(t\),
\end{enumerate}
This \(U\)\nbd data is \emph{strong \(U\)\nbd data}. \(U\)\nbd data that is
not strong \(U\)\nbd data (Definition~\ref{def:1U data}) is called \emph{weak
\(U\)\nbd data}.

The \emph{\(U\)\nbd link} connects the root of the priority tree and the
\(\ctr\)\nbd outcome of~\(\alpha_j\), in the sense that whenever the root is
visited, we skip the actions to \(U\)\nbd functionals and directly travel the
link and let~\(\alpha_j\) encounter the \(\ctr\)\nbd outcome with the
\(U\)\nbd data \(\cE_s^U(\alpha_j)\).
\end{definition}

Encountering the \(\ctr\)\nbd outcome will make~\(\alpha_j\) a controller
following Definition~\ref{def:1U controller}. Note that in
Definition~\ref{def:1U controller}, it does not matter whether the \(U\)\nbd
data \(\cE^U(\beta)\) is strong or not. However, if it is a strong \(U\)\nbd
data, it will not have a \(U\)\nbd problem.

As in Example~\ref{eg:1U way out}, we of course hope that
\(\Cond^U(\alpha_j,s)\) implies that~\(\alpha_j\) is restorable at
stage~\(s\) and that \(\Cond^U(\alpha_i,s)\) implies that~\(\alpha_i\) is
restorable at stage~\(s\). This will be proved in the verification section.

The following example continues Example~\ref{eg:1U way out} and demonstrates
the situation after~\(R_{c_\lambda}^2\) becomes a controller with the strong
\(U\)\nbd data
\(\cE^U(R_{c_\lambda}^2)=\{R_{c_\lambda}^2,R_{c_\lambda}^0\}\).

\begin{example}
Suppose~\(R_{c_\lambda}^2\) becomes a controller with strong \(U\)\nbd data
\(\cE^U(R_{c_\lambda}^2)=\{R_{c_\lambda}^2,R_{c_\lambda}^0\}\) and that
\(\cD(R_{c_\lambda}^2)=R_{c_\lambda}^2\). In this case, we are not going to
skip any nodes below~\(R_{c_\lambda}^0\) and directly visit
\((R_{c_\lambda}^2)^\frown d\). For example, we might have~\(R_{c_0}^1\)
visiting its \(w\)\nbd outcome for a very long time. Meanwhile,~\(R_{c_0}^7\)
becomes a controller with \(\cE^{\ctr}(R_{c_0}^7)=\{R_{c_0}^7,
R_{c_\lambda}^6\}\) in the same sense as~\(R_{c_0}^3\). Then
\(\cD(R_{c_0}^7)=R_{c_\lambda}^6\). We will turn the GREEN \(U\)\nbd outcome
of~\(R_{c_0}^5\) RED\@. Later~\(R_{c_0}^5\) becomes a controller with
\(\cE^{\ctr}(R_{c_0}^5)=\{R_{c_0}^5, R_{c_\lambda}^0\}\). Note that
this~\(R_{c_\lambda}^0\) has a new computation and a new diagonalizing
witness, which is different from the data stored in \(\cE^
{\ctr}(R_{c_0}^3)\). Then \(\cD(R_{c_0}^5)=R_{c_\lambda}^0\) and we obtain
strong \(U\)\nbd data \(\cE^U(R_{c_\lambda}^6) =
\{R_{c_\lambda}^6,R_{c_\lambda}^0\}\), and then~\(R_{c_\lambda}^6\) becomes a
controller. It can be the case that \(\cD (R_{c_\lambda}^2)=R_{c_\lambda}^2\)
and \(\cD(R_{c_\lambda}^6)=R_{c_\lambda}^0\). Whenever~\(R_{c_\lambda}^2\)
sees some noise, all nodes to the right of \((R_{c_\lambda}^2)^\frown
{\ctr}\) are initialized, including the controller~\(R_{c_\lambda}^6\).
\end{example}

One last remark is that~(\ref{it:U restorable 2b}) in Definition~\ref{def:U
restorable} can actually happen. However, we have to consider a more
complicated lattice, for example, the diamond lattice in
Figure~(\ref{fig:lattice4}) or the \(6\)\nbd element lattice in
Figure~(\ref{fig:lattice6}). For this reason, we let the readers sort out the
details, and we will see in the verification that if~(\ref{it:U restorable
2a}) fails, then~(\ref{it:U restorable 2b}) must hold.

\subsection{\texorpdfstring{The \(G\)-strategy}{The G-strategy}}%
\label{sec:G strategy}

Recall that we have a global requirement
\[
G: K = \Theta^{j(1)},
\]
where we assume \(j(1) = E \oplus C_0 \oplus C_1 \oplus \cdots
C_{\abs{\Ji(\cL)}-1}\) where \(\{c_0, c_1,\ldots,c_{\abs{\Ji(\cL)}-1}\} =
\Ji(\cL)\).~\(\Theta(x)\) is always defined with fresh large use
\(\theta(x)+1\) the first time, which never changes. (In order to prevent
coding by this requirement to interfere with witnesses and killing points, we
agree that no use block for a \(U\)\nbd functional will be allowed to contain
any \(\Theta\)\nbd uses.)

Now, when~\(x\) is enumerated into~\(K\) at stage \(s\), we choose~\(C_k\)
for some~\(k\) and simply enumerate \(a=\theta(x)\) into it. The correct
set~\(C_k\) will be denoted by~\(\chi_s(a)\), and the following describes how
we decide it:

List all controllers in decreasing order of priority as
\[
\beta_0, \beta_1, \dots, \beta_{n-1}.
\]
(Recall that~\(\beta_i\) has higher priority than~\(\beta_j\) if
\(\beta_i^\frown \ctr <_P \beta_j^\frown \ctr\).) We assume that
each~\(\beta_i\) is an \(R_{c_{\beta_i}}\)\nbd node for some \(c_{\beta_i}
\in \Ji(\cL)\), so it is restraining the set \(\hat{C}\res
s_{\beta_i}^{\ctr}\) for each \(\hat{c}\neq c_{\beta_i}\). If \(i<j\),
then~\(\beta_j\) becomes a controller after~\(\beta_i \), so
\(s_{\beta_i}^{\ctr} < s_{\beta_j}^{\ctr} \). Now, we let~\(\beta_i\) be the
controller of highest priority, if any, such that \(a < s_{\beta_i}^{\ctr}\).
If such \(\beta_i\) exists, we let \(\chi_s(a)=C_{\beta_i}\); otherwise, we
let \(\chi_s(a)=E\).

We remark that if \(\chi_s(\theta(x))=E\), then enumerating \(\theta(x)\)
into~\(E\) does not affect any controller. If
\(\chi_s(\theta(x))=C_{\beta_i}\), then enumerating \(\theta(x)\)
into~\(C_{\beta_i}\) does not affect \(\beta_0,\beta_1,\dots,\beta_i\) since
\(\theta(x)\) is relatively large for \(\beta_0,\dots,\beta_{i-1}\) and
\(\beta_i\) has no restraint on~\(C_{\beta_i}\). The controllers
\(\beta_{i+1},\dots,\beta_{n-1}\) will be simply initialized. In fact, as
\(x<\theta(x)<s_{\beta_i}^{\ctr}\), this controller~\(\beta_i\) sees some
noise at stage~\(s\) (Definition~\ref{def:noise}) and hence all nodes,
including~\(\beta_j\) with \(j>i\), to the right of \(\beta_i^\frown {\ctr}\)
are initialized in the first place.

\subsection{The threshold point and diagonalizing witness}%
\label{sec:threshold point and diagonalizing witness}

Let~\(\beta\) be an \(R_c\)\nbd node for some \(c\in \Ji(\cL)\). The
threshold point is denoted by~\(\tp(\beta)\), and the diagonalizing witness
is denoted by~\(\dw(\beta)\) (see Section~\ref{sec:1Utree}). As usual, we
should define \(\dw(\beta)>\tp(\beta)\). Note that~\(\tp(\beta)\) is
associated to~\(\beta\) and is undefined only when~\(\beta\) is
initialized.~\(\dw(\beta)\) is associated with the \(w\)\nbd outcome
of~\(\beta\) and will become undefined whenever a node/outcome (for example,
the \(U\)\nbd outcome of~\(\beta\)) to the left of it is visited. Of course,
the next time we visit~\(\beta\) and~\(\dw(\beta)\) becomes undefined, we
define it to be a fresh number. That being said, each time~\(\beta\) visits
its \(U\)\nbd outcome, it has a different computation with a different
witness. Suppose the \(w\)\nbd outcome is the \emph{true outcome} (the
leftmost outcome of~\(\beta\) that is visited infinitely often);
then~\(\dw(\beta)\) will become stable.

Let~\(\beta\) be a controller with~\(\beta^*\) as its \(U^a\)\nbd
problem~(where \(a>0\)). Note that the \(w\)\nbd outcome of an \(R\)\nbd node
\(\alpha\subseteq\beta\) is to the right of \(\beta^\frown \ctr\) and
therefore~\(\dw(\alpha)\) becomes undefined at \(s_\beta^{\ctr}\). In
particular, such~\(\alpha\) will pick their diagonalizing witnesses larger
than \(s_\beta^{\ctr}\) next time, and so they are free to enumerate them
without worrying about the restraint set by~\(\beta\).

Consider Figure~\ref{fig:1Utree} with the \(w\)\nbd outcome of~\(R_{c_0}^3\)
changed to the \(d\)\nbd outcome. Suppose that~\(R_{c_0}^3\) becomes a
controller at~\(s_*\) but at each \(s>s_*\), we have
\(\cD_s(R_{c_0}^3)=R_{c_0}^3\). The previous diagonalizing witness \(x_2 =
\dw_{s_*}(R_{c_\lambda}^2)\) is not enumerated yet but it is prepared
by~\(R_{c_\lambda}^2\) to become a controller (as in Example~\ref{eg:1U way
out}). We should avoid using~\(x_2\) in other places. By our
convention,~\(R_{c_\lambda}^2\) should pick a new diagonalizing witness
\(x_2'=\dw_s(R_{c_\lambda}^2)>s_*\) next time, and
perhaps~\(R_{c_\lambda}^4\) becomes a controller at~\(s\). In this situation,
we have both \(\cE_{s_*}^\varnothing(R_{c_\lambda}^2)\subseteq
\cE^{\ctr}(R_{c_0}^3)\) and \(\cE_{s}^\varnothing(R_{c_\lambda}^2)\subseteq
\cE^{\ctr}(R_{c_\lambda}^4)\), each of which has a distinct diagonalizing
witness with a corresponding computation.

We have the usual conflicts between a threshold point and a computation. In
our construction, all \(\varnothing\)\nbd data will be discarded by the end
of each stage unless there is a controller~\(\beta\) collecting them. Suppose
that \(\beta\) is a controller with \(\cE^{\ctr}(\beta)= \{\beta,\alpha\}\)
with computations~\(y_\alpha\) and~\(y_\beta\). Let \(k=\tp(\beta)\ge
\tp(\alpha)\) (we can assume \(\tp(\alpha)\le \tp(\beta)\) if
\(\alpha\subseteq \beta\)). Whenever there is a set \(X\) relevant to
\(\beta\) such that \(\diff(X,k,s-1,s)\), we initialize \(\beta^\frown
\ctr\), that is, we discard \(\cE^{\ctr}(\beta)\) and~\(\beta\) is no longer
a controller. Note that we do not directly initialize an \(R\)\nbd node, so
\(\tp(\beta)\) and \(\tp(\alpha)\) remain defined. Therefore such an
initialization to the \(\ctr\)\nbd outcome happens only finitely often to a
fixed controller.

\subsection{The construction}\label{sec:1U construction}

We can initialize not only a node but also an outcome. As we will always have
that the \(o\)\nbd outcome of~\(\alpha\) is initialized iff the node
\(\alpha^\frown o\) is initialized, we simply write \(\alpha^\frown o\) for
both events.

\begin{definition}[initialization]
An \(S\)\nbd node~\(\alpha\) is \emph{initialized} by canceling all
functionals are defined by~\(\alpha\). An \(R\)\nbd node~\(\alpha\) is
\emph{initialized} by canceling~\(\tp(\alpha)\), all parameters stored at
each outcome of~\(\alpha\), and all functionals (if any) that are defined
by~\(\alpha\). \(\alpha^\frown w\) is \emph{initialized} by canceling
\(\dw(\alpha)\). \(\alpha^\frown d\) is \emph{initialized} by making it
inactive. \(\alpha^\frown U\) is \emph{initialized} by canceling the
\(\Delta\)\nbd functional (if any) that belongs to \(\mt(\alpha,U)\).
\(\alpha^\frown \ctr\) is \emph{initialized} by discarding
\(\cE^{\ctr}(\alpha)\) and making~\(\alpha\) no longer a controller.

If~\(\alpha\) is initialized, then we also tacitly initialize all outcomes
and nodes to the right of~\(\alpha\).
\end{definition}

The following is a special case of Definition~\ref{def:noise}.

\begin{definition}[threats]\label{def:threats}
Let~\(\beta\) be a controller with \(\cE^{\ctr}(\beta)\). At stage
\(s>s_\beta^{\ctr}\) (the stage at which~\(\beta\) becomes a controller), if
there is some~\(X\) that is relevant (Definition~\ref{def:relevant}) to
\(\beta\) such that
\[
\diff(X,\tp(\beta),s-1,s),
\]
then~\(\beta\) \emph{sees some threats}. (We are assuming that
\(\alpha\subseteq \beta\) implies \(\tp(\alpha)\le \tp(\beta)\).)
\end{definition}

\subsubsection*{Construction}
At stage~\(s\), we first run the controller strategy (see below) and then the
\(G\)\nbd
strategy (see below). Then we perform \(\visit(\lambda)\) (see below),
where~\(\lambda\) is the root
of the priority tree~\(\cT\). We stop the current stage whenever we perform
\(\visit(\alpha)\) for some~\(\alpha\) with \(\abs{\alpha}=s\).

\subsubsection*{\(\visit(\alpha)\) for an \(S\)-node:}

Suppose that there is some \(U\)\nbd link connecting~\(\alpha\) and
\(\beta^\frown o\) for some \(o\)\nbd outcome of~\(\beta\), then we perform
\(\enc(\beta,o)\).

Suppose that there is no link. For each \(\Gamma^{E\oplus U}=C\) (for some
\(c\in \Ji(\cL)\)) that belongs to \(\mt(\alpha,U)\) and for each \(x\le s\),
\begin{enumerate}
\item
Suppose \(\Gamma_s^{E\oplus U}(x)\downarrow=C_s(x)\).~\(\beta\) does
nothing else.
\item
Suppose \(\Gamma_s^{E\oplus U}(x)\downarrow\neq C_s(x)\) with use block
\(\sB=\bB_s(\gamma,x)\).
\begin{enumerate}
\item
If~\(\sB\) is killed and not \(E\)\nbd restrained, then~\(\beta\)
enumerates an unused point, referred to as a \emph{killing point},
into~\(\sB\) via~\(E\). Then we go to~(3) immediately.
\item
If~\(\sB\) is not killed and not \(E\)\nbd restrained, then~\(\beta\)
enumerates an unused point, referred to as a \emph{correcting point},
into~\(\sB\) via~\(E\). Then we redefine \(\Gamma_s^{E\oplus U}(x) =
C_s(x)\) with the \emph{same} use block~\(\sB\).
\end{enumerate}
\item
Suppose \(\Gamma_s^{E\oplus U}(x)\uparrow\). If each
\(\bB_{\bday{t}}(\gamma,x)\) with \(t<s\) has been \emph{killed} (see
Section~\ref{sec:R and killing}), then~\(\beta\) picks a fresh use block
\(\sB'\) and defines \(\Gamma_s^{E\oplus U}(x) = C_s(x)\) with use block
\(\sB'\) (hence \(\sB'=\bB_{\bday{s}}(\gamma,x)\)); otherwise, we define
\(\Gamma_s^{E\oplus U}(x)=C_s(x)\) with the use block that is not killed
(there will be at most one such use block).
\end{enumerate}

Then we stop the current substage and perform \(\visit(\alpha^\frown 0)\) for
the \(R\)\nbd node \(\alpha^\frown 0\).

\subsubsection*{\(\visit(\alpha)\) for an \(R\)-node:}

If~\(\tp(\alpha)\) is not defined, we define it with a fresh number. Then we
perform \(\enc(\alpha,d)\).

Without loss of generality, we assume that~\(\alpha\) is assigned an
\(R_c(\Phi)\)\nbd requirement for some \(c\in \Ji(\cL)\) and~\(\Phi\).

\subsubsection*{\(\enc(\alpha,d)\):}

If~\(d\) is active, then we perform \(\visit(\alpha^\frown d)\). If~\(d\) is
inactive, we perform \(\enc(\alpha,w)\).

\subsubsection*{\(\enc(\alpha,w)\):}
If \(\dw(\alpha)\) is not defined, then we pick a fresh number
\(x>\tp(\alpha)\) and define \(\dw(\alpha)=x\).
\begin{enumerate}
\item
If a computation~\(y\) is found by~\(\alpha\) with slowdown condition
(Definition~\ref{def:sd}), we obtain \(\varnothing\)\nbd data
\(\cE_s^\varnothing(\alpha)\) (Definition~\ref{def:0 data}) and perform
\(\enc(\alpha,U)\), where~\(U\) is the first outcome (recall from
Definition~\ref{def:1Utree} that we add outcomes in order).
\item
If no computation is found, then we perform \(\visit(\alpha\frown w)\).
\end{enumerate}

\subsubsection*{\(\enc(\alpha,U)\):}

Notice that we must have obtained \(\cE^\varnothing(\alpha)\).

\begin{enumerate}
\item
If the \(U\)\nbd outcome is Type~I, then let \(v=\tp(\alpha)\). For each
functional~\(\Gamma\) (\(\Delta\) is dealt with similarly) that belongs to
\(\kl(\alpha,U)\) and for each~\(x\) with \(v\le x\le s\), let
\(\sB_x=\bB_s(\gamma,x)\) be the use block. We enumerate an unused point
(\emph{killing point}) into~\(\sB_x\) and say~\(\sB_x\) is \emph{killed}.

Let~\(\Delta\) belong to \(\mt(\alpha,U)\). Without loss of generality, we
assume that this functional is to ensure \(\Delta^{E\oplus C_0\oplus \cdots
\oplus C_{k-1}} = U\) (allowing for \(k=0\), i.e., that there are
no~\(C_i\)). For each \(x\le s\), we do the following:
\begin{enumerate}
\item
Suppose that \(\Delta_s^{E\oplus C_0\oplus \cdots \oplus
C_{k-1}}(x)\downarrow=U_s(x)\). Then~\(\beta\) does nothing else.
\item
Suppose we have \(\Delta_s^{E\oplus C_0\oplus \cdots \oplus
C_{k-1}}(x)\downarrow\neq U_s(x)\) with use block
\(\sB=\bB_s(\delta,x)\).
\begin{enumerate}
\item
If~\(\sB\) is killed and \(E\)\nbd free, then~\(\beta\) enumerates an
unused point, referred to as a \emph{killing point}, into~\(\sB\)
via~\(E\). Then we go to~(3) immediately.
\item
If~\(\sB\) is killed and \(E\)\nbd restrained, we let~\(C_i\) (\(i<k\))
be the set such that~\(\sB\) is \(C_i\)\nbd free (we will show such
a~\(C_i\) exists); then~\(\beta\) enumerates an unused point, referred
to as a \emph{killing point}, into~\(\sB\) via~\(C_i\).~\(\sB\) is then
\emph{permanently killed} (as~\(C_i\) will be a c.e.\ set). Then we go
to~(3) immediately.
\item
If~\(\sB\) is not killed and \(E\)\nbd free, then~\(\beta\) enumerates
an unused point, referred to as a \emph{correcting point}, into~\(\sB\)
via~\(E\). Then we define \(\Delta_s^{E\oplus C_0\oplus \cdots \oplus
C_{k-1}}(x)=U_s(x)\) with the same use block~\(\sB\).
\item
If~\(\sB\) is not killed and \(E\)\nbd restrained, we let~\(C_i\) for
some \(i<k\) be the set such that~\(\sB\) is \(C_i\)\nbd free (we will
show such a~\(C_i\) exists); then~\(\beta\) enumerates an unused point,
referred to as a \emph{correcting point}, into~\(\sB\) via~\(C_i\).
Then we define \(\Delta_s^{E\oplus C_0\oplus \cdots \oplus
C_{k-1}}(x)=U_s(x)\) with the same use block~\(\sB\).
\end{enumerate}
\item
Suppose that \(\Delta_s^{E\oplus C_0\oplus \cdots \oplus
C_{k-1}}(x)\uparrow\). If for each \(t<s\), \(\bB_{\bday{t}}(\delta,x)\)
is killed,~\(\beta\) chooses a fresh use block~\(\sB'\) and defines
\(\Delta_s^{E\oplus C_0\oplus \cdots \oplus C_{k-1}}(x) = U_s(x)\) with
use block~\(\sB'\) (hence \(\sB'=\bB_{\bday{s}}(\delta,n)\)); otherwise,
we will define \(\Delta_s^{E\oplus C_0\oplus \cdots \oplus C_{k-1}}(x) =
U_s(x)\) with the use block that is not killed (there will be at most one
such use block).
\end{enumerate}

Then we stop the current substage and perform \(\visit(\alpha^\frown U)\) for
the \(S\)\nbd node \(\alpha^\frown U\).

\item
If the \(U\)\nbd outcome is GREEN\@, then let \(v=\tp(\alpha)\). For each
functional~\(\Gamma\) (\(\Delta\) is dealt with similarly) that belongs to
\(\kl(\alpha,U)\) and for each~\(x\) with \(v\le x\le s\), let
\(\sB_x=\bB_s(\gamma,x)\) be the use block. We enumerate an unused point
(\emph{killing point}) into~\(\sB_x\) and say~\(\sB_x\) is \emph{killed}.
Then we stop the current substage and perform \(\visit(\alpha^\frown U)\) for
the \(S\)\nbd node \(\alpha^\frown U\).
\item
If the \(U\)\nbd outcome is RED\@, then we obtain the weak \(U\)\nbd data
\(\cE_s^U(\alpha)\) (Definition~\ref{def:1U data}) and perform
\(\enc(\alpha,\ctr)\).
\end{enumerate}

\subsubsection*{\(\enc(\alpha,\ctr)\):}

Notice that we must have obtained~\(\cE^U(\alpha)\). Suppose that~\(\alpha\)
is an \(R_c\)\nbd node for some \(c\in \Ji(\cL)\).
\begin{enumerate}
\item
Let \(\cE^{\ctr}(\alpha) = \cE^U(\alpha)\) (Definition~\ref{def:1U
controller}), and let~\(\alpha\) become a controller.
\item
We enumerate the diagonalizing witness for each \(\xi\in
\cE^{\ctr}(\alpha)\) into the set~\(C\) if~\(\xi\) is not a \(U\)\nbd
problem (see Definition~\ref{def:1U controller}).
\item
While~\(\alpha\) is a controller, we put a restraint on \(\hat{C}\res
s_\alpha^{\ctr}\) for each \(\hat{c}\neq c\).
\end{enumerate}
We then stop the current stage.

\subsubsection*{controller-strategy:}

Let~\(\beta\) (if any) be a controller of highest priority such
that~\(\beta\) sees some noise (Definition~\ref{def:noise}). We initialize
all nodes to the right of \(\beta^\frown \ctr\). Suppose that~\(\beta\) is an
\(R_c\)\nbd node, \(\seq_0(\beta)=b\), and
\(\cE^{\ctr}(\beta)=\cE^U(\beta)\).
\begin{enumerate}
\item
If~\(\beta\) sees also some threats (Definition~\ref{def:threats}), then we
also initialize \(\beta^\frown \ctr\).
\item
If~\(\beta\) does not changes its decision (Definition~\ref{def:1U
decision}), then we do nothing.
\item
If~\(\beta\) changes its decision and \(\cD_s(\beta)=\xi\), then we set
\(E\res y_\xi\) to be restrained (so each use block \(\sB<y_\xi\) is
\(E\)\nbd restrained) until the next time~\(\beta\) changes its decision.
Furthermore:
\begin{enumerate}
\item
If~\(\xi\) is not a \(U^b\)\nbd problem of~\(\beta\), then we
restore~\(y_\xi\) and let \(\xi^\frown d\) be active.
\item
If~\(\xi\) is a \(U^b\)\nbd problem of~\(\beta\) and \(b>0\), we
restore~\(y_\xi\) and turn the GREEN \(U\)\nbd outcome
of~\(\xi^{\sharp}\) into a RED \(U\)\nbd outcome (once~\(\beta\) changes
its decision or is initialized, it turns back to GREEN).
\item\label{it:1U controller 3c}
If~\(\xi\) is a \(U^0\)\nbd problem, then we restore~\(y_\xi\) and pick
(as per Lemma~\ref{lem:1U pair}) \(\alpha_i\subsetneq \alpha_j\) such
that they are \(U\)\nbd problems and are both \(R_d\)\nbd nodes for some
\(d\in \Ji(\cL)\). We obtain the strong \(U\)\nbd data
\(\cE^U(\alpha_j)\) (Definition~\ref{def:1U strong data}) and establish a
\(U\)\nbd link connecting the root of the priority tree and the
\(\ctr\)\nbd outcome of~\(\alpha_j\) (the \(U\)\nbd link will be
destroyed once traveled or~\(\beta\) changes its decision).
\end{enumerate}
\end{enumerate}

\subsubsection*{\(G\)-strategy:}

Suppose \(K=\Theta^{j(1)}=\Theta^{E\oplus C_0\oplus \cdots \oplus
C_{\abs{\cL}-1}}\) where \(\Ji(\cL)=\{c_0,\ldots, c_{\abs{\cL}-1}\}\). For
each \(x\le s\),
\begin{enumerate}
\item\label{it:G strategy 1}
if \(\Theta^{j(1)}(x)\) has never been defined (so \(x=s\)), we define
\(\Theta_s^{j(1)}(x)=K_s(x)\) with a fresh use \(\theta(x)+1\) (which never
changes).
\item\label{it:G strategy 2}
If \(\Theta^{j(1)}(x)\downarrow\neq K_s(x)\), we enumerate \(\theta(x)\)
into the set \(\chi_s(\theta(x))\) (see Section~\ref{sec:G strategy}). Then
we go to~(3) immediately.
\item\label{it:G strategy 3}
If \(\Theta^{j(1)}(x)\uparrow\), we define \(\Theta_s^{j(1)}(x)=K_s(x)\)
with the same use \(\theta(x)+1\).
\end{enumerate}

We remark that we did not explicitly mention yet how big a use block should
be to avoid being distracted by this technical issue (see
Lemma~\ref{lem:block size}).

\subsection{The verification}\label{sec:1U verification}

First of all, one has to show that the use block is sufficiently large and
also justify the controller strategy~(\ref{it:1U controller 3c}) so that the
construction will not terminate unexpectedly.

\begin{lemma}[Block size]\label{lem:block size}
Each use block can be chosen sufficiently large.
\end{lemma}

\begin{proof}
Let \(\sB=[a,b)\) be a use block. Such~\(\sB\) can interact with a
controller~\(\beta\) with \(\cE^{\ctr}(\beta)=\{\beta,\alpha\}\) in the
following way: If \(\cD(\beta)=\alpha\) with \(y_\alpha>\sB\), we will
possibly extract a point from~\(\sB\). In this case, we say that~\(\sB\) is
\emph{injured}. If \(\cD(\beta)=\beta\) with \(y_\beta<\sB\), we will
possibly enumerate a point into \(\sB\) when the node which maintains~\(\sB\)
is visited. In this case, we say that~\(\sB\) is \emph{restored}. We only
have to consider each use block~\(\sB\) with \(y_\beta<\sB\) (and
\(\sB<y_\alpha\)) as otherwise restoring either~\(y_\beta\) or~\(y_\alpha\)
makes no changes to~\(\sB\). Therefore, we need to count how many
times~\(\sB\) can potentially be injured and then restored.

According to Definition~\ref{def:1U decision}, \(\cD_s(\beta) = \xi\) for the
longest~\(\xi\) such that \(\Cond^U(\xi,s)\) and this is determined by
\(U_s\res y_\beta\) (where \(y_\beta<a\)). Therefore the number of times that
\(\sB=[a,b)\) can be injured and then restored depends on the number of
changes that \(U\res a\) can have, i.e., the size of
\[
S=\{s\mid \diff(U,a,s-1,s)\}.
\]
It is clear that the number of~\(S\) can be bounded by a computable
function~\(p(a)\) since~\(U\) is a d.c.e.\ set.

Therefore, when we define \(\Gamma(x)\) with a fresh use block, we pick a
fresh number~\(a\) and let the use block be \([a,a+p(a))\). Defined in this
way, a use block is sufficiently large.

It remains to show that a use block cannot interact with two controllers:
Suppose that~\(\sB\) interacts with~\(\beta\) with
\(\cE^{\ctr}(\beta)=\{\beta,\alpha\}\). Then a controller~\(\beta'\) of lower
priority believes that~\(B\) never changes again (if it changes it,
then~\(\beta\) sees some noise and~\(\beta'\) is initialized). For a
controller~\(\beta''\) of higher priority, we can assume that
\(s_{\beta''}^{\ctr}<y_\beta<\sB\), so~\(\sB\) does not interact
with~\(\beta''\).
\end{proof}

\begin{lemma}\label{lem:1U pair}
Let~\(\beta\) be a controller such that \(\cD_s(\beta)=\xi\) where~\(\xi\) is
a \(U^0\)\nbd problem, then there exists \(\alpha_i\subsetneq \alpha_j\) both
of which are \(R_d\)\nbd nodes for some \(d\in \Ji(\cL)\).
\end{lemma}

\begin{proof}
Let \(\beta_0=\beta\). This \(\cE^U(\beta)\) must be weak \(U\)\nbd data as
strong \(U\)\nbd data has no \(U\)\nbd problem. Therefore the \(U\)\nbd
outcome of \(\beta\) must be turned RED by another (unique)
controller~\(\beta_1\) with \(\alpha_1=\cD(\beta_1)\) as a \(U\)\nbd problem
of~\(\beta_1\). Continuing this fashion, we find
\[
\xi=\alpha_0\subsetneq \alpha_1\subsetneq \cdots \subsetneq \alpha_{m-1}
\]
where each~\(\alpha_i\) is a \(U^i\)\nbd problem of some
controller~\(\beta_i\). Note that we assume \(m=\abs{\Ji(\cL)}+1\). By the
Pigeonhole Principle, there are \(\alpha_i\), \(\alpha_j\), and some \(d\in
\Ji(\cL)\) such that~\(\alpha_i\) and~\(\alpha_j\) are both \(R_d\)\nbd
nodes.
\end{proof}

Considering weak \(U\)\nbd data, we are going to put each use block into one
of several categories.

\begin{definition}
Let \(\cE^U(\beta)=\{\beta,\beta^*\}\) be weak \(U\)\nbd data
(Definition~\ref{def:1U data}) obtained at stage~\(s\). Let
\(\cA_1=\{\eta\mid \eta\subsetneq \beta^*\}\), \(\cA_2=\{\eta \mid
\beta^*\subseteq \eta\subsetneq \beta\}\), and \(\cA_3=\{\eta\mid
\beta\subseteq \eta\}\). For a use block~\(\sB\) that is maintained by a node
in~\(\cA_i\) and killed by a node in~\(\cA_j\) at stage~\(s\), we define
\(\cQ_{\cE^U (\beta)}^U(\sB)=(i,j)\) (for \(i\le j\)); if~\(\sB\) is not
killed, then \(\cQ_{\cE^U(\beta)}^U(\sB)=(i,\infty)\).
\end{definition}

We write~\(\cQ\) for \(\cQ_{\cE^U(\beta)}^U\) if there is no confusion.

To tell whether a computation \(y\) is restorable or not, we only care about
those blocks~\(\sB\) with \(\sB<y\). The slowdown conditions
(Definition~\ref{def:sd}) allow us to exclude some of the blocks from
consideration:

\begin{lemma}\label{lem:relevant}
Let \(\cE^U(\beta)=\{\beta,\beta^*\}\) be weak \(U\)\nbd data
(Definition~\ref{def:1U data}) obtained at stage~\(s\). Suppose
that~\(\beta\) is an \(R_c\)\nbd node and~\(\beta^*\) is an \(R_d\)\nbd node
for some \(d\le c\in \Ji(\cL)\). Given a use block~\(\sB\), if
\(\cQ(\sB)=(1,1)\), then \(\sB>y_{\beta^*}\); if \(\cQ(\sB)=(i,j)\) with
\(i\le j\in \{1,2\}\), then \(\sB>y_\beta\). \qed
\end{lemma}

Therefore, to tell whether \(y_{\beta^*}\) is \((U,D)\)\nbd restorable, we
consider only those blocks \(\sB\) with \(\cQ(\sB)=(1,j), j\in
\{2,3,\infty\}\) (and \(\sB<y_{\beta^*}\)); to tell whether \(y_\beta\) is
\((U,C)\)-restorable, we consider only those blocks \(\sB\) with
\(\cQ(\sB)=(i,j), i\in \{1,2\}, j\in \{3,\infty\}\) (and \(\sB<y_\beta\)).

\begin{lemma}\label{lem:structure of use block}
Let \(\cE^U(\beta)=\{\beta,\beta^*\}\) be weak \(U\)\nbd data obtained at
stage~\(s\). Suppose that~\(\beta\) is an \(R_c\)\nbd node and~\(\beta^*\) is
an \(R_d\)\nbd node for some \(d\le c\in \Ji(\cL)\). Suppose that~\(\sB\)
belongs to a \(\Gamma\)\nbd functional.
\begin{enumerate}
\item\label{it:use block 1}
If \(\cQ(\sB)=(1,j)\) with \(j\in \{3,\infty\}\), then this \(\Gamma\)\nbd
functional computes a set~\(\hat{C}\) with \(\hat{c}\ngeq d\) (hence
\(\hat{c}\ngeq c\) since \(d\le c\)).
\item\label{it:use block 2}
If \(\cQ(\sB)=(2,j)\) with \(j\in \{3,\infty\}\), then this \(\Gamma\)\nbd
functional computes a set~\(\hat{C}\) with \(\hat{c}\ngeq c\).
\end{enumerate}
Suppose that~\(\sB\) belongs to a \(\Delta\)\nbd functional.
\begin{enumerate}[resume]
\item\label{it:use block 3}
If \(\cQ(\sB)=(1,j)\) with \(j\in \{2,3,\infty\}\), then~\(\sB\) crosses
over~\(D\).
\item\label{it:use block 4}
If \(\cQ(\sB)=(1,j)\) with \(j\in \{3,\infty\}\), then~\(\sB\) crosses
over~\(C\).
\end{enumerate}
\end{lemma}

\begin{proof}
\begin{enumerate}
\item
Suppose \(\hat{c}\ge d\), \(\seq((\beta^*)^-)=(b,\xi)\) and that
\(\Gamma^{E\oplus U}=\hat{C}\) and so also \(\Gamma^{E\oplus U}=D\) belongs
to \(F_\xi(U)\). Recall Definition~\ref{def:Fsigma} and
Lemma~\ref{lem:spec1}. Setting \(c_\sigma = \hat{c}\), there is some
\(\tau\) such that \(\tau 0 \subseteq \sigma 0\) and \(c_\tau = d\). Since
\(\Gamma^{E\oplus U}=D\) belongs to \(\kl(\beta^*,U)\), we have by
Definition~\ref{def:kill} that \(\Gamma^{E\oplus U}=\hat{C}\) also belongs
to \(\kl(\beta^*,U)\), but this implies \(\cQ(\sB)=(1,2)\), a
contradiction.
\item
Suppose \(\hat{c}\ge c\) and \(\seq((\beta)^-)=(b,\xi)\). Note that
\(\Gamma^{E\oplus U}=\hat{C}\) belongs to \(F_\xi(U)\), so
\(\Gamma^{E\oplus U}=C\) also belongs to \(F_\xi(U)\) by
Lemma~\ref{lem:spec1}. Hence the \(U\)\nbd outcome should be Type~I. But
weak \(U\)\nbd data \(\cE^U(\beta)\) can only be obtained when the
\(U\)\nbd outcome is RED\@, a contradiction.
\item
By Lemma~\ref{lem:conflicts}~(\ref{it:R-plus}).
\item
By Lemma~\ref{lem:conflicts}~(\ref{it:R-minus}).
\end{enumerate}
\end{proof}

\begin{lemma}\label{lem:U restorable}
Let \(\cE^U(\beta)=\{\beta,\beta^*\}\) be weak \(U\)\nbd data obtained at
stage~\(s\). Suppose that~\(\beta\) is an \(R_c\)\nbd node and~\(\beta^*\) is
an \(R_d\)\nbd node for some \(d\le c\in \Ji(\cL)\). At each stage \(t>s\)
(independent of whether \(\cE^U(\beta)\) is discarded or not),
\begin{enumerate}
\item\label{it:lem U restorable 1}
if \(\diff(U,y_\beta,s,t)\) and \(\bigsame(\hat{D},y_{\beta^*},s,t)\) for
each \(\hat{d}\ngeq d\), then~\(y_{\beta^*}\) is \((U,D)\)-restorable
(Definition~\ref{def:U restorable}) at stage~\(t\);
\item\label{it:lem U restorable 2}
if \(\same(U,y_\beta,s,t)\) and \(\bigsame(\hat{C},y_\beta,s,t)\) for each
\(\hat{c}\ngeq c\), then~\(y_\beta\) is \((U,C)\)-restorable at
stage~\(t\).
\end{enumerate}
\end{lemma}

\begin{proof}
Let \(t>s\). We assume that we \(U\)\nbd restore~\(y_{\beta^*}\)
or~\(y_\beta\) at the beginning of the stage~\(t\).
\begin{enumerate}
\item
For~\(y_{\beta^*}\), we consider each~\(\sB\) with \(\cQ(\sB)=(1,j)\) where
\(j \in \{2,3,\infty\}\) by Lemma~\ref{lem:relevant}. If~\(\sB\) belongs to
a \(\Delta\)\nbd functional, then~\(\sB\) crosses over~\(D\) by
Lemma~\ref{lem:structure of use block}(\ref{it:use block 3}). Hence
Definition~\ref{def:U restorable}(\ref{it:U restorable 2b}) holds for this
use block~\(\sB\).

If~\(\sB\) belongs to a \(\Gamma\)\nbd functional and \(\cQ(\sB)=(1,j)\)
with \(j\in \{3,\infty\}\), then by Lemma~\ref{lem:structure of use
block}(\ref{it:use block 1}) we have that \(\Gamma^{E\oplus U}=\hat{D}\)
for some \(\hat{d}\ngeq d\). At stage~\(s\) when \(\beta^*\) found its
computation \(y_{\beta^*}\), we have \(\Gamma^{E\oplus
U}_s(x)=\hat{D}_s(x)\) if the former is defined. Then
\(\bigsame(\hat{D},y_{\beta^*},s,t)\) tells us that in particular
\(\hat{D}_s(x)=\hat{D}_t(x)\). If~\(\sB\) is available for correcting~\(x\)
at stage~\(t\) (Section~\ref{sec:use block}), as \(\Gamma(x)\) is correct,
we conclude that Definition~\ref{def:U restorable}(\ref{it:U restorable
1a}) holds for this use block~\(\sB\).

If~\(\sB\) belongs to a \(\Gamma\)\nbd functional and \(\cQ(\sB)=(1,2)\),
we let~\(s^*\) be the last stage when we visit~\(\beta\). Therefore, by the
slowdown condition of~\(\beta\), we have \(\bigsame(U,y_\beta,s^*,s)\).
Note that we also have \(s^*<\created(\sB)\le s\). Therefore, if we have
\(\diff(U,y_\beta,s,t)\), then~\(\sB\) is not available for correcting even
if we restore \(E\res \sB\) to \(y_{\beta^*}\res \sB\) at the beginning of
stage~\(t\). We conclude that Definition~\ref{def:U restorable}(\ref{it:U
restorable 1b}) holds for this use block~\(\sB\).

Hence~\(y_{\beta^*}\) is \((U,D)\)\nbd restorable.

\item
For~\(y_{\beta}\), we consider each~\(\sB\) with \(\cQ(\sB)=(i,j)\) where
\(i\in \{1,2\}, j\in \{3,\infty\}\) by Lemma~\ref{lem:relevant}.

If~\(\sB\) belongs to a \(\Gamma\)\nbd functional, we have by
Lemma~\ref{lem:structure of use block}(\ref{it:use block 1})(\ref{it:use
block 2}) and by \(\bigsame(\hat{C},y_{\beta},s,t)\) for each
\(\hat{c}\ngeq c\) that \(\Gamma^{E\oplus U}\) is correct and therefore
Definition~\ref{def:U restorable}(\ref{it:U restorable 1a}) holds for this
use block \(\sB\).

If~\(\sB\) belongs to a \(\Delta\)\nbd functional and \(\cQ(\sB)=(1,j)\)
with \(j\in \{3,\infty\}\), by Lemma~\ref{lem:structure of use
block}(\ref{it:use block 4}),~\(\sB\) crosses over~\(C\). Hence
Definition~\ref{def:U restorable}(\ref{it:U restorable 2b}) holds for this
use block.

If~\(\sB\) belongs to a \(\Delta\)\nbd functional and \(\cQ(\sB)=(2,j)\)
with \(j\in \{3,\infty\}\), then \(\same(U,y_\beta,s,t)\) says
\(\Delta_t(x)=U_s(x)=U_t(x)\), that is, \(\Delta(x)\) is correct at
stage~\(t\). Therefore Definition~\ref{def:U restorable}(\ref{it:U
restorable 2a}) holds for this use block.

Hence~\(y_\beta\) is \((U,C)\)\nbd restorable.
\end{enumerate}
\end{proof}

If~\(\beta\) becomes a controller at stage~\(s_\beta^{\ctr}\) with
\(\cE^{\ctr}(\beta)=\cE^U(\beta)=\{\beta,\beta^*\}\) and~\(\beta\) is an
\(R_c\)\nbd node for some \(c\in \Ji(\cL)\) then for each~\(\hat{C}\) with
\(\hat{c}\ngeq c\), we have \(\bigsame(U,s_\beta^{\ctr},s_\beta^{\ctr},t)\)
at each \(t>s_\beta^{\ctr}\) (while~\(\beta\) is not initialized) since we
have a restraint on~\(\hat{C}\).~\(\beta\) is \(U\)\nbd restorable
while~\(\beta^*\) might be weakly \(U\)\nbd restorable if it is an
\(R_d\)\nbd node for some \(d<c\).

\begin{lemma}\label{lem:U restorable strong}
Let \(\cE_s^U(\alpha_j)=\cE_{s_*}^\varnothing(\alpha_j)\cup
\cE_{s_{**}}^\varnothing(\alpha_i)\) be strong \(U\)\nbd data
(Definition~\ref{def:1U strong data}), where \(\alpha_i\subsetneq \alpha_j\)
are \(U\)\nbd problems of controller~\(\beta_j\) and~\(\beta_i\),
respectively. Suppose that~\(\beta_j\) and~\(\beta_i\)
are~\(R_{c_{\beta_j}}\)- and \(R_{c_{\beta_i}}\)\nbd nodes and~\(\alpha_j\)
and~\(\alpha_i\) are both \(R_d\)\nbd nodes with \(d<c_{\beta_i}\) and
\(d<c_{\beta_j}\). We let \(s_*=s_{\beta_j}^{\ctr}\),
\(s_{**}=s_{\beta_i}^{\ctr}\), \(s_0>s_*\) be the stage when
\(\cD_{s_0}(\beta_j)=\alpha_j\) (or equivalently,
\(\diff(U,y_{\beta_j},s_*,s_0)\)) and \(\bigsame(U,s_*,s_0,s)\), and
\(s_{00}\) be the stage when \(\cD_{s_{00}}(\beta_i)=\alpha_i\) (or
equivalently, \(\diff(U,y_{\beta_i},s_{**},s_{00})\)) and
\(\bigsame(U,s_{**},s_{00},s)\). Without loss of generality, we may assume
\[
y_{\beta_j}<y_{\alpha_j}<s_*<s_0<y_{\beta_i}<y_{\alpha_i}<s_{**}<s_{00}<s.
\]
Recall that \(\Cond^U(\alpha_j,t)\) is \(\same(U,y_{\alpha_j},s,t)\) and
\(\Cond^U(\alpha_i,t)\) is \(\diff(U,y_{\alpha_j},s,t)\).
\begin{enumerate}
\item\label{it:U restorable strong 1}
If \(\same(U,y_{\alpha_j},s,t)\) and \(\bigsame(\hat{D},y_{\alpha_j},s,t)\)
for each \(\hat{d}\ngeq d\), then~\(y_{\alpha_j}\) is \((U,D)\)\nbd
restorable at stage~\(t\).
\item\label{it:U restorable strong 2}
If \(\diff(U,y_{\alpha_j},s,t)\) and \(\bigsame(\hat{D},y_{\alpha_j},s,t)\)
for each \(\hat{d}\ngeq d\), then~\(y_{\alpha_i}\) is \((U,D)\)-restorable
at stage~\(t\).
\end{enumerate}
\end{lemma}

\begin{proof}
\begin{enumerate}
\item
From \(\diff(U,y_{\beta_j},s_*,s_0)\) and \(\bigsame(U,s_*,s_0,s)\), we deduce
\[
\same(U,y_{\alpha_j},s,t) \Rightarrow \diff(U,y_{\beta_j},s_*,t).
\]
Notice that from~\(s_*\), the set \(\hat{D}\res s_*\) for each
\(\hat{d}\neq c_{\beta_j}\) is restrained by the controller~\(\beta_j\).
Since \(d< c_{\beta_j}\), \(\hat{D}\res s_*\) is restrained for each
\(\hat{d}\ngeq d\). Therefore we have
\(\bigsame(\hat{D},y_{\alpha_j},s_*,t)\) for each \(\hat{d}\ngeq d\). By
Lemma~\ref{lem:U restorable}(\ref{it:lem U restorable 1}) we conclude
that~\(y_{\alpha_j}\) is \((U,D)\)\nbd restorable.

\item
From \(\diff(U,y_{\beta_i},s_{**},s_{00})\) and
\(\bigsame(U,s_{**},s_{00},s)\), we deduce
\[
\diff(U,y_{\alpha_j},s,t) \Rightarrow \diff(U,y_{\beta_i},s_{**},t).
\]
A similar argument as above shows that
\(\bigsame({\hat{D},y_{\beta_i},s_{**},t})\) holds for each \(\hat{d}\ngeq
d\). By Lemma~\ref{lem:U restorable}(\ref{it:lem U restorable 1}) we
conclude that \(y_{\alpha_i}\) is \((U,D)\)\nbd restorable.
\end{enumerate}
\end{proof}

As the sets are properly restrained by the controllers, we have the following

\begin{lemma}
Let~ \(\beta\) be a controller with
\(\cE^{\ctr}(\beta)=\cE^U(\beta)=\{\beta,\alpha\}\) where \(\cE^U(\beta)\) is
either strong or weak. Suppose that~\(\beta\) is an \(R_c\)\nbd node
and~\(\alpha\) is an \(R_d\)\nbd node with \(d\le c\). At each stage
\(s>s_\beta^{\ctr}\), if \(\cD_s(\beta)=\beta\), then~\(y_\beta\) is
\(U\)\nbd restorable; if \(\cD_s(\beta)=\alpha\), then~\(y_\alpha\) is
\(U\)\nbd restorable if \(d=c\) and weakly \(U\)\nbd restorable if \(d<c\).
\qed
\end{lemma}

Recall that if \(\cE^U(\beta)\) is strong \(U\)\nbd data, then \(d=c\). We
summarize what we have proved in the following

\begin{lemma}\label{lem:1U main lemma}
Let~\(\beta\) be a controller with \(\cE^{\ctr}(\beta)=\cE^U(\beta)\). If
\(\cD_s(\beta)=\xi\) and~\(\xi\) is \(U\)\nbd restorable, then we can
restore~\(y_\xi\) at the beginning of stage~\(s\) and activate the \(d\)\nbd
outcome of \(\xi\);~\(y_\xi\) remains restored at all substages of
stage~\(s\). \qed
\end{lemma}

\begin{lemma}[decision]\label{lem:decision}
Let~\(\beta\) be a controller with \(\cE^{\ctr}(\beta)=\{\beta,\alpha\}\)
(with \(\alpha\subsetneq \beta\)). Then for each \(s>s_\beta^{\ctr}\),
\(\cD_s(\beta)\) is defined.
\end{lemma}

\begin{proof}
Either via Definition~\ref{def:1U data} or Definition~\ref{def:1U strong
data}, we have that \(\Cond^U(\beta,s)\) iff \(\lnot\Cond^U(\alpha,s)\). Thus
\(\cD_s(\beta)\) is always defined.
\end{proof}

Given the construction, we define the \emph{true path} \(p \in [\cT]\) by
induction: We first specify \(\lambda \subseteq p\) for the root~\(\lambda\)
of~\(\cT\). Suppose \(\sigma \subseteq p\) is specified; then we say that the
\(o\)\nbd outcome of~\(\sigma\) is the \emph{true outcome} of~\(\sigma\) if
it is the leftmost outcome, if any, that is visited infinitely often, and we
specify \(\sigma^\frown o \subseteq p\). This completes the definition
of~\(p\). That~\(p\) is infinite follows from the next

\begin{lemma}[Finite Initialization Lemma]
\label{lem:1U finite initialization lemma}
Let~\(p\) be the true path. Each node \(\alpha\in p\) is initialized finitely
often, and~\(p\) is infinite.
\end{lemma}

\begin{proof}
The root of the priority is never initialized. Using induction, we consider
\(\alpha\subseteq p\) and suppose that for each \(\beta\subseteq
\alpha\),~\(\beta\) is initialized finitely often. We first show that
\(\alpha\) has a true outcome, say \(o\)\nbd outcome, then we show that
\(\alpha^\frown o\) is initialized finitely often.

Suppose that \(\alpha\) is an \(S\)\nbd node.
Note that if some \(U\)\nbd link, established by a controller~\(\beta\), is
traveled at some \(\alpha\)\nbd stage~\(s\), the \(U\)\nbd link is destroyed
and the controller \(\beta\) is also initialized (since \(\beta\) is to the
right of the \(U\)\nbd link). Establishing another \(U\)\nbd link at the
beginning of the next \(\alpha\)\nbd stage requires a controller to the left of
\(\beta\). However, there are potentially only finitely many of them by
slowdown condition (an~\(R\)\nbd node which is visited for the first time
cannot become a controller at the same stage). Therefore, there will be
infinitely many stages when \(0\)\nbd outcome of~\(\alpha\) is visited.

Suppose that \(\alpha\) is an \(R\)\nbd node and \(s_0\) is the stage after
which \(\alpha\) is not initialized. If there exists some stage \(s_1>s_0\)
when \(\alpha\) becomes a controller, then for each \(\alpha\)\nbd stage
\(s>s_1\), we must have \(\cD_s(\alpha)=\alpha\) and \(\alpha^\frown d\) is
visited. That is, the \(d\)\nbd outcome is the true outcome. If such \(s_1\)
does not exist, then \(\alpha\) will visit either the~\(d\)-, the~\(w\)-,
or the \(U\)\nbd outcome of~\(\alpha\). The true outcome of \(\alpha\) is
therefore well-defined for~\(\alpha\).

Let \(o\)\nbd outcome be the true outcome of \(\alpha\) and \(s_0\) be the
stage after which we do not visit any node to the left of the \(o\)\nbd
outcome. There are at most finitely many controllers to the left of
\(\alpha^\frown o\), each of which sees at most finitely much
noise. Therefore, there is some \(s_1>s_0\) after which \(\alpha^\frown o\)
is not initialized.

Hence, \(p\) is infinite and each \(\alpha\subseteq p\) is initialized
finitely often.
\end{proof}

The following lemmas argue along the true path.

\begin{lemma}\label{lem:1U R is satisfied}
Each \(R^e\)\nbd requirement is satisfied for each \(e\in \omega\).
\end{lemma}

\begin{proof}
Let~\(p\) be the true path. By Lemma~\ref{lem:1U full requirement}, let
\(\alpha\subset p\) be the longest \(R\)\nbd node assigned an \(R^e\)\nbd
requirement, say, it is an \(R_c(\Psi)\)\nbd requirement. Suppose the
\(w\)\nbd outcome is the true outcome. By Lemma~\ref{lem:1U finite
initialization lemma}, there are~\(s_0\) and~\(x\) such that for each
\(s>s_0\), the only outcome of~\(\alpha\) that we visit is the \(w\)\nbd
outcome and \(\dw(\alpha)=x\). Then we claim that \(\Psi^{E\oplus U}(x)\neq
C(x)\): Otherwise, we will find a computation~\(y\) and hence obtain
\(\cE^\varnothing(\alpha)\), and so we will perform \(\enc(\alpha,U)\) and
then visit an outcome to the left of \(w\)\nbd outcome, contradicting the
choice of~\(s_0\).

Suppose the \(d\)\nbd outcome is the true outcome. Then there is~\(s_0\) such
that for each \(s>s_0\), the only outcome of~\(\alpha\) that will be visited
is the \(d\)\nbd outcome. By Lemma~\ref{lem:1U main lemma},~\(y_\alpha\) is
restored for each~\(s\). Thus, \(\Psi^{E\oplus U}(x)= 0\neq C(x) = 1\).
\end{proof}

\begin{lemma}\label{lem:1U S is satisfied}
The \(S(U)\)\nbd requirement is satisfied.
\end{lemma}

\begin{proof}
By Lemma~\ref{lem:1U full requirement}, let~\(\alpha\) be the \(S(U)\)\nbd
node such that for each~\(\beta\) with \(\alpha\subsetneq \beta\subset p\) we
have \(\seq(\alpha)=\seq(\beta)=(b,\xi)\). Note that for such~\(\beta\), we
have \(\mt(\beta,U)=\varnothing\) (Definition~\ref{def:maintain1U}).
Therefore, each~\(b\)\nbd th copy of any functional in~\(F_\xi(U)\) is not
killed by an \(R\)\nbd node below~\(\alpha\). Let~\(s_0\) be the stage such
that for each \(s>s_0\),~\(\alpha\) is not initialized. Since we never stop
an \(S\)\nbd node from correcting its functional, (the~\(b\)\nbd th copy of)
the functional in \(F_\xi(U)\) is correct and total. Hence, the \(S(U)\)\nbd
requirement is satisfied.
\end{proof}

\begin{lemma}\label{lem:1U G is satisfied}
The \(G\)\nbd requirement is satisfied.
\end{lemma}

\begin{proof}
By the \(G\)\nbd strategy~(\ref{it:G strategy 1}) and~(\ref{it:G strategy
3}), \(\Theta^{j(1)}(x)\) is eventually defined for each \(x\). Since the
\(G\)\nbd strategy~(\ref{it:G strategy 2}) can always act,
\(\Theta^{j(1)}(x)\) is correct.
\end{proof}

This completes the proof if we only have a single \(S(U)\)\nbd requirement to
satisfy. One can think of this construction as a sub-construction dealing
with a single \(U\)\nbd set. When we have \(S(U_i)\)\nbd requirements for
\(i\le k\), we have multiple sub-constructions organized in a nested way,
each giving a \(U_i\)\nbd condition that tells us whether a computation~\(y\)
is \(U_i\)\nbd restorable. Now we can simply take the conjunction of
\(U_i\)\nbd conditions for each \(i\le k\) to have a condition implying
that~\(y\) is restorable. Organizing multiple sub-constructions now only
requires (a lot of) patience.

\section{The full construction}\label{sec:2U}

We will make general definitions for the priority tree, some of the basic
strategies, etc., and give examples to demonstrate some of the combinatorics.
The three-element chain in Figure~\ref{fig:lattice3} will be used for all
examples in this section, and we will often restrict ourselves to considering
only~\(S_{U_0}\)- and \(S_{U_1}\)\nbd requirements (Figure~\ref{fig:2Utree}).
\emph{Lemma~\ref{lem:U restorable} and Lemma~\ref{lem:U restorable strong}
are essential to the validity of our construction, and will be tacitly
applied in all examples in this section.} The complexity of the construction
increases with the number of join-irreducible elements and therefore the
three-element chain will allow us to illustrate concretely the combinatorial
ideas used in the general case.

In sections~\ref{sec:2U construction} and~\ref{sec:2U verification}, we
present the general construction and verification. In the verification
section we try to strike a delicate balance between readability and being
formal.
Our formal proofs in the example sections are representative enough so that
in these final sections, we will appeal to those examples when there is no
loss in generality.

\subsection{The priority tree}\label{sec:2Utree}

Recall from Section~\ref{sec:1Utree} that we are considering a space
\(\{0,1,\ldots, m-1 \}\times [T_\cL] = m \times [T_\cL]\) and nondecreasing
maps \(R_c:[T_\cL] \to [T_\cL]\) for each \(c \in \Ji(\cL)\), where
\(m=\abs{\Ji(\cL)}+1\).
An \(S\)\nbd node~\(\alpha\) working for \(S_{U_0}, \ldots, S_{U_{k-1}}\) is
assigned to an element \((f_0,\ldots, f_{k-1}) \in {(m \times [T_\cL])}^k\),
where each \(f_i = (a_i, \xi_i) \in m \times [T_\cL]\). We let
\(\seq(\alpha)(i) = f_i\), \(\seq_0(\alpha)(i) = a_i\), and
\(\seq_1(\alpha)(i) = \xi_i\).

\begin{definition}\label{def:2Utree}
We define the priority tree~\(\cT\) by recursion: We assign \(((0,\iota))\in
{(m \times [T_\cL])}^1\) to the root node~\(\lambda\) and call it an
\(S\)\nbd node (\(\iota\) is the finite string \(00\cdots 0\) of the proper
length depending on \([T_\cL]\)).

Suppose that~\(\alpha\) is an \(S\)\nbd node. We determine the least~\(e\)
such that there is no \(R^e\)\nbd node \(\beta \subset \alpha\) with
\(\beta^\frown w \subseteq \alpha\) or \(\beta^\frown d \subseteq \alpha\),
and assign \(\alpha^\frown 0\) to~\(R^e\).

Suppose~\(\alpha\) is an \(R_c\)\nbd node for some \(c\in\Ji(\cL)\)
and~\(\alpha^-\) is assigned to \((f_0,\ldots, f_{k-1})\in {(m \times
[T_\cL])}^k\) where each \(f_i = (a_i,\xi_i)\).
We sequentially add \(U_i\)\nbd outcomes from \(i=k-1\) to \(i = 0\), where
each \(U_i\)\nbd outcome could be of Type~I or Type~II (and in that case
GREEN or RED)\@. Then we add a single \(\ctr\)\nbd outcome, and finally we
add the~\(w\)- and the \(d\)\nbd outcome. The priority order of each outcome,
however, varies and is described as follows.

Proceeding from \(i=k-1\) down to \(i=0\), we add the \(U_i\)\nbd outcomes:
\begin{enumerate}
\item
If \(\xi_i<R_c(\xi_i)\), then this \(U_i\)\nbd outcome is a \emph{Type~I}
outcome, and we assign \(\alpha^\frown U_i\) to
\[
(f_0,\ldots, f_{i-1}, (a_i, R_c(\xi_i)),(0,\iota),\ldots,
  (0,\iota) )\in {(m \times [T_\cL])}^k.
\]
The \emph{next} outcome to be added is placed just to the left of
this outcome.
\item
If \(\xi_i = R_c(\xi_i)\), then this \(U_i\)\nbd outcome is a
\emph{Type~II} outcome. If \(a_i<m-1\), then this outcome is GREEN and we
assign \(\alpha^\frown U_i\) to
\[
(f_0,\ldots, f_{i-1}, (a_i+1, \iota),
   (0,\iota),\ldots, (0,\iota) )\in {(m \times [T_\cL])}^k.
\]
If \(a_i = m-1\), then this outcome is RED\@, and we do not assign
\(\alpha^\frown U_i\) to any requirement; it is a terminal node. In either
case, the \emph{next} outcome to be added is placed just to the right of
this outcome.
\end{enumerate}
After the \(U_0 \)\nbd outcome has been added, the next outcome we add is the
\(\ctr\)\nbd outcome. We add it immediately to the left or right of the
\(U_0\)\nbd outcome depending on whether the \(U_0 \)\nbd outcome is a Type~I
outcome or a Type~II outcome. We do not assign the node \(\alpha^\frown
\ctr\) to any requirement; it is a terminal node.

Finally, we add the~\(w\)- and \(d\)\nbd outcomes to the right of all
existing outcomes with \(d\) the rightmost outcome and assign both
\(\alpha^\frown w\) and \(\alpha^\frown d\) to
\[
    (f_0,\ldots, f_{k-1}, (0,\iota))\in {(m \times [T_\cL])}^{k+1}.
\]
\end{definition}
By the same argument as in Lemma~\ref{lem:1U full requirement}, it is clear
that along any infinite path through~\(\cT\), all requirements are
represented by some node.

\begin{lemma}\label{lem:2Utree}
Let~\(p\) be an infinite path through~\(\cT\).
\begin{enumerate}
\item For each \(S_{U_i}\)\nbd requirement, there is an
\(S\)\nbd node~\(\alpha\) such that for all~\(\beta\) with \(\alpha
\subseteq \beta \subset p\), we have \(\seq(\alpha)(i) = \seq(\beta)(i)\).
\item For each~\(e\), there is an \(R^e\)\nbd node~\(\alpha\) such
that either \(\alpha^\frown d \subset p\) or \(\alpha^\frown w \subset p\).
\qed
\end{enumerate}
\end{lemma}

Suppose that we consider only two \(U\)\nbd sets. Then we assign both
\(\alpha^\frown w\) and \(\alpha^\frown d\) to \((f_0,f_1) \in{ (m \times
[T_\cL])}^2\) (where \((f_0, f_1) = \seq(\alpha^-)\)) instead of assigning
them to an element in \({(m \times [T_\cL])}^3\) as in
Definition~\ref{def:2Utree}. An example for the three element-lattice (see
Figure~\ref{fig:lattice3} and Figure~\ref{fig:lattice3-S}) is given in
Figure~\ref{fig:2Utree}. We hide some of the \(R\)\nbd nodes with~\(w\)- or
\(d\)\nbd outcomes from the tree. A \(U_i\)\nbd outcome has label~\(i\) for
short. A Type~II outcome is denoted by a thick line. A terminal node is
represented by a~\(\bullet\). Therefore, we also hide the label of the
\(\ctr\)\nbd outcome (i.e., any outcome not labeled but shown in
Figure~\ref{fig:2Utree} is a \(\ctr\)\nbd outcome). An \(S\)\nbd node
assigned to \(((1,1),(0,2))\), for example, is abbreviated as \(11,02\). The
only outcome of an \(S\)\nbd node is also hidden. Sometimes, to avoid having
repeated scenarios and monstrous priority, we replace~\(f_1\) by~\(*\) so
that no \(U_1\)\nbd functionals will be built, i.e., we do not attempt to
satisfy an \(S_{U_1}\)\nbd requirement at this node. Depending on
whether~\(R^{15}\) is an~\(R_{c_0}\)- or an \(R_{c_\lambda}\)\nbd strategy,
the priority tree grows in different ways, both of which are worth mentioning
and shown in Figure~\ref{fig:2Utree}, at the left and right bottom,
respectively.

\begin{figure}[ht]\centering
    \includegraphics[width=1\textwidth]{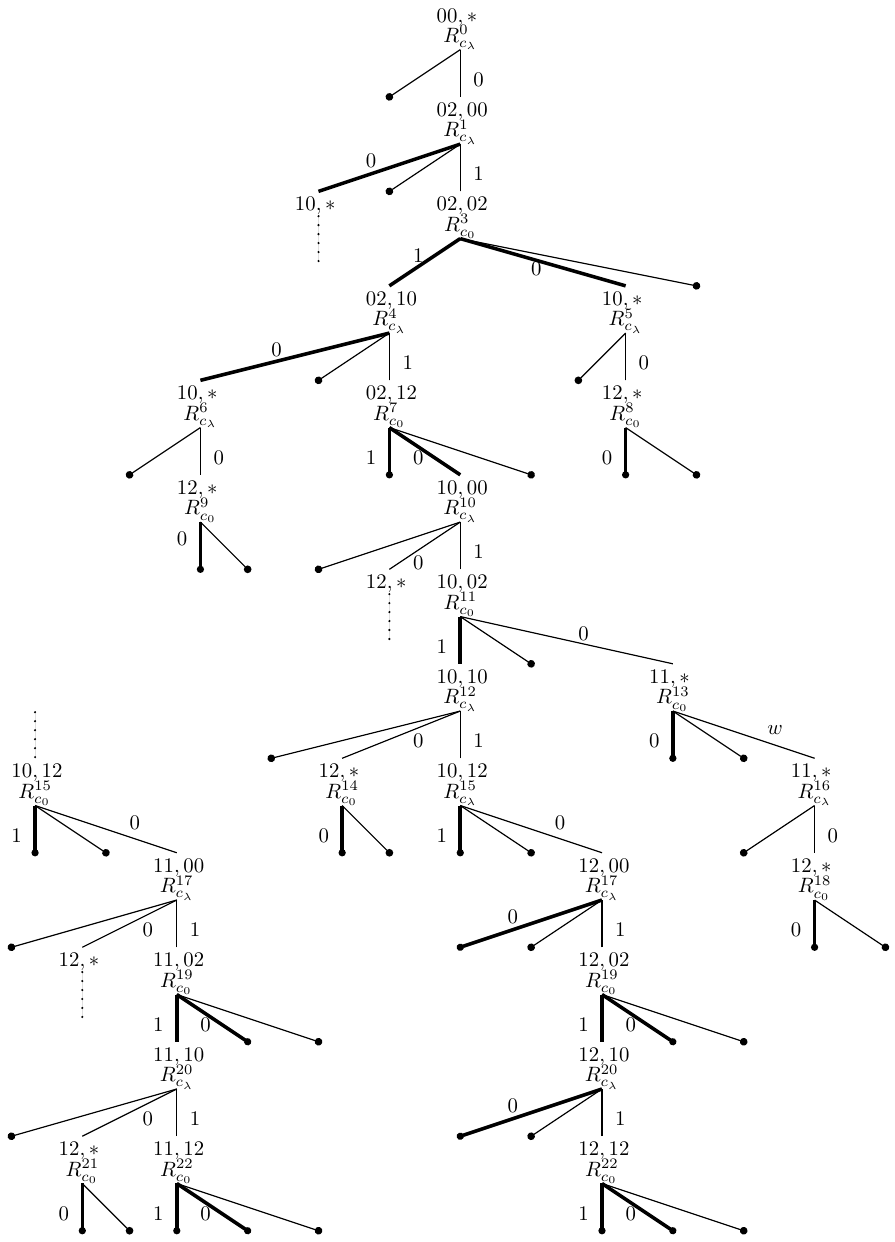}
    \caption{The priority tree for the 3-element chain}\label{fig:2Utree}
\end{figure}

Before we delve into Figure~\ref{fig:2Utree} and discuss the combinatorics,
we still have to explain the basic strategy for each node in order to have
some basic idea of how the construction works. (With a little extra effort,
we could discuss the strategies in general instead of the special case of
considering only~\(S_{U_0}\)- and \(S_{U_1}\)\nbd requirements, but we hope
the reader will be able to extrapolate from the case of two $U$-sets in
his/her mind.)

\subsection{\texorpdfstring{Functionals manipulated at \(S\)-nodes and
\(R\)-nodes}{Functionals manipulated at S-nodes and R-nodes}}%
\label{sec:2U functionals at S and R}

An \(S\)\nbd node~\(\beta\) builds and maintains each \(\Gamma\)\nbd
functional in \(\mt(\beta)\) defined as follow.
\begin{definition}\label{def:maintain2U}
Let~\(\beta\) be an \(S\)\nbd node with \(\seq(\beta)=(f_0,\ldots,f_{k-1})\in
(m\times [T_\cL])^k\) and \(\iota\) be the string \(00\cdots 0\) of the
proper length depending on~\([T_\cL]\).
\begin{enumerate}
\item
Suppose~\(\beta\) is the root of the priority tree, then in fact
\(\seq(\beta)=((0,\iota))\) and we define \(\mt(\beta)=F_\iota(U_0)\).
\item
Suppose otherwise, we let \(\alpha=(\beta^-)^-\). If
\(\seq(\alpha)=(g_0,\ldots,g_{k-2})\in (m\times [T_\cL])^{k-1}\), then in
fact we have \(g_i=f_i\) for \(i<k-1\) and \(f_{k-1}=(0,\iota)\). We define
\(\mt(\beta)=F_\iota(U_{k-1})\).

If \(\seq(\alpha)=(g_0,\ldots,g_{k-1})\in (m\times [T_\cL])^{k}\), we
let~\(i\) be the least such that \(g_i\neq f_i\). Then for \(j>i\) we have
\(f_j=(0,\iota)\). Suppose \(g_i=(a,\eta)\) and \(f_i=(b,\xi)\). Recall
from Definition~\ref{def:maintain1U} the definitions of \(\mt(\beta,U_i)\).
Then we define \(\mt(\beta)=\mt(\beta,U_i)\cup F_\iota(U_{i+1})\cup \cdots
\cup F_\iota(U_{k-1})\).
\end{enumerate}
Let \(\mt(\beta,U_{< i})\) be the subset of \(\mt(\beta)\) consisting of all
\(U_j\)\nbd functionals for \(j<i\).
\end{definition}

An \(R_c\)\nbd node~\(\beta\) will have to kill some functionals and build
and maintain at most one \(\Delta\)\nbd functional along each \(U_i\)\nbd
outcome.

\begin{definition}\label{def:kill2U}
Let~\(\beta\) be an \(R_c\)\nbd node for some \(c\in \Ji(\cL)\). Suppose
\(\seq(\beta^-)=(f_0,\ldots, f_{k-1})\in (m\times [T_\cL])^k\) and
\(f_i=(a_i,\eta_i)\) for each \(i<k\). Applying Definition~\ref{def:kill} to
each~\(f_i\), we have \(\kl(\beta,U_i)\) and \(\mt(\beta,U_i)\) defined
properly. For each \(i<k\), we define \(\kl(\beta,U_{\ge i})\) be the union
of \(\kl(\beta,U_i)\) and \(F_{\eta_j}(U_j)\) for each \(j>i\).
\end{definition}

By visiting the \(U_i\)\nbd outcome, the \(R_c\)\nbd node~\(\beta\) kills
each functional in \(\kl(\beta,U_{\ge i})\) and builds and maintains the
\(\Delta\)\nbd functional (if any) in \(\mt(\beta,U_i)\).

\subsection{\texorpdfstring{\(\beta^{*i}\) and \(\beta^{\sharp i}\)}%
{beta star and beta sharp}}\label{sec:2U star and sharp}

Let~\(\beta\) an \(R_c\)\nbd node for some \(c\in \Ji(\cL)\). Suppose
\(\seq(\beta^-)=(f_0,\ldots,f_{k-1})\in (m\times [T_\cL])^k\) with each
\(f_i=(a_i,\eta_i)\). If~\(U_i\) (\(i<k\)) is a Type~II outcome of~\(\beta\),
we define~\(\beta^{*i}\) by adapting the definition in Section~\ref{sec:star
and sharp}. Similarly for each \(i<k\) with \(a_i\neq 0\), we define
\(\beta^{\sharp i}\) by adapting the definition in Section~\ref{sec:star and
sharp}.

\subsection{Basic strategies}\label{sec:2U basic strategies}

An \(S\)\nbd node~\(\beta\) builds and maintains each \(\Gamma\)\nbd
functional that belongs to~\(\mt(\beta)\) using exactly the same correcting
strategy described in Section~\ref{sec:S strategy}. An \(R\)\nbd
node~\(\beta\) visiting its \(U_i\)\nbd outcome kills each functional that
belongs to \(\kl(\beta,U_{\ge i})\) using exactly the same killing strategy
described in Section~\ref{sec:R and killing}, and builds and maintains the
\(\Delta\)\nbd functional (if any) that belongs to \(\mt(\beta,U_i)\) using
exactly the same correcting strategy described in Section~\ref{sec:R and
delta}. Collecting \(\varnothing\)\nbd data when performing \(\enc(\beta,w)\)
is the same as in Section~\ref{sec:R and data}. After the \(\varnothing\)\nbd
data is obtained, we begin by performing \(\enc(\beta,U_{k-1})\),
where~\(U_{k-1}\) is the first outcome added to~\(\beta\). Analogously,
performing \(\enc(\beta,U_i)\) for \(i<k-1\) now requires \(U_{i+1}\)\nbd
data \(\cE^{U_{i+1}}(\beta)\), and performing \(\enc(\beta,\ctr)\) requires
\(U_0\)\nbd data \(\cE^{U_0}(\beta)\). We remark that a more suggestive
notation for \(\cE^{U_i}(\beta)\) would be \(\cE^{U_{\ge i}}(\beta)\), but we
choose to keep the notation short.

\subsection{\texorpdfstring{Weak \(U_i\)-data}{Weak U\_i-data}}%
\label{sec:2U examples}

As we are dealing with multiple \(S(U)\)\nbd requirements, it will be
convenient for us to reformulate \(\varnothing\)\nbd data and \(U\)\nbd data
that were used in Section~\ref{sec:1U} to a more general format. For each
\(\cE^{U_i}(\beta)\), we will also define a \emph{\(U_i\)\nbd data tree}
\(\cS^{U_i}(\beta)\) to reveal the combinatorics.

For the rest of the paper, we adopt the following:
\begin{center}
\emph{Convention 1: \(\cE^{U_i}(\beta)\) is abbreviated as~\(\cE^i(\beta)\).}

\emph{Convention 2: \(\cS^{U_i}(\beta)\) is abbreviated as~\(\cS^i(\beta)\).}

\emph{Convention 3: If \(k=\abs{\seq(\beta^-)}\), then
\(\cE^k(\beta)\coloneqq\cE^\varnothing(\beta)\).}
\end{center}

\begin{definition}[data]\label{def:2U data}
\(\cE=(Z,y,\cC)\) is \emph{data} if
\begin{enumerate}
\item
\(Z\) is a finite set of \(R\)\nbd nodes,
\item
\(y:\xi\mapsto y_\xi\) for each \(\xi\in Z\) such that~\(y_\xi\) is the
computation for~\(\xi\), and
\item
\(\cC\) is a finite set of functions of the form \(\Cond^i(\beta):\omega\to
\{0,1\}\) where \(\beta\in Z\) and \(i\in\omega\).
\end{enumerate}
\end{definition}

The symbol~\(y\), now viewed as a function defined on~\(Z\), is a bit
overused but it should cause no confusion. We also confuse~\(\cE\)
with~\(Z\). We will write \(\cE=\{\beta_1,\ldots ,\beta_k\}\) for
\(Z=\{\beta_1,\ldots ,\beta_k\}\) and \(\xi\in \cE\) for \(\xi\in Z\).

\begin{definition}[data tree]\label{def:2U data tree}
\(\cS=(S,f,g,h)\) is a \emph{data tree} if
\begin{enumerate}
\item
\(S\) is a finite binary tree,
\item
\(f:S\to \Ji(\cL)\),
\item\label{it:2U data tree 3}
\(g:S\to \{0,1\}\) such that \(g(\sigma)=1\) implies \(f(\sigma 0)=f(\sigma
1)\) and that \(g(\sigma)=0\) implies \(f(\sigma 0) \le f(\sigma 1)\), and
\item \(h:S\to \cT\).
\end{enumerate}
For \(\sigma\in S\), we say \(\sigma\) has \emph{type} \(c\) if
\(f(\sigma)=c\),~\(\sigma\) is \emph{strong} if \(g(\sigma)=1\),~\(\sigma\)
is \emph{weak} if \(g(\sigma)=0\),~\(\lambda\) denotes the empty string as
usual.

If \(\cS_0=(S_0,f_0,g_0,h_0)\) and \(\cS_1=(S_1,f_1,g_1,h_1)\) are two data
trees, and \(l\in \{0,1\}\), then \(\cS_0\otimes_l \cS_1 = (S,f,g,h)\) where
\(S(i\sigma)=S_i(\sigma)\), \(f(i\sigma)=f_i(\sigma)\),
\(g(i\sigma)=g_i(\sigma)\), and \(h(i\sigma)=h_i(\sigma)\) for \(i\in
\{0,1\}\), \(f(\lambda) = f_1(\lambda)\), \(g(\lambda) = l\),
\(h(\lambda)=h_1(\lambda)\).
\end{definition}
Here, \(f(\sigma) = c\) denotes that \(h(\sigma)\) is an \(R_c\)\nbd node.

\begin{definition}[\(\varnothing\)\nbd data]\label{def:2U 0 data}
Let~\(\beta\) be an \(R_c\)\nbd node for some \(c\in \Ji(\cL)\). Suppose that
the current stage is~\(s\). The \(\varnothing\)\nbd data is
\(\cE_s^\varnothing(\beta)=(Z,y,\cC)\) consisting of
\begin{enumerate}
\item
\(Z=\{\beta\}\),
\item
a computation~\(y_\beta\) for~\(\beta\), and
\item
\(\cC = \varnothing\).
\end{enumerate}
We also define the \(\varnothing\)\nbd data tree
\(\cS_s^\varnothing(\beta)=(S,f,g,h)\) where \(S=\{\lambda\}\), \(f(\lambda)
= c\), \(g(\lambda)=1\), and \(h(\lambda)=\beta\). As usual, the
subscript~\(s\) is often omitted.
\end{definition}

We adopt the same notation \(\cE^\varnothing(\beta)\) as in
Definition~\ref{def:0 data} even though the definition here is slightly
reformulated.

\begin{definition}[operation on data]\label{def:2U oplus data}

For an \(R_c\)\nbd node~$\beta$ such that
\(\seq(\beta^-)=(f_0,\ldots,f_{k-1})\in (m\times [T_\cL])^k\), we let the
\(U_k\)\nbd data of~\(\beta\) simply be \(\cE_s^\varnothing(\beta)\).

Given two \(U_{i+1}\)\nbd data (defined inductively from \(\varnothing\)\nbd
data) \(\cE^{i+1}(\alpha)=(Z_0,y_0,\cC_0)\) and
\(\cE^{i+1}(\beta)=(Z_1,y_1,\cC_1)\) (where \(Z_0\cap Z_1=\varnothing\)) and
a stage~\(s\), we define
\[
\cE^{i}(\beta)=\cE^{i+1}(\alpha)\otimes_s \cE^{i+1}(\beta)=(Z,y,\cC)
\]
as follows:
\begin{enumerate}
\item
\(Z=Z_0\cup Z_1\),
\item
\(y=y_0\cup y_1\) (so for \(\xi\in Z_i\), we have \(y_\xi=y_{i,\xi}\)), and
\item
\(\cC\) consists of
\begin{enumerate}
\item
\(\cC_0\) and~\(\cC_1\),
\item
for each \(\xi\in Z_1\), \(\Cond^i(\xi)(t)=\same(U_i,y_\xi,s,t)\), (In
this case,~\(y_\xi\) is the \(U_i\)\nbd reference length and~\(s\) the
\(U_i\)\nbd reference stage. \(\Cond^i(\xi)\) has \emph{type}~\(\same\).)
\item
for each \(\xi\in Z_0\), \(\Cond^i(\xi)(t)= \diff(U_i,y_\gamma,s,t)\),
where~\(\gamma\) is the shortest node in~\(Z_1\) (By our conventions
\(y_\gamma\ge y_{\gamma'}\) for other \(\gamma'\in Z_1\)). (In this
case,~\(y_\gamma\) is the \(U_i\)\nbd reference length and~\(s\) the
\(U_i\)\nbd reference stage. \(\Cond^i(\xi)\) has \emph{type}~\(\diff\).)
\end{enumerate}
\end{enumerate}
(Here we are identifying the predicate \(\same(U_i,y_\xi,s,t)\) with its
characteristic function, and we will keep doing so for other predicates.) We
say that \(\cE^i(\beta)\) \emph{extends} \(\cE^{i+1}(\alpha)\) and
\(\cE^{i+1}(\beta)\), and that \(\cE^{i+1}(\alpha)\) and \(\cE^{i+1}(\beta)\)
\emph{belongs to} \(\cE^i(\beta)\).
\end{definition}

In the above definition,~\(y_\gamma\) is simply the least~\(z\) such that if
\(\lnot \same(U,y_\xi,s,t)\) for each \(\xi\in \cE^{i+1}(\beta)\), then we
have \(\diff(U_i,z,s,t)\). Therefore, for a fixed \(t\), if
\(\Cond^i(\xi)(t)=0\) for each \(\xi\in \cE^{i+1}(\beta)\), then
\(\Cond^i(\xi)(t)=1\) for each \(\xi\in \cE^{i+1}(\alpha)\). For a
stage~\(t\) (usually clear from context), \(\Cond^i(\xi)\) \emph{holds} if
\(\Cond^i(\xi)(t)=1\).

The following is analogous to Definition~\ref{def:1U data}.

\begin{definition}[weak \(U_i\)\nbd data]\label{def:2U weak data}
Let~\(\beta\) be an \(R_c\)\nbd node for some \(c\in \Ji(\cL)\). Suppose
\(\seq(\beta^-)=(f_0,\ldots,f_{k-1})\) with \(f_i=(b_i,\xi_i)\). At
stage~\(s\), when we perform \(\enc(\beta,U_i)\) and~\(U_i\) is a RED
outcome,~\(\beta^{*i}\) is well defined (which is an \(R_d\)\nbd node for
some \(d\le c\)) and~\(\beta^{*i}\) is visiting its \(U_i\)\nbd outcome with
(weak or strong) \(U_{i+1}\)\nbd data
\(\cE^{i+1}(\beta^{*i})=(Z_0,y_0,\cC_0)\) and a \(U_{i+1}\)\nbd data tree
\(\cS^{i+1}(\beta^{*i})\).~\(\beta\) itself also has (weak or strong)
\(U_{i+1}\)\nbd data \(\cE^{i+1}(\beta)=(Z_1,y_1,\cC_1)\) and a \(U_i\)\nbd
data tree \(\cS^{i+1}(\beta)\). We define
\begin{align*}
\cE_s^i(\beta)&=\cE^{i+1}(\beta^{*i}) \otimes_s \cE^{i+1}(\beta)\\
\cS_s^i(\beta) &=\cS^{i+1}(\beta^{*i})\otimes_0 \cS^{i+1}(\beta)
\end{align*}
If there is no confusion, we might drop the subscript~\(s\)
of~\(\cE_s^i(\beta)\) and~\(\cS_s^i(\beta)\).
\end{definition}

Note that~\(\cS_s^i(\beta)\) defined as above satisfies
Definition~\ref{def:2U data tree}\eqref{it:2U data tree 3} and therefore is a
valid data tree. It is inductively clear from Lemma~\ref{lem:U restorable}
(and Lemma~\ref{lem:U restorable strong} for later discussion) that for each
\(\xi\in \cE^{i+1}(\beta^{*i})\), that \(\Cond^i(\xi)(t)=1\) holds implies
that \(\xi\) is \((U_i,D)\)\nbd restorable (Definition~\ref{def:U
restorable}); and for each \(\xi\in \cE^{i+1}(\beta)\), that
\(\Cond^i(\xi)(t)=1\) holds implies that~\(\xi\) is \((U_i,C)\)\nbd
restorable. \(\cE^i(\beta)\) focuses on the recursive aspects of the data
and~\(\cS^i(\beta)\) focuses on the combinatorial aspects of the data. They
are closed related.

As there will be many stages for the examples, we have the following
\begin{center}
\emph{Convention:~\(s_\beta^o\) denotes the stage when we perform
\(\enc(\beta,o)\) or \(\visit(\beta,o)\) for an \(o\)\nbd outcome.}
\end{center}
Note that~\(s_\beta^{\ctr}\) defined in Definition~\ref{def:1U controller}
conforms with this convention.

The next example is for \(\cE^0(R_{c_0}^{22})\) in Figure~\ref{fig:2Utree}.

\begin{example}[\(U_0\)\nbd data and \(U_0\)\nbd data tree]\label{eg:R22}
In the case of~\(R_{c_\lambda}^{15}\) (the other case~\(R_{c_0}^{15}\) can be
handled similarly), we will implement the basic strategies on the priority
tree in Figure~\ref{fig:2Utree}. We may assume that each node collects its
\(\varnothing\)\nbd data without delay. At stage~\(s\), the first node
along~\(R_{c_0}^{22}\) that is encountering a RED outcome is~\(R_{c_0}^7\),
encountering its \(U_1\)\nbd outcome (\(s_7^1 = s\)). We let
\begin{align*}
\cE^1(R_{c_0}^7)&= \cE^\varnothing(R_{c_\lambda}^4)\otimes_s
   \cE^\varnothing(R_{c_0}^7) = \{R_{c_\lambda}^4, R_{c_0}^7\},\\
\cS^1(R_{c_0}^7)&= \cS^\varnothing(R_{c_\lambda}^4)\otimes_0
   \cS^\varnothing(R_{c_0}^7),
\end{align*}
and~\(R_{c_0}^7\) immediately visits its \(U_0\)\nbd outcome, which
is~GREEN\@. So the \(U_1\)\nbd data \(\cE^1(R_{c_0}^7)\) is discarded. The
next node that is encountering a RED outcome is~\(R_{c_\lambda}^{15}\)
(\(s_{15}^1 = s\)). We let
\begin{align*}
\cE^1(R_{c_\lambda}^{15})&= \cE^\varnothing(R_{c_\lambda}^{12})\otimes_s
   \cE^\varnothing(R_{c_\lambda}^{15}) =
   \{R_{c_\lambda}^{12}, R_{c_\lambda}^{15}\},\\
\cS^1(R_{c_\lambda}^{15})&= \cS^\varnothing(R_{c_\lambda}^{12})\otimes_0
   \cS^\varnothing(R_{c_\lambda}^{15}),
\end{align*}
with
\begin{align*}
\Cond^1(R_{c_\lambda}^{15})(t)&=\same(U_1, y_{15}, s_{15}^1, t),\\
\Cond^1(R_{c_\lambda}^{12})(t)&=\diff(U_1, y_{15}, s_{15}^1, t).
\end{align*}
Then~\(R_{c_\lambda}^{15}\) visits its \(U_0\)\nbd outcome (\(s_{15}^0 =
s\)). At the same stage, we will have~\(R_{c_0}^{22}\) encounter its RED
\(U_1\)\nbd outcome (\(s_{22}^1 = s\)). Again, we let
\begin{align*}
\cE^{1}(R_{c_0}^{22})&=\cE^\varnothing(R_{c_\lambda}^{20})\otimes_s
   \cE^\varnothing(R_{c_0}^{22}) = \{R_{c_\lambda}^{20}, R_{c_0}^{22}\},\\
\cS^{1}(R_{c_0}^{22})&=\cS^\varnothing(R_{c_\lambda}^{20})\otimes_0
   \cS^\varnothing(R_{c_0}^{22}),
\end{align*}
with
\begin{align*}
\Cond^1(R_{c_0}^{22})(t)&=\same(U_1, y_{22}, s_{22}^1, t),\\
\Cond^1(R_{c_\lambda}^{20})(t)&=\diff(U_1, y_{22}, s_{22}^1, t).
\end{align*}
Then~\(R_{c_0}^{22}\) encounters its RED \(U_0\)\nbd outcome (\(s_{22}^0 =
s\)). Note that \({(R_{c_0}^{22})}^{*0} = R_{c_\lambda}^{15}\), so we now let
\begin{align*}
\cE^0(R_{c_0}^{22})&= \cE^1(R_{c_\lambda}^{15})\otimes_s
   \cE^1(R_{c_0}^{22})=\{R_{c_\lambda}^{12},
   R_{c_\lambda}^{15},R_{c_\lambda}^{20},R_{c_0}^{22}\},\\
\cS^0(R_{c_0}^{22})&= \cS^1(R_{c_\lambda}^{15})\otimes_0
   \cS^1(R_{c_0}^{22}),
\end{align*}
with
\begin{align*}
\Cond^0(R_{c_0}^{22})(t)&=\same(U_0, y_{22}, s_{22}^0, t),\\
\Cond^0(R_{c_\lambda}^{20})(t)&=\same(U_0, y_{20}, s_{22}^0, t),\\
\Cond^0(R_{c_\lambda}^{15})(t)&=\diff(U_0, y_{20}, s_{22}^0, t),\\
\Cond^0(R_{c_\lambda}^{12})(t)&=\diff(U_0, y_{20}, s_{22}^0, t).
\end{align*}
(Recall that the choice of~\(\gamma\) in Definition~\ref{def:2U oplus data}
is~\(R_{c_\lambda}^{20}\) here.)

As in Section~\ref{sec:1U}, we can introduce \emph{temporarily} (a formal
definition will be given later) a decision map \(\cD_t(R_{c_0}^{22})=\xi\in
\cE^0(R_{c_0}^{22})\) for the longest~\(\xi\) with \(\Cond^0(\xi)(t) =
\Cond^1(\xi)(t)=1\). For the decision map we temporarily define here, we note
that
\begin{itemize}
\item
\(\cD_t(R_{c_0}^{22})\) is defined for all \(t>s\),
\item
if \(\cD_t(R_{c_0}^{22}) = R_{c_0}^{22}\), then~\(R_{c_0}^{22}\) is
\((U_0,C_0)\)\nbd restorable and \((U_1,C_0)\)\nbd restorable,
\item
if \(\cD_t(R_{c_0}^{22}) = R_{c_\lambda}^{20}\),
then~\(R_{c_\lambda}^{20}\) is \((U_0,C_\lambda)\)\nbd restorable and
\((U_1,C_\lambda)\)\nbd restorable,
\item
if \(\cD_t(R_{c_0}^{22}) = R_{c_\lambda}^{15}\),
then~\(R_{c_\lambda}^{15}\) is \((U_0,C_\lambda)\)\nbd restorable and
\((U_1,C_\lambda)\)\nbd restorable, and
\item
if \(\cD_t(R_{c_0}^{22}) = R_{c_\lambda}^{12}\),
then~\(R_{c_\lambda}^{12}\) is \((U_0,C_\lambda)\)\nbd restorable and
\((U_1, C_\lambda)\)\nbd restorable.
\end{itemize}
Here we are applying Lemma~\ref{lem:U restorable} independently to~\(U_0\)
and~\(U_1\). To be precise, we have to state the \(\bigsame()\) condition as
in Lemma~\ref{lem:U restorable}~(\ref{it:lem U restorable 1})
and~(\ref{it:lem U restorable 2}); \emph{we choose to keep it tacit in all
remaining examples in this section}.

We note that \(f(\sigma 0)\le f(\sigma 1)=f(\sigma)\) for each \(\sigma\in
\cS^0(R_{c_0}^{22})\). This follows inductively from the properties
of~\(\beta^{*i}\) and~\(\beta\).
\end{example}

\begin{lemma}\label{lem:order}
Let~\(\beta\) be an \(R_c\)\nbd node with~\(\cE^0(\beta)\) and
\(\cS^0(\beta)=(S,f,g,h)\). For each \(\sigma\in S\), we have \(f(\sigma
0)\le f(\sigma 1) = f(\sigma)\); for two leaves of \(S\), if
\(\sigma<_{lex}\tau\), then \(h(\sigma)\subsetneq h(\tau)\).
\end{lemma}

\begin{proof}
An easy induction on~\(S\) (using Definition~\ref{def:2U data
tree}(\ref{it:2U data tree 3})).
\end{proof}

\subsection{Controllers}\label{sec:2U controller}

In Example~\ref{eg:R22}, is~\(R_{c_0}^{22}\) ready to be a controller based
on \(\cE^{0}(R_{c_0}^{22})\)? Recall from Definition~\ref{def:1U controller}
that we have the notion of a \(U^b\)\nbd problem. Looking
at~\(\cE^0(R_{c_0}^{22})\), we see that~\(R_{c_0}^{22}\) is an
\(R_{c_0}\)\nbd node and the others are \(R_{c_\lambda}\)\nbd nodes. This
will surely cause problems.

\begin{definition}[\(U_i^b\)\nbd problem]\label{def:U_problem}
Let~\(\beta\) be an \(R_c\)\nbd node with~\(\cE^0(\beta)\) and
\(\cS^0(\beta)=(S,f,g,h)\). For a \(\sigma\in S\), if \(f(\sigma 0)<c\), then
we let \(\alpha=h(\sigma 0)\in \cE^0(\beta)\), \(i=\abs{\sigma}\), and
\(b=\seq_0(\alpha^-)(i)\) (defined in the first paragraph of
Section~\ref{sec:2Utree}). We say that~\(\alpha\) is a \emph{\(U_i^b\)\nbd
problem} of \(\beta\), or simply a \(U_i\)\nbd problem. If in addition
\(f(\sigma)=c\), then~\(\alpha\) is the \emph{critical} \(U_i^b\)\nbd
problem, or simply a critical \(U_i\)\nbd problem.
\end{definition}

If \(\alpha=h(\tau)\) is a critical \(U_j\)\nbd problem for some~\(j\), then
there exists some~\(\sigma\) such that \(\tau=\sigma 0\) and in fact
\(j=\abs{\sigma}\).

\begin{example}[critical problems]
Continuing Example~\ref{eg:R22}, we have, in the case of
\(R^{15}=R_{c_\lambda}^{15}\), that~\(R_{c_\lambda}^{20}\) is a critical
\(U_1^1\)\nbd problem,~\(R_{c_\lambda}^{15}\) is a critical \(U_0^1\)\nbd
problem, and~\(R_{c_\lambda}^{12}\) is a (noncritical) \(U_1^1\)\nbd problem.

On the other hand, in the case of \(R^{15}=R_{c_0}^{15}\), we see
that~\(R_{c_\lambda}^{20}\) is a critical \(U_1^1\)\nbd
problem,~\(R_{c_\lambda}^{12}\) is a critical \(U_1^1\)\nbd problem,
and~\(R_{c_0}^{15}\) is not a \(U_0\)\nbd problem.
\end{example}

The next lemma follows from Definition~\ref{def:2U data tree}(\ref{it:2U data
tree 3}).
\begin{lemma}\label{lem:critical}
Let~\(\beta\) be an \(R_c\)\nbd node with~\(\cE^0(\beta)\) and
\(\cS^0(\beta)=(S,f,g,h)\). Suppose that \(\alpha=h(\sigma 0)\) is a critical
\(U_i^b\)\nbd problem (\(i=\abs{\sigma}\) and \(b=\seq_0(\alpha^-)(i)\)),
then \(g(\sigma)=0\). \qed
\end{lemma}

If \(\alpha=h(\sigma 0)\) is a critical \(U_i^b\)\nbd problem, then
\(g(\sigma)=0\) allows us to consider the node \(\hat{\alpha}=h(\sigma)\)
with the weak \(U_i\)\nbd data \(\cE^i(\hat{\alpha})=\cE^{i+1}(\alpha)\otimes
\cE^{i+1}(\hat{\alpha})\) where \(\alpha = \hat{\alpha}^{*i}\) (as we will
see later, we have \(g(\sigma)=0\) if and only if \(\cE^i(h(\sigma))\) is
weak \(U_i\)\nbd data defined in Definition~\ref{def:2U weak data}). If
\(\alpha=h(\sigma 0)\) is a noncritical \(U_i^b\)\nbd problem, we might have
that \(g(\sigma)=1\) (see Example~\ref{eg:monster}).

In Example~\ref{eg:R22}, having identified the critical problems
for~\(R_{c_0}^{22}\), we would like to group~\(R_{c_\lambda}^{15}\)
and~\(R_{c_\lambda}^{12}\) together and ignore the \(U_1\)\nbd condition for
now and only check if \(\Cond^0(R_{c_\lambda}^{15})\)
(\(=\Cond^0(R_{c_\lambda}^{12})\)) holds or not. In case it holds (and both
\(\Cond^0(R_{c_0}^{22})\) and \(\Cond^0(R_{c_\lambda}^{20})\) fail), we can
turn the GREEN \(U_0\)\nbd outcome of~\(R^7\) RED and proceed as usual. For
this purpose we have to make some modifications based
on~\(\cE^0(R_{c_0}^{22})\) and \(\cS^0(R_{c_0}^{22})\) to
get~\(\cE^{\ctr}(R_{c_0}^{22})\) and~\(\cS^{\ctr}(R_{c_0}^{22})\).

\begin{definition}[modified data and decision map]
\label{def:modified data and decision map}

Let~\(\beta\) be an \(R_c\)\nbd node with \(\cE^0(\beta)=(Z,y,\cC)\) and
\(\cS^0(\beta)=(S,f,g,h)\). When we perform \(\enc(\beta,\ctr)\) at~\(s\), we
do the following:
\begin{enumerate}
\item
Let \(\cS^{\ctr}(\beta)=(S',f',g',h')\) be a data tree defined recursively
as follows: we enumerate~\(\lambda\) into~\(S'\). Suppose that we have
enumerated~\(\sigma\) into~\(S'\), if~\(h(\sigma)\) is a critical
\(U_{\abs{\sigma}-1}\)\nbd problem, we stop; otherwise we
enumerate~\(\sigma 0\) and~\(\sigma 1\) into~\(S'\) and continue. Then for
each \(\sigma\in S'\), we define \(f'(\sigma)=f(\sigma)\),
\(g'(\sigma)=g(\sigma)\), and \(h'(\sigma)=h(\sigma)\)
\item
Let \(\cE^{\ctr}(\beta)=(Z',y',\cC')\) consist of the following:
\begin{enumerate}
\item
\(Z'=\{h(\sigma)\mid \sigma\in \cS^{\ctr}(\beta)\}\).
\item\label{it:modified data 2b}
If \(\alpha=h(\sigma 0)\in Z'\) is a critical \(U_{\abs{\sigma}}\)\nbd
problem, we let \(y'_\alpha = \max\{y_{h(\tau)} \mid \sigma0\subseteq
\tau\} = y_{h(\sigma 00\cdots 0)}\) (the last equality follows from our
assumption that if \(\alpha_0\subsetneq \alpha_1\), then
\(y_{\alpha_0}>y_{\alpha_1}\)). If \(\alpha\in Z'\) is not a critical
problem, we let \(y'_\alpha=y_\alpha\).
\item\label{it:modified data 2c}
To obtain~\(\cC'\), for each \(\alpha\in Z'\), we do the following:
\begin{enumerate}
\item\label{it:modified data 2c1}
If~\(\alpha\) is not a critical problem, then we enumerate
\(\Cond^j(\alpha)\in\cC\) into~\(\cC'\).
\item\label{it:modified data 2c2}
If~\(\alpha\) is a critical \(U_i\)\nbd problem, then for each \(j\le
i\) such that \(\Cond^j(\alpha)\in\cC\) has type~\(\diff\), we
enumerate \(\Cond^j(\alpha)\) into~\(\cC'\).
\item\label{it:modified data 2c3}
If~\(\alpha\) is a critical \(U_i\)\nbd problem, then for each \(j\le
i\) such that \(\Cond^j(\alpha)(t) = \same(U_j,y_\alpha,s_*,t)\in\cC\)
where~\(s_*\) is the \(U_j\)\nbd reference stage, we enumerate
$\Cond^j(\alpha)(t)=\same(U_j,y'_\alpha,s_*,t)$ into~\(\cC'\).
\end{enumerate}
For each \(\xi\in Z'\), we sometimes write \(\Cond_\beta^i(\xi)\) for
\(\Cond^i(\xi)\) to emphasize which node is the controller.
\end{enumerate}
\end{enumerate}
Based on~\(\cE^{\ctr}(\beta)\) and~\(\cS^{\ctr}(\beta)\), the \emph{decision
map}~\(\cD_s(\beta)\) is defined to be the longest \(\xi\in
\cE^{\ctr}(\beta)\) such that \(\Cond_\beta^i(\xi)(s)=1\) for each \(i\).

Note that~\(\cE^0(\beta)\) and~\(\cS^0(\beta)\) are not discarded yet.
\end{definition}

\begin{example}\label{eg:R22 controller}
Continuing Example~\ref{eg:R22}, \(\cE^{\ctr}(R_{c_0}^{22}) =
\{R_{c_\lambda}^{15},R_{c_\lambda}^{20},R_{c_0}^{22}\}\) with
\begin{align*}
\Cond_{22}^0(R_{c_\lambda}^{15})(t)&=\diff(U_0,y_{20},s_{22}^0,t),\\
\Cond_{22}^0(R_{c_\lambda}^{20})(t)&=\same(U_0,y_{20},s_{22}^0,t),\\
\Cond_{22}^0(R_{c_0}^{22})(t)&=\same(U_0,y_{22},s_{22}^0,t),\\
\Cond_{22}^1(R_{c_\lambda}^{20})(t)&=\diff(U_1,y_{22},s_{22}^1,t),
   \text{ and} \\
\Cond_{22}^1(R_{c_0}^{22})(t)&=\same(U_1,y_{22},s_{22}^1,t).
\end{align*}
Here \(s_{22}^1=s_{22}^0\) as in Example~\ref{eg:R22}. Note that
\(\cD_s(R_{c_0}^{22})\) is defined at each stage \(t>s_{22}^{\ctr}\).
\(\cE^1(R_{c_\lambda}^{15})\), which belongs to \(\cE^0(R_{c_0}^{22})\), is
not discarded; it will be used later.
\end{example}

Let us summarize what we have now. For an \(R_c\)\nbd node~\(\beta\) (where
\(c\in \Ji(\cL)\)), encountering its \(\ctr\)\nbd outcome
with~\(\cE^0(\beta)\) and~\(\cS^0(\beta)\), we first
obtain~\(\cE^{\ctr}(\beta)\) and~\(\cS^{\ctr}(\beta)\) as in
Definition~\ref{def:modified data and decision map} and~\(\beta\) becomes a
controller. While~\(\beta\) is a controller, as in Section~\ref{sec:1U},
we will put a restraint on \(\hat{C}\res s_\beta^{\ctr}\) for each
\(\hat{c}\neq c\). For a stage \(t>s_\beta^{\ctr}(\beta)\), if
\(\cD_t(\beta)=\xi\) is not a problem, then we simply restore~\(y_\xi\) and
activate the \(d\)\nbd outcome of~\(\xi\); if \(\cD_t(\beta)=\xi\) is a
critical \(U_i^b\)\nbd problem where \(b>0\), we restore~\(y_\xi\) (this is
the~\(y_\xi'\) in Definition~\ref{def:modified data and decision map}) and
turn the GREEN \(U_i\)\nbd outcome of~\(\xi^{\sharp i}\) into a RED outcome;
if \(\cD_t(\beta)=\xi\) is a critical \(U_i^b\)\nbd problem where \(b=0\), we
restore~\(y_\xi\) and search for two critical \(U_i\)\nbd problems that are
both \(R_d\)\nbd nodes for the same \(d\in \Ji(\cL)\) in the history and
obtain strong \(U_i\)\nbd data as in Section~\ref{sec:1U}. We will elaborate
on this in the next section.

\subsection{\texorpdfstring{Strong \(U_i\)-data and \(U_i\)-link}%
{Strong U\_i-data and U\_i-link}}\label{sec:2U strong data and link}

We begin with an example that requires strong \(U_1\)\nbd data.

\begin{example}[strong \(U_1\)-data and \(U_1\)-link]%
\label{eg:strong U1-data and U1-link}
Continuing Example~\ref{eg:R22 controller}, we suppose that at stage~\(s_*\)
we have \(\cD(R_{c_0}^{22})=R_{c_\lambda}^{20}\), the critical \(U_1^1\)\nbd
problem for~\(R_{c_0}^{22}\). Then we restore~\(y_{20}\) and turn the GREEN
\(U_1\)\nbd outcome of \(R_{c_0}^{19}=(R_{c_\lambda}^{20})^{\sharp 1}\)
RED\@. We also assume that for each \(t>s_*\), \(R_{c_0}^{22}\) does not see
any noise. Hence we have \(\bigsame(U_i,s_{22}^{\ctr},s_*,t)\) for \(i=0,1\)
and
\begin{align*}
\same(U_0,y_{20},s_{22}^{\ctr},s_*),\\
\diff(U_1,y_{22},s_{22}^{\ctr},s_*).
\end{align*}

Let us suppose that at \(s_{19}^{\ctr}>s_*\) \(R_{c_0}^{19}\) becomes a
controller with \(\cE^{\ctr}(R_{c_0}^{19})=\{R_{c_\lambda}^{15},
R_{c_\lambda}^{17}, R_{c_0}^{19}\}\) with
\begin{align*}
\Cond_{19}^0(R_{c_\lambda}^{15})(t)&=\diff(U_0,y_{17},s_{19}^{\ctr},t)\\
\Cond_{19}^0(R_{c_\lambda}^{17})(t)&=\same(U_0,y_{17},s_{19}^{\ctr},t)\\
\Cond_{19}^0(R_{c_0}^{19})(t)&=\same(U_0,y_{19},s_{19}^{\ctr},t)\\
\Cond_{19}^1(R_{c_\lambda}^{17})(t)&=\diff(U_1,y_{19},s_{19}^{\ctr},t)\\
\Cond_{19}^1(R_{c_0}^{19})(t)&=\same(U_1,y_{19},s_{19}^{\ctr},t)
\end{align*}
Suppose that at \(s_{**}>s_{19}^{\ctr}\), we have \(\cD_{s_
{**}}(R_{c_0}^{19})=R_{c_\lambda}^{17}\), a critical \(U_1^0\)\nbd problem
for~\(R_{c_0}^{19}\), so we are in the same situation as in
Example~\ref{eg:1U way out}. We restore~\(y_{17}\) and collect
\[
\cE^1(R_{c_\lambda}^{20}) = \cE^\varnothing(R_{c_\lambda}^{17})
   \otimes_{s_{**}}\cE^\varnothing(R_{c_\lambda}^{20}),
\]
where \(\cE^\varnothing(R_{c_\lambda}^{20})\) belongs to
\(\cE^{\ctr}(R_{c_0}^{22})\) and \(\cE^\varnothing(R_{c_\lambda}^{17})\) from
\(\cE^{\ctr}(R_{c_0}^{19})\). From Definition~\ref{def:2U oplus data}, the
\(U_1\)\nbd Conditions are
\begin{align*}
\Cond^1(R_{c_\lambda}^{20})(t)&=\same(U_1,y_{20},s_{**},t),\\
\Cond^1(R_{c_\lambda}^{17})(t)&=\diff(U_1,y_{20},s_{**},t).
\end{align*}
Then we establish a \(U_1\)\nbd link starting from the \(S\)\nbd
node~\(12,00\) and ending with the \(U_0\)\nbd outcome
of~\(R_{c_\lambda}^{20}\).

From Lemma~\ref{lem:U restorable strong}, we know that \(\Cond^1(R_
{c_\lambda}^{20})(t)=\same(U_1,y_{20},s_{**},t)\) implies that \(\diff(U_1,
y_{22},s_{22}^{\ctr},t)\) and hence~\(y_{20}\) is \((U_1,C_\lambda)\)\nbd
restorable at stage~\(t\), and
\(\Cond^1(R_{c_\lambda}^{17})(t)=\diff(U_1,y_{20},s_{**},t) \) implies that
\(\diff(U_1,y_{19},s_{19}^{\ctr},t)\) and hence~\(y_{17}\) is
\((U_1,C_\lambda)\)\nbd restorable at stage~\(t\).
\end{example}

We now give the formal definitions of \(U_i\)\nbd link and strong \(U_i\)\nbd
data.

\begin{definition}[environment]\label{def:env}
Let~\(\alpha\) be an \(R\)\nbd node. Define \(\env^{<i}(\alpha)\) to be the
shortest \(S\)\nbd node~\(\eta\) such that for each \(j<i\),
\(\seq(\eta)(j)=\seq(\alpha^-)(j)\) (if \(i=0\), then~\(\eta\) is defined to
be the root of~\(\cT\)). (This node~\(\eta\) will be the starting point of
the \(U_i\)\nbd link in the next definition.)

We say that~\(\alpha\) and~\(\beta\) \emph{have the same \(<i\)\nbd
environment} if \(\env^{<i}(\alpha) = \env^{<i} (\beta)\).
\end{definition}

A useful observation is the following: If \(\beta^{*i}=\alpha\) or
\(\beta^{\sharp i}=\alpha\), then \(\env^{<i}(\alpha) = \env^{<i} (\beta)\).

The following is analogous to Definition~\ref{def:1U strong data}.

\begin{definition}[strong \(U_i\)\nbd data]\label{def:2U strong data}
Suppose \(m=\abs{\cL}+1\). Suppose that for each \(k<m\), \(\alpha_a\) is a
\(U_i^a\)\nbd problem for the controller~\(\beta_a\) with \(\alpha_0
\subsetneq \cdots \subsetneq \alpha_{m-1}\) having the same
\emph{strong}\footnote{This is a subtle point and can be ignored for now. We
refer the reader to Definition~\ref{def:strong env} and to Example~\ref{eg:R3
decides R0} for the intuition. The existence of such a sequence is proved in
Lemma~\ref{lem:2U strong env}.} \(<i\)\nbd environment. For each \(a<m\), let
\(\cE^{i+1}(\alpha_a)\) belong to \(\cE^{\ctr}(\beta_a)\). Let~\(s\) be the
stage when \(\cD_s(\beta_0)=\alpha_0\). By the Pigeonhole Principle, we have
for some \(0\le a<b<m\) and some \(c\in \Ji(\cL)\) such that
both~\(\alpha_a\) and~\(\alpha_b\) are \(R_c\)\nbd nodes. We then define
\begin{align*}
\cS^{i}(\alpha_b)&=\cS^{i+1}(\alpha_a)\otimes_1 \cS^{i+1}(\alpha_b),\\
\cE^{i}(\alpha_b)&=\cE^{i+1}(\alpha_a)\otimes_s \cE^{i+1}(\alpha_b)
\end{align*}

We also establish a \emph{\(U_i\)\nbd link} starting from
\(\env^{<i}(\alpha_a)\) and ending with \((\alpha_b)^\frown U_i\). It will be
destroyed immediately after it is traveled.
\end{definition}

\emph{We will tacitly assume without further proof that Lemma~\ref{lem:U
restorable strong} applies to the strong data in the remaining examples in
this section.}

When we visit an \(S\)\nbd node~\(\alpha\) with a \(U_i\)\nbd link, we only
maintain functionals that belong to \(\mt(\beta,U_{<i})\)
(Definition~\ref{def:maintain2U}) and travel immediately along the
\(U_i\)\nbd link. Note that in \(\cS^i(\alpha_b)\) defined in the above
definition, we have \(g(\lambda)=1\) and also \(f(0)=f(1)=f(\lambda)\).
Therefore \(\cS^i(\alpha_b)\) still satisfies Definition~\ref{def:2U data
tree}\eqref{it:2U data tree 3}.

The key ingredients of the construction have been covered. Controllers follow
the same strategies as in Section~\ref{sec:1U}. Continuing
Example~\ref{eg:strong U1-data and U1-link}, in the next example we quickly
have a controller without problems.

\begin{example}[Controller \(R_{c_\lambda}^{20}\)]
In Example~\ref{eg:strong U1-data and U1-link}, we obtained at
stage~\(s_{**}\) our first strong \(U_1\)\nbd data
\(\cE^1(R_{c_\lambda}^{20})\) and established a \(U_1\)\nbd link starting
from the \(S\)\nbd node \(12,00\) and ending with
\((R_{c_\lambda}^{20})^\frown U_0\), where the \(U_0\)\nbd outcome is a RED
outcome.

Suppose that at stage \(s>s_{**}\) the \(U_1\)\nbd link is traveled and we
are encountering the RED \(U_0\)\nbd outcome of~\(R_{c_\lambda}^{20}\). We
are then obtain
\begin{align*}
\cE^0(R_{c_\lambda}^{20}) &= \cE^1(R_{c_\lambda}^{15})\otimes_s
   \cE^1(R_{c_\lambda}^{20}) = \{R_{c_\lambda}^{12}, R_{c_\lambda}^{15},
   R_{c_\lambda}^{17},R_{c_\lambda}^{20}\},\\
\cS^0(R_{c_\lambda}^{20}) &= \cS^1(R_{c_\lambda}^{15})\otimes_0
   \cS^1(R_{c_\lambda}^{20}).
\end{align*}
Then we should encounter the next outcome, the \(\ctr\)\nbd outcome
of~\(R_{c_\lambda}^{20}\). As~\(R_{c_\lambda}^{12}\),
\(R_{c_\lambda}^{15}\),~\(R_{c_\lambda}^{17}\), and~\(R_{c_\lambda}^{20}\)
are all \(R_{c_\lambda}\)\nbd nodes, we have
\(\cS^{\ctr}(R_{c_\lambda}^{20})=\cS^0(R_{c_\lambda}^{20})\) and
\(\cE^{\ctr}(R_{c_\lambda}^{20}) = \cE^0(R_{c_\lambda}^{20})\)
(Definition~\ref{def:modified data and decision map}). (Note that since
\(\cS^{\ctr}(R_{c_\lambda}^{20})=\cS^0(R_{c_\lambda}^{20})\), we also have
\(\cE^{\ctr}(R_{c_\lambda}^{20})=\cE^0(R_{c_\lambda}^{20})\).) Then
\(R_{c_\lambda}^{20}\) becomes a controller at this stage
\(s=s_{R_{c_\lambda}^{20}}^{\ctr}\), so we enumerate each diagonalizing
witness into~\(C_\lambda\) and put a restraint on \(C_0\res s\). As in
Section~\ref{sec:1U}, if \(\cD_t(R_{c_\lambda}^{20}) = \xi\), then~\(y_\xi\)
is both~\(U_0\)- and \(U_1\)\nbd restorable. Therefore we can safely
restore~\(y_\xi\) if \(\cD_t(R_{c_\lambda}^{20}) = \xi\) and activate the
\(d\)\nbd outcome of~\(\xi\).
\end{example}

\subsection{A monstrous example}

The next example is a monstrous example in which we see how the combinatorics
grows complicated.

\begin{example}\label{eg:monster}
Suppose that we have~\(R_{c_0}^{15}\) in Figure~\ref{fig:2Utree}. The first
\(U_0\)\nbd data would be
\begin{align*}
\cE^0(R_{c_0}^{22})& =
   \cE^1(R_{c_0}^{15}) \otimes_s \cE^1(R_{c_\lambda}^{22})\\
& = (\cE^\varnothing(R_{c_\lambda}^{12})\otimes_s \cE^\varnothing
   (R_{c_0}^{15})) \otimes_s (\cE^\varnothing
   (R_{c_\lambda}^{20})\otimes_s \cE^\varnothing(R_{c_0}^{22})),
\end{align*}
obtained at some stage~\(s\). Then we encounter the \(\ctr\)\nbd outcome
of~\(R_{c_0}^{22}\) and have \(\cS^{\ctr}(R_{c_0}^{22})=\cS^0(R_{c_0}^{22})\)
and hence \(\cE^{\ctr}(R_{c_0}^{22}) = \cE^0(R_{c_0}^{22})\).
\(R_{c_0}^{22}\) becomes a controller at
\(s_{22}=s_{R_{c_0}^{22}}^{\ctr}=s\). Let \(s_{22}'>s_{22}\) be the stage
when \(\cD_{s_{22}'}(R_{c_0}^{22})=R_{c_\lambda}^{12}\) and~\(R_{c_0}^{22}\)
sees no more noise from then on. Since~\(y_{12}\), the computation
for~\(R_{c_\lambda}^{12}\), is only weakly restorable, we restore it and turn
the GREEN \(U_1\)\nbd outcome of \(R_{c_0}^{11}\) RED\@.

Next we might reach~\(R_{c_0}^{18}\) at some stage \(s>s_{22}'>s_{22}\) (we
recycle the symbol~\(s\)) through \({R_{c_0}^{13}}^\frown w\) and have
\begin{align*}
\cE^0(R_{c_0}^{18})& = \cE^1(R_{c_0}^{11}) \otimes_s
   \cE^\varnothing(R_{c_0}^{18})\\
& = (\cE^\varnothing(R_{c_\lambda}^{10})\otimes_s
   \cE^\varnothing(R_{c_0}^{11})) \otimes_s \cE^\varnothing(R_{c_0}^{18}),
\end{align*}
Again since \(\cS^{\ctr}(R_{c_0}^{18})=\cS^0(R_{c_0}^{18})\), we have
\(\cE^{\ctr}(R_{c_0}^{18})=\cE^0(R_{c_0}^{18})\) and~\(R_{c_0}^{18}\) becomes
a controller at \(s_{18}=s_{R_{c_0}^{18}}^{\ctr}=s\). Let~\(s_{18}'\) be the
stage when \(\cD_{s_{18}'}(R_{c_0}^{18})=R_{c_\lambda}^{10}\), we
restore~\(R_{c_\lambda}^{10}\) and notice that~\(R_{c_\lambda}^{10}\)
and~\(R_{c_\lambda}^{12}\), both of which are \(R_{c_\lambda}\)\nbd nodes and
have the same \(<1\)\nbd environment (Definition~\ref{def:env}), are a
critical \(U_1^0\)\nbd problem and a critical \(U_1^1\)\nbd problem for the
controller~\(R_{c_0}^{22}\) and the controller~\(R_{c_0}^{18}\),
respectively. Thus we obtain the following strong \(U_1\)\nbd data and
\(U_1\)\nbd data tree
\begin{align*}
\cE^1(R_{c_\lambda}^{12}) &= \cE^\varnothing(R_{c_\lambda}^{10})
   \otimes_{s_{18}'} \cE^\varnothing(R_{c_\lambda}^{12}),\\
\cS^1(R_{c_\lambda}^{12}) &= \cS^\varnothing(R_{c_\lambda}^{10})
   \otimes_1 \cS^\varnothing(R_{c_\lambda}^{12}),
\end{align*}
and establish a \(U_1\)\nbd link starting from the \(S\)\nbd node~\(10,00\)
and ending at the \(U_0\)\nbd outcome of~\(R_{c_\lambda}^{12}\).

Let \(s>s_{18}'>s_{18}>s_{22}'>s_{22}\) (we recycle the symbol~\(s\) again)
be the stage when we travel the \(U_1\)\nbd link and visit~\(R_{c_0}^{14}\).
If we fail to obtain \(\cE^\varnothing(R_{c_0}^{14})\) (which is always the
case by the slowdown condition if this is the first time we
visit~\(R_{c_0}^{14}\)), we have to initialize the
controllers~\(R_{c_0}^{22}\) and~\(R_{c_0}^{18}\) anyway. Wasting a lot of
work does not matter as long as we make progress towards some nodes to the
left of the controllers~\(R_{c_0}^{22}\) and~\(R_{c_0}^{18}\). If we obtain
\(\cE^\varnothing(R_{c_0}^{14})\), then we \(\enc(R_{c_0}^{14},U_0)\) and
obtain
\begin{align*}
\cE^0(R_{c_0}^{14}) &= \cE^1(R_{c_\lambda}^{12})
   \otimes_s \cE^\varnothing(R_{c_0}^{14}),\\
\cS^0(R_{c_0}^{14})&= \cS^1(R_{c_\lambda}^{12})
   \otimes_0 \cS^\varnothing(R_{c_0}^{14}),
\end{align*}
where \(R_{c_\lambda}^{12}=(R_{c_0}^{14})^{*0}\) and
\(\cE^1(R_{c_\lambda}^{14})\) is strong \(U_1\)\nbd data. Another important
observation is the following:
\begin{enumerate}[resume=observations]
\item
If~\(s^*\) is the last stage when we visit~\(R_{c_0}^{14}\), then we have
\[
s^*<s_{22}<s_{22}'<s_{18}<s_{18}'<s,
\]
and by slowdown condition
\[
\bigsame(U_0,y_{14},s^*,s).
\]
\end{enumerate}

Then we perform \(\enc(R_{c_0}^{14},\ctr)\). Notice that we have
\(\cS^{\ctr}(R_{c_0}^{14})\neq \cS^0(R_{c_0}^{14})\)
(Definition~\ref{def:modified data and decision map}) and hence
\(\cE^{\ctr}(R_{c_0}^{14}) = \{R_{c_\lambda}^{12},R_{c_0}^{14}\}\).
\emph{From this, we can see the necessity to introduce the notion of a
critical problem:} If we allow \(\cD(R_{c_0}^{14})=R_{c_\lambda}^{10}\), we
have that~\(R_{c_\lambda}^{10}\) is a \(U_1^0\)\nbd problem
for~\(R_{c_0}^{14}\) but~\((R_{c_\lambda}^{10})^{\sharp 1}\) is not defined.
We cannot obtain any new strong \(U_1\)\nbd data at this point. In this case,
we should instead look at the critical \(U_0\)\nbd
problem~\(R_{c_\lambda}^{12}\).

Let \(s_{14}=s_{R_{c_0}^{14}}^{\ctr}\) be the stage when~\(R_{c_0}^{14}\)
becomes a controller and \(s_{14}'>s_{14}\) be the stage when
\(\cD_{s_{14}'}(R_{c_0}^{14})=R_{c_\lambda}^{12}\). We will verify first that
\(R_{c_\lambda}^{12}\) and \(R_{c_\lambda}^{10}\) are \((U_0,C_\lambda)\)\nbd
restorable under
\(\Cond_{R_{c_0}^{14}}^0(R_{c_\lambda}^{12})(t)=\diff(U_0,y_{14},s_{14},t)\).
In fact, this follows directly from Lemma~\ref{lem:U restorable}. However, a
cautious reader might be worried about the use block that was killed by the
end of \(s_{22}\). Therefore we will show that
\[
\Cond_{R_{c_0}^{14}}^0(R_{c_\lambda}^{12})(t) \Rightarrow
   \Cond_{R_{c_0}^{22}}^0(R_{c_\lambda}^{12})(t) \land
   \Cond_{R_{c_0}^{18}}^0(R_{c_\lambda}^{10})(t).
\]
That is,
\[
\diff(U_0,y_{14},s_{14},t) \Rightarrow \diff(U_0,y_{20},s_{22},t)
   \land \diff(U_0,y_{18},s_{18},t).
\]
This follows clearly from
\[
y_{14}<y_{22}<y_{20}<y_{18},
\]
and
\[
\bigsame(U_0,y_{14},s^*,s_{14}).
\]

Let us continue. At stage~\(s_{14}'\), we have
\(\cD_{s_{14}'}(R_{c_0}^{14})=R_{c_\lambda}^{12}\) and therefore we turn the
GREEN \(U_0\)\nbd outcome of \(R_{c_0}^{7}=(R_{c_\lambda}^{12})^{\sharp 0}\)
RED\@.

Let \(s_7>s_{14}'\) be the stage when~\(R_{c_0}^{7}\) becomes a controller
with
\begin{align*}
\cE^{\ctr}(R_{c_0}^{7})&= \cE^0(R_{c_0}^{7})\\
&= \cE^\varnothing(R_{c_\lambda}^{0}) \otimes_{s_7}\cE^1(R_{c_0}^{7})\\
&= \cE^\varnothing(R_{c_\lambda}^{0}) \otimes_{s_7} (\cE^\varnothing
   (R_{c_\lambda}^{4})\otimes_{s_7} \cE^\varnothing(R_{c_0}^{7}))
\end{align*}
We now distinguish the two cases: If \(\cD_{s_7'}(R_{c_0}^{7}) =
R_{c_\lambda}^0\) for some \(s_7'>s_7\), then we continue in
Example~\ref{eg:R7 decides R0}; if \(\cD_{s_7'}(R_{c_0}^{7}) =
R_{c_\lambda}^4\) for some \(s_7'>s_7\), then we continue in
Example~\ref{eg:R7 decides R4}.
\end{example}

\begin{example}[\(\cD(R_{c_0}^{7})=R_{c_\lambda}^{0}\)]%
\label{eg:R7 decides R0}
Continuing Example~\ref{eg:monster}, we suppose \(\cD_{s_7'}(R_{c_0}^{7}) =
R_{c_\lambda}^{0}\). \(R_{c_\lambda}^{0}\) is a critical \(U_0^0\)\nbd
problem for~\(R_{c_0}^{7}\), and~\(R_{c_\lambda}^{12}\) is a critical
\(U_0^1\)\nbd problem for~\(R_{c_0}^{14}\), both have the same \(<0\)\nbd
environment. We will now obtain the following strong \(U_0\)\nbd data and
\(U_0\)\nbd data tree:
\begin{align*}
\cE^0(R_{c_\lambda}^{12})&=\cE^\varnothing(R_{c_\lambda}^{0})
   \otimes_{s_7'} \cE^1(R_{c_\lambda}^{12}),\\
\cS^0(R_{c_\lambda}^{12})&=\cS^\varnothing(R_{c_\lambda}^{0})
   \otimes_1 \cS^1(R_{c_\lambda}^{12}),
\end{align*}
where \(\cE^1(R_{c_\lambda}^{12})\) is the strong \(U_1\)\nbd data belonging
to~\(\cE^{\ctr}(R_{c_0}^{14})\).
Thus for each \(\sigma\in \cS^0(R_{c_\lambda}^{12})\), we have
\(g(\sigma)=1\) and hence \(f(\sigma)=c_\lambda\) by Definition~\ref{def:2U
data tree}\eqref{it:2U data tree 3}. We also establish a \(U_0\)\nbd link
starting from the root of the tree and ending at the \(\ctr\)\nbd outcome
of~\(R_{c_\lambda}^{12}\).

At \(s_{12}>s_{7}'\), the \(U_0\)\nbd link is traveled
and~\(R_{c_\lambda}^{12}\) becomes a controller with data
\(\cE^{\ctr}(R_{c_\lambda}^{12}) = \cE^0(R_{c_\lambda}^{12})\), which has no
more problems.
\end{example}

\begin{example}[\(\cD(R_{c_0}^{7})=R_{c_\lambda}^{4}\)]%
\label{eg:R7 decides R4}
Continuing Example~\ref{eg:monster}, we suppose \(\cD_{s_7'}(R_{c_0}^{7}) =
R_{c_\lambda}^{4}\). We will turn the GREEN \(U_1\)\nbd outcome
of~\(R_{c_0}^{3}\) RED\@. Note that the \(U_0\)\nbd outcome
of~\(R_{c_0}^{3}\) is still GREEN\@. Let \(s_8>s_7'\) be the stage
when~\(R_{c_0}^{8}\) becomes a controller with data
\[
\cE^{\ctr}(R_{c_0}^{8}) = \cE^0(R_{c_0}^{8}) = \cE^\varnothing
   (R_{c_\lambda}^{5})\otimes_{s_8} \cE^\varnothing(R_{c_0}^{8}).
\]
Let \(s_8'>s_8\) be the stage when
\(\cD_{s_8}(R_{c_0}^{8})=R_{c_\lambda}^{5}\). We will turn the GREEN
\(U_0\)\nbd outcome of~\(R_{c_0}^{3}\) RED\@.

Let \(s_3>s_8'\) be the stage when~\(R_{c_0}^{3}\) becomes a controller with
data
\begin{align*}
\cE^{\ctr}(R_{c_0}^{3})&= \cE^0(R_{c_0}^{3})\\
&= \cE^\varnothing(R_{c_\lambda}^{0}) \otimes_{s_3}\cE^1(R_{c_0}^{3})\\
&= \cE^\varnothing(R_{c_\lambda}^{0}) \otimes_{s_3} (\cE^\varnothing
   (R_{c_\lambda}^{1})\otimes_{s_3} \cE^\varnothing(R_{c_0}^{3}))
\end{align*}

Depending on the decision of~\(R_{c_0}^{3}\), we again have to split cases
and consider the following two examples separately.
\end{example}

\begin{example}[\(\cD(R_{c_0}^{3})=R_{c_\lambda}^{1}\)]\label{eg:R3 decide R1}
Continuing Example~\ref{eg:R7 decides R4}, we let~\(s_3'\) be the stage when
we have \(\cD_{s_3'}(R_{c_0}^{3}) = R_{c_\lambda}^{1}\). Then we will obtain
strong \(U_1\)\nbd data
\begin{align*}
\cE^1(R_{c_\lambda}^{4})&= \cE^\varnothing(R_{c_\lambda}^{1})\otimes_{s_3'}
   \cE^\varnothing(R_{c_\lambda}^{4}),\\
\cS^1(R_{c_\lambda}^{4})&= \cS^\varnothing(R_{c_\lambda}^{1})\otimes_{1}
   \cS^\varnothing(R_{c_\lambda}^{4}),
\end{align*}
where \(\cE^\varnothing(R_{c_\lambda}^{4})\) belongs to
\(\cE^{\ctr}(R_{c_0}^{7})\). We also establish a \(U_1\)\nbd link starting
from the \(S\)\nbd node~\(02,00\) and ending at the \(U_0\)\nbd outcome
of~\(R_{c_\lambda}^{4}\).

Since the \(U_0\)\nbd outcome of~\(R_{c_\lambda}^{4}\) is GREEN, our strong
\(U_1\)\nbd data is discarded when we visit this outcome at stage~\(s\).
(Here, discarding the data is acceptable as now we are able to visit the nodes
below the \(U_0\)\nbd outcome, which is a progress.)
Let us assume that~\(R_{c_0}^{9}\) becomes a controller at \(s_9=s\) with data
\(\cE^{\ctr}(R_{c_0}^9)=\{R_{c_\lambda}^6,R_{c_0}^9\}\). Suppose
\(\cD(R_{c_0}^{9})=R_{c_\lambda}^{6}\) at \(s_9'>s_9\) and therefore the
\(U_0\)\nbd outcome of~\(R_{c_\lambda}^{4}\) becomes RED\@. Then, we go over
the procedures once again starting from Example~\ref{eg:monster} and
Example~\ref{eg:R7 decides R4} to the point when we obtain the strong
\(U_1\)\nbd data \(\cE^1(R_{c_\lambda}^{4})\) and establish the \(U_1\)\nbd
link starting from the \(S\)\nbd node~\(02,00\) and ending at the \(U_0\)\nbd
outcome of~\(R_{c_\lambda}^{4}\).

Let \(s>s_3'\) be the stage when we travel the link. This time, as the
\(U_0\)\nbd outcome of~\(R_{c_\lambda}^{4}\) is RED, we obtain
\begin{align*}
\cE^0(R_{c_\lambda}^{4})&=\cE^\varnothing(R_{c_\lambda}^{0})
   \otimes_s \cE^1(R_{c_\lambda}^{4})\\
\cS^0(R_{c_\lambda}^{4})&=\cS^\varnothing(R_{c_\lambda}^{0})
   \otimes_0 \cS^1(R_{c_\lambda}^{4})&
\end{align*}
and then we perform \(\enc(R_{c_\lambda}^{4},\ctr)\).~\(R_{c_\lambda}^{4}\)
is then a controller without any more problems. If~\(R_{c_0}^{9}\) never sees
any noise again, then the \(U_0\)\nbd outcome of~\(R_{c_\lambda}^{4}\)
remains RED forever and we never go back to~\(R_{c_\lambda}^{6}\)
and~\(R_{c_0}^{9}\); if~\(R_{c_0}^{9}\) sees some noise, then
\({(R_{c_\lambda}^4)}^\frown \ctr\) is initialized.
\end{example}

\begin{example}[\(\cD(R_{c_0}^{3})=R_{c_\lambda}^{0}\)]%
\label{eg:R3 decides R0}
Continuing Example~\ref{eg:R7 decides R4}, we let~\(s_3'\) be the stage when
we have \(\cD_{s_3'}(R_{c_0}^{3}) = R_{c_\lambda}^{0}\)
and~\(R_{c_\lambda}^{0}\) is a critical \(U_0^0\)\nbd problem
for~\(R_{c_0}^{3}\).

Recall that for each \(t\ge s_{14}'\) we assume that
\(\cD_{t}(R_{c_0}^{14})=R_{c_\lambda}^{12}\), which is a critical
\(U_0^1\)\nbd problem for~\(R_{c_0}^{14}\), and also that for each \(t\ge
s_8'\) we assume \(\cD_t(R_{c_0}^{8})=R_{c_\lambda}^{5}\), which is also a
critical \(U_0^1\)\nbd problem for~\(R_{c_0}^{8}\). Now we have a choice:
which \(U_0^1\)\nbd problem do we combine with~\(R_{c_\lambda}^0\) to get a
strong \(U_0\)\nbd data? We refer the reader to Example~\ref{eg:R7 decides
R0} in which we do obtain the strong \(U_0\)\nbd data by
combining~\(R_{c_\lambda}^{0}\) and~\(R_{c_\lambda}^{12}\). However, after a
moment of thought, we may prefer~\(R_{c_\lambda}^{5}\)
over~\(R_{c_\lambda}^{12}\) in the current situation; we say that
\(R_{c_\lambda}^{5}\) and \(R_{c_\lambda}^{0}\) has the same \emph{strong}
\(<0\)\nbd environment (see Definition~\ref{def:strong env}). (To have a
better intuition for our choice, we should \emph{imagine} that the
\(U_0\)\nbd outcome of~\(R_{c_0}^{3}\) is a Type~I outcome.
Then~\(R_{c_0}^{3} \) itself could be a critical \(U_0^0\)\nbd problem for
some controller below the \((R_{c_0}^{3})^\frown U_0\). We will naturally
search for a \(U_0^1\)\nbd problem also below the \((R_{c_0}^{3})^\frown
U_0\) instead of searching for one below the RED \(U_1\)\nbd outcome
of~\(R_{c_0}^{3}\).) Therefore, we obtain
\begin{align*}
    \cE^0(R_{c_\lambda}^{5})&=\cE^\varnothing(R_{c_\lambda}^{0})
       \otimes_s \cE^1(R_{c_\lambda}^{5})\\
    \cS^0(R_{c_\lambda}^{5})&=\cS^\varnothing(R_{c_\lambda}^{0})
       \otimes_1 \cS^1(R_{c_\lambda}^{5})&
\end{align*}
and establish a \(U_0\)\nbd link starting from the root of the tree and ending
at the \(\ctr\)\nbd outcome of \(R_{c_\lambda}^5\).

At \(s_5\), the \(U_0\)\nbd link is traveled and \(R_{c_\lambda}^5\) becomes
a controller with data
\(\cE^{\ctr}(R_{c_\lambda}^{5})=\cE^0(R_{c_\lambda}^{5}) \), which has no
more problems.

This finishes our monstrous example.
\end{example}

Our final remark is the following: Our priority tree is defined in a uniform
way and the examples above strictly follow this definition even though we
might have used a shortcut in certain cases; but optimizing the priority tree
is not our concern.

\begin{definition}\label{def:strong env}
Let \(\alpha\subseteq \beta\) be two nodes. We say that~\(\alpha\)
and~\(\beta\) have the same \emph{strong} \(<i\)\nbd environment
at stage~\(s\) if
\begin{enumerate}
\item\label{it:strong env 1}
\(\env^{<i}(\alpha) = \env^{<i}(\beta) =\eta\) for some \(S\)\nbd
node~\(\eta\). (This depends only on the priority tree.)
\item\label{it:strong env 2}
For each~\(\xi\) with \(\alpha\subseteq \xi^\frown U_j\subseteq \beta\) for
some \(j>i\),
if~\(U_j\) is Type~II, then~\(U_j\) is GREEN\@.
\end{enumerate}
\end{definition}

If \(\cD(\beta)=\alpha\) where~\(\alpha\) is a critical \(U_i^0\)\nbd
problem, then we should have \(\alpha=\alpha_0\subsetneq \alpha_1 \subsetneq
\cdots \subsetneq \alpha_{m-1}\) where each~\(\alpha_j\) is a \(U_i^j\)\nbd
problem and has the same strong \(<i\)\nbd environment. Then we
obtain strong \(U_i\)\nbd data according to Definition~\ref{def:2U strong
data}. See Lemma~\ref{lem:2U strong env} for a proof.

\subsection{The construction}\label{sec:2U construction}

At stage~\(s\), we first run the controller strategy (see below) and then the
\(G\)\nbd
strategy (see Section~\ref{sec:1U construction}). Then we perform \(\visit
(\lambda)\) (see below), where~\(\lambda\) is the root
of the priority tree~\(\cT\). We stop the current stage whenever we
perform
\(\visit(\alpha)\) for some~\(\alpha\) with \(\abs{\alpha}=s\).

\subsubsection*{\(\visit(\alpha)\) for an \(S\)-node:}

Suppose that there is some \(U_i\)\nbd link connecting~\(\alpha\) and
\(\beta^\frown o\) for some \(o\)\nbd outcome of~\(\beta\) (in fact,
\(o=U_{i-1}\) if \(i>0\); \(o=\ctr\) if \(i=0\)). Then we maintain each
\(\Gamma\)\nbd functional in \(\mt(\alpha,U_{<i})\)
(Definition~\ref{def:maintain2U}) and perform \(\enc(\beta,o)\).

Suppose that there is no link. Then we maintain each \(\Gamma\)\nbd
functional in \(\mt(\alpha)\) (Definition~\ref{def:maintain2U}) and stop
the current substage and perform
\(\visit(\alpha^\frown 0)\) for the \(R\)\nbd node \(\alpha^\frown 0\).

To build and maintain a \(\Gamma^{E\oplus U}=C\) (for some \(c\in \Ji(\cL)\))
that belongs to \(\mt(\alpha,U_{<i})\) or \(\mt(\alpha)\), depending on which
case we have, we do the following: For each \(x\le s\):
\begin{enumerate}
\item
Suppose \(\Gamma_s^{E\oplus U}(x)\downarrow=C_s(x)\). Then~\(\beta\) does
nothing else.
\item
Suppose \(\Gamma_s^{E\oplus U}(x)\downarrow\neq C_s(x)\) with use block
\(\sB=\bB_s(\gamma,x)\). \begin{enumerate}
\item
If~\(\sB\) is killed and not \(E\)\nbd restrained, then~\(\beta\)
enumerates an unused point, referred to as a \emph{killing point},
into~\(\sB\) via~\(E\). Then we go to~(3) immediately.
\item
If~\(\sB\) is not killed and not \(E\)\nbd restrained, then~\(\beta\)
enumerates an unused point, referred to as a \emph{correcting point},
into~\(\sB\) via~\(E\). Then we redefine \(\Gamma_s^{E\oplus U}(x) =
C_s(x)\) with the \emph{same} use block~\(\sB\).
\end{enumerate}
\item
Suppose \(\Gamma_s^{E\oplus U}(x)\uparrow\). If each
\(\bB_{\bday{t}}(\gamma,x)\) with \(t<s\) has been \emph{killed},
then~\(\beta\) will pick a fresh use block~\(\sB'\) and define
\(\Gamma_s^{E\oplus U}(x) = C_s(x)\) with use block~\(\sB'\) (hence
\(\sB'=\bB_{\bday{s}}(\gamma,x)\)); otherwise, we define \(\Gamma_s^{E\oplus
U}(x)=C_s(x)\) with the use block that is not killed (there will be at most
one such use block).
\end{enumerate}

\subsubsection*{\(\visit(\alpha)\) for an \(R\)-node:}

If \(\tp(\alpha)\) is not defined, then we define it to be a fresh number.
Then we perform \(\enc(\alpha,d)\).

Without loss of generality, we assume that~\(\alpha\) is assigned an
\(R_c(\Phi)\)\nbd requirement for some \(c\in \Ji(\cL)\) and \(\Phi\); let
\(\abs{\seq(\alpha^{-})}=k\).

\subsubsection*{\(\enc(\alpha,d)\):}

If~\(d\) is active, then we perform \(\visit(\alpha^\frown d)\). If~\(d\) is
inactive, we perform \(\enc(\alpha,w)\).

\subsubsection*{\(\enc(\alpha,w)\):}

If \(\dw(\alpha)\) is not defined, then we pick a fresh number
\(x>\tp(\alpha)\) and define \(\dw(\alpha)=x\). If a computation~\(y\) is
found by~\(\alpha\), then we obtain \(\cE_s^\varnothing(\alpha)\) and
\(\cS_s^\varnothing(\alpha)\) (Definition~\ref{def:2U 0 data}). Then we
perform \(\enc(\alpha,U_{k-1})\), where~\(U_{k-1}\) is the first outcome
(recall from Definition~\ref{def:2Utree} that we add outcomes in order).

\subsubsection*{\(\enc(\alpha,U_i)\):}

Inductively we must have obtained \(\cE^{i+1}(\alpha)\) (or
\(\cE^\varnothing(\alpha)\) if \(i=k-1\)).
\begin{enumerate}
\item
If the \(U_i\)\nbd outcome is Type~I, then let \(v=\tp(\alpha)\). For each
functional~\(\Gamma\) (where~\(\Delta\) is dealt with similarly) that
belongs to \(\kl(\alpha,U_{\ge i})\) (Definition~\ref{def:kill2U}) and for
each~\(x\) with \(v\le x\le s\), let \(\sB_x=\bB_s(\gamma,x)\) be the use
block. We enumerate an unused point (\emph{killing point}) into~\(\sB_x\)
and declare that~\(\sB_x\) is \emph{killed}.

Let~\(\Delta\) belong to \(\mt(\alpha,U_i)\). Without loss of generality,
we assume this functional is to ensure \(\Delta^{E\oplus C_0\oplus \cdots
\oplus C_{r-1}} = U_i\) for some \(c_0, \ldots, c_{r-1}\in \Ji(\cL)\).

\begin{enumerate}
\item
Suppose that \(\Delta_s^{E\oplus C_0\oplus \cdots \oplus
C_{r-1}}(x)\downarrow=U_s(x)\). Then~\(\beta\) does nothing else.
\item
Suppose we have \(\Delta_s^{E\oplus C_0\oplus \cdots \oplus
C_{r-1}}(x)\downarrow\neq U_s(x)\) with use block
\(\sB=\bB_s(\delta,x)\).
\begin{enumerate}
\item
If~\(\sB\) is killed and \(E\)\nbd free, then~\(\beta\) enumerates an
unused point, referred to as a \emph{killing point}, into~\(\sB\)
via~\(E\). Then we go to~(3) immediately.
\item
If~\(\sB\) is killed and \(E\)\nbd restrained, we let~\(C_i\) (\(i<k\))
be a set such that~\(\sB\) is \(C_i\)\nbd free (we will show
such~\(C_i\) exists); then~\(\beta\) enumerates an unused point,
referred to as a \emph{killing point}, into~\(\sB\)
via~\(C_i\).~\(\sB\) is then \emph{permanently killed} (as~\(C_i\) will
be a c.e.\ set). Then we go to~(3) immediately.
\item
If~\(\sB\) is not killed and \(E\)\nbd free, then~\(\beta\) enumerates
an unused point, referred to as a \emph{correcting point}, into~\(\sB\)
via~\(E\). Then we define \(\Delta_s^{E\oplus C_0\oplus \cdots \oplus
C_{r-1}}(x)=U_s(x)\) with the same use block~\(\sB\).
\item
If~\(\sB\) is not killed and \(E\)\nbd restrained, we let~\(C_i\) for
some \(i<k\) be a set such that~\(\sB\) is \(C_i\)\nbd free (we will
show such~\(C_i\) exists); then~\(\beta\) enumerates an unused point,
referred to as a \emph{correcting point}, into~\(\sB\) via~\(C_i\).
Then we define \(\Delta_s^{E\oplus C_0\oplus \cdots \oplus
C_{r-1}}(x)=U_s(x)\) with the same use block~\(\sB\).
\end{enumerate}
\item
Suppose that \(\Delta_s^{E\oplus C_0\oplus \cdots \oplus
C_{r-1}}(x)\uparrow\). If for each \(t<s\), \(\bB_{\bday{t}}(\delta,x)\)
is killed, then~\(\beta\) will choose a fresh use block~\(\sB'\) and
define
\[
\Delta_s^{E\oplus C_0\oplus \cdots \oplus C_{r-1}}(x) = U_s(x)
\]
with use block~\(\sB'\) (hence \(\sB'=\bB_{\bday{s}}(\delta,n)\));
otherwise, we will define \(\Delta_s^{E\oplus C_0\oplus \cdots \oplus
C_{r-1}}(x) = U_s(x)\) with the use block that is not killed (there will
be at most one such use block).
\end{enumerate}
Then we stop the current substage and perform \(\visit(\alpha^\frown U_i)\)
for the \(S\)\nbd node \(\alpha^\frown U_i\).

\item
If the \(U_i\)\nbd outcome is GREEN\@, then let \(v=\tp(\alpha)\). For each
functional~\(\Gamma\) (where~\(\Delta\) is dealt with similarly) that
belongs to \(\kl(\alpha,U_{\ge i})\) and for each~\(x\) with \(v\le x\le
s\), let \(\sB_x=\sB_s(\gamma,x)\) be the use block. We enumerate an unused
point (\emph{killing point}) into~\(\sB_x\) and say~\(\sB_x\) is
\emph{killed}. Then we stop the current substage and perform \(\visit
(\alpha^\frown U_i)\) for the \(S\)\nbd node \(\alpha^\frown U_i\).
\item
If the \(U_i\)\nbd outcome is RED\@, then we obtain \(\cE_s^i(\alpha)\) and
\(\cS_s^i(\alpha)\) by Definition~\ref{def:2U weak data}. Then we perform
\(\enc(\alpha,U_{i-1})\) if \(i>0\); or \(\enc(\alpha,\ctr)\) if \(i=0\).
\end{enumerate}

\subsubsection*{\(\enc(\alpha,\ctr)\):}

Notice that we must have obtained~\(\cE^0(\alpha)\) and~\(\cS^0(\alpha)\).
Suppose that~\(\alpha\) is an \(R_c\)\nbd node for some \(c\in \Ji(\cL)\).
\begin{enumerate}
\item
Let \(\cE^{\ctr}(\alpha)\) and \(\cS^{\ctr}(\alpha)\) be obtained by
Definition~\ref{def:modified data and decision map}.
\item
Let~\(\alpha\) become a controller.
\item
We enumerate, for each \(\xi\in \cE^{\ctr}(\alpha)\) and each~\(i\), the
diagonalizing witness into the set~\(C\) if~\(\xi\) is not a (critical)
\(U_i\)\nbd problem (i.e.\@, \(\xi\) is also an \(R_c\)\nbd node).
\item
While~\(\alpha\) is a controller, we put a restraint on \(\hat{C}\res
s_\alpha^{\ctr}\) for each \(\hat{c}\neq c\), where \(s_\alpha^{\ctr}\) is
the current stage~\(s\).
\end{enumerate}
Then we stop the current stage.

\subsubsection*{controller-strategy:}

Let~\(\beta\) (if any) be a controller of highest priority such
that~\(\beta\) sees some noise (Definition~\ref{def:noise}). We initialize
all nodes to the right of \(\beta^\frown \ctr\). Suppose~\(\beta\) is an
\(R_c\)\nbd node.
\begin{enumerate}
\item
If~\(\beta\) sees also some threats (Definition~\ref{def:threats}), then we
also initialize \(\beta^\frown \ctr\) (i.e.\@, we
discard~\(\cE^{\ctr}(\beta)\) and~\(\cS^{\ctr}(\beta)\)).
\item
If~\(\beta\) does not change its decision, then we do nothing.
\item\label{it:2U controller 3}
If~\(\beta\) changes its decision and (by Lemma~\ref{lem:decision}) \(\cD_s
(\beta)=\xi\), then we
restore~\(y_\xi\) and set a restraint on \(E\res y_\xi\) (each use block
\(\sB<y_\xi\) becomes \(E\)\nbd restrained) until the next time~\(\beta\)
changes its decision. Furthermore,
\begin{enumerate}
\item\label{it:2U controller a}
Suppose that \(\xi\in \cE^{\ctr}(\beta)\) is not a critical \(U_i\)\nbd
problem for~\(\beta\). Then we activate the \(d\)\nbd outcome of~\(\xi\).
\item\label{it:2U controller b}
Suppose that~\(\xi\) is a critical \(U_i^b\)\nbd problem for~\(\beta\)
and \(b>0\). Then we turn the GREEN \(U_i\)\nbd outcome of~\(\xi^{\sharp
i}\) RED (once~\(\beta\) changes its decision or is initialized, it turns
back to GREEN).
\item\label{it:2U controller c}
Suppose that~\(\xi\) is a critical \(U_i^0\)\nbd problem for~\(\beta\).
Let \(\xi=\alpha_0\subsetneq \cdots \subsetneq \alpha_{m-1}\) be a
sequence of nodes where each~\(\alpha_j\) is a critical \(U_i^j\)\nbd
problem for some controller and all nodes have the same strong \(<i\)\nbd
environment (see Lemma~\ref{lem:2U strong env}). We
pick two nodes \(\alpha_j\subsetneq \alpha_k\) which are both \(R_d\)\nbd
nodes for some \(d\in \Ji(\cL)\). We obtain the strong \(U_i\)\nbd data
\(\cE^i(\alpha_k)\) and~\(\cS^i(\alpha_k)\) (Definition~\ref{def:2U
strong data}) and establish a \(U_i\)\nbd link starting from the
\(S\)\nbd node \(\env^{<i}(\alpha_k)\) and ending at the \(U_{i-1}\)\nbd
outcome of~\(\alpha_k\) if \(i>0\); or the \(\ctr\)\nbd outcome if
\(i=0\). The \(U_i\)\nbd link will be destroyed once traveled or
\(\beta\) changes decision.
\end{enumerate}
\end{enumerate}



We remark that we did not explicitly mention how big a use block should be
just to avoid getting distracted by this technical issue. See
Lemma~\ref{lem:2U block size}.

\subsection{The verification}\label{sec:2U verification}

We have to show that the size of each use block can actually be chosen
sufficiently large so that this construction does not terminate unexpectedly.

\begin{lemma}[Block size]\label{lem:2U block size}
Each use block can be chosen sufficiently large.
\end{lemma}

\begin{proof}
We refer the reader to Lemma~\ref{lem:block size}. Consider a use
block~\(\sB\) interacting with a controller~\(\beta\) with \(y_\beta<\sB\).
Recall that the number of times that \(B=[a,b)\) is injured and then restored
depends on the number of changes that \(U\res a\) can have. Since we have
multiple \(U\)\nbd sets now, we need to be more careful. Suppose that~\(B\)
is interacting with a controller~\(\beta\) such that \(U_0,\ldots, U_{k-1}\)
are relevant to~\(\beta\). According to the decision function, the number of
times that \(B=[a,b)\) is injured and then restored is bounded by the size of
the following set
\[
S=\{s\mid \diff(U_i,a,s-1,s) \text{for some \(i<k\)}\},
\]
which can be bounded by a computable function \(p(a,k)\).

We caution the reader that when we define the use block~\(\sB\)
at~\(\alpha\), we do not know where the controller~\(\beta\) is located and
which \(U\)\nbd sets are relevant to~\(\beta\). In particular, the number
\(k=\abs{\seq(\beta^-)}\) is unknown to~\(\alpha\). Therefore, we have to
prepare ahead of time. We simply define~\(B\) to be \([a,a+p(a,a))\) for some
fresh number~\(a\) and~\(B\) will be sufficiently large. More precisely, we
assume that \(y_\beta>k\) by taking \(\max\{k+1,y_\beta\}\) as the
value~\(y_\beta\). Therefore \(k<y_\beta<a<b\) and \(\sB\) is sufficiently
large.
\end{proof}

The following lemma justifies the first line of the controller
strategy~(\ref{it:2U controller 3}), which requires that \(\cD_s(\beta)\) is
always defined.

\begin{lemma}[decision]\label{lem:2U decision}
Let~\(\beta\) be a controller. For each
\(s>s_\beta^{\ctr}(\beta)\),~\(\cD_s(\beta)\) is defined.
\end{lemma}

\begin{proof}
We refer the reader to Definition~\ref{def:modified data and decision map}
for \(\cE^{\ctr}(\beta)\).
Our goal is to show that for each \(s>s_\beta^{\ctr}(\beta)\),
\[
    \{\xi\in \cE^{\ctr}(\beta)\mid \Cond^i(\xi)(s)=1
\text{ for each~}i\}\neq \varnothing,
\]
which then implies that \(\cD_s(\beta)\) is defined.

The proof proceeds by a straightforward induction with the following claim:

\emph{Claim. Given \(\sigma\in \cS^{\ctr}(\beta)\) with \(\abs{\sigma}=i\).
Let \(y_*=\max\{y_\tau\mid \sigma 1\subseteq \tau\in \cS^{\ctr}(\beta)\}\)
and \(s_*\) be the stage when \(\cE^{i}(h(\sigma))=\cE^{i+1}(h(\sigma
0))\otimes_ {s_*} \cE^{i+1}(h(\sigma 1))\) is obtained. If
\(\same(U_i,y_*,s_*,s)\) holds, then for each \(\tau\) with \(\sigma
1\subseteq \tau\), \(\Cond^i(h(\tau))(s) =1\); if \(\diff(U_i,y_*,s_*,s)\)
holds, then for each \(\tau\) with \(\sigma 0\subseteq \tau\),
\(\Cond^i(h(\tau))(s)=1\).}

\emph{Proof of the claim.}
The proof follows directly from Definition~\ref{def:2U oplus data} and
Definition~\ref{def:modified data and decision map}.
\end{proof}

The following lemma with \(v=0\) justifies the controller
strategy~\eqref{it:2U controller c}.

\begin{lemma}\label{lem:2U strong env}
Suppose that at stage~\(s\) we have \(\cD_s(\beta)=\xi\) for a controller
strategy~\(\beta\) and that~\(\xi\) is a critical \(U_i^v\)\nbd problem
for~\(\beta\) (where \(v<m\)). Then there exists a sequence of nodes
\[
\xi=\alpha_v\subsetneq \alpha_{v+1} \subsetneq \cdots \subsetneq \alpha_{m-1},
\]
where each~\(\alpha_k\) is a critical \(U_i^k\)\nbd problem for some
controller~\(\beta_k\) with \(\cD_s(\beta_k)=\alpha_k\), and each pair of
them has the same strong \(<i\)\nbd environment
(Definition~\ref{def:strong env}).
\end{lemma}

\begin{proof}
For \(v=m-1\), the lemma holds vacuously. So assume \(v<m-1\) and suppose
that the lemma holds for \(v+1\), we will show that it holds for~\(v\).

Let \(\beta\) be the controller as in the hypothesis of the lemma, and
\(\cS^{\ctr}(\beta)\) be the data tree of it. Since \(\xi=\alpha_v\) is a
critical \(U_i^v\)\nbd problem of \(\beta\), there exists some \(\sigma\in
\cS^{\ctr}(\beta)\) such that \(\alpha_v=h(\sigma 0)\) (by the remark after
Definition~\ref{def:U_problem}). Let \(\eta=h(\sigma 1)=h(\sigma)\),
then~\(\eta^{*i} =\alpha_v\). Moreover, the data
\(\cE^{i}(\eta)=\cE^{i+1}(\eta)\otimes \cE^{i+1} (\alpha_v)\) belongs to
\(\cE^{\ctr}(\beta)\). As a consequence, we have that
\begin{enumerate}
\item\label{it:2Uit1}
\({\alpha_v}^\frown U_i \subseteq \eta \subseteq \beta\);
\item\label{it:2Uit2}
the \(U_i\)\nbd outcome of \(\eta\) is a RED outcome; and
\item\label{it:2Uit3}
\(\eta^\frown U_i\) is to the left of \(\beta^\frown {\ctr}\) (allowing
\(\eta=\beta\)).
\end{enumerate}
From Item~(\ref{it:2Uit2}), there exists some controller \(\beta_{v+1}\) with
\(\cD(\beta_{v+1})=\alpha_{v+1}\) such that \(\alpha_{v+1}^{\sharp i}=\eta\)
and \(\alpha_{v+1}\) is a critical \(U_i^{v+1}\)\nbd problem for \(\beta_{v+1}
\). By induction hypothesis, we can find a sequence of nodes
\[
\alpha_{v+1}\subsetneq \alpha_{v+2}\subsetneq \cdots \subsetneq \alpha_{m-1}
\]
such that the lemma holds. By the remark below Definition~\ref{def:env}, we
have that \(\env^{<i}(\alpha_v)=\env^ {<i}(\alpha_{v+1})\). It remains to
show that~\(\alpha_v\) and~\(\alpha_{v+1}\) satisfy Item~(\ref{it:strong env
2}) in Definition~\ref {def:strong env}.

Suppose towards a contradiction that there exists some \(\rho\) such that
\(\alpha_v^\frown U_i\subseteq \rho^\frown U_j \subseteq \alpha_{v+1}\),
where \(j>i\) and the \(U_j\)\nbd outcome is RED. Let \(\chi\) be the
controller which is responsible for turning this \(U_j\)\nbd outcome RED;
i.e., \(\cD(\chi)\) is a critical \(U_j\)\nbd problem and
\((\cD(\chi))^{\sharp j}=\rho\).

By Items~(\ref{it:2Uit1}) and~(\ref{it:2Uit3}), either \(\alpha_v^\frown
U_i\subseteq \rho^\frown U_j \subseteq \eta\subseteq \beta\) or \(\eta^\frown
U_i\subseteq \rho^\frown U_j\subseteq \alpha_{v+1}\). In the first case, if
\(\chi^\frown \ctr\) is to the \emph{left} of \(\beta^\frown \ctr\), then the
moment~\(\chi\) turns the \(U_j\)\nbd outcome of~\(\rho\) RED,
we will not visit any node below \(\rho^\frown U_j\), and so in particular
not~\(\beta\). This contradicts that \(\beta\) is a working controller at the
current stage; if \(\chi^\frown \ctr\) is to the \emph{right} of
\(\beta^\frown \ctr\), then as soon as \(\cD(\beta)=\alpha_v\), \(\chi^\frown
\ctr\) is initialized and the \(U_j\)\nbd outcome of \(\rho\) is reset to
GREEN, a contradiction. In the second case, we have both the~\(U_i\)- and the
\(U_j\)\nbd outcomes RED. If \(\chi^\frown \ctr\) is to the \emph{left} of
\({\beta_{v+1}}^\frown \ctr\), the moment \(\chi^\frown \ctr\) turns the
\(U_j\) \nbd outcome of \(\rho\) RED, \(\beta_{v+1}\) is initialized and we
will never visit nodes below \(\rho^\frown U_j\), in particular
never~\(\beta_{v+1}\), contradicting that \(\beta_{v+1}\) is a working
controller; if \(\chi^\frown \ctr\) is to the \emph{right} of
\({\beta_{v+1}}^\frown \ctr\), then as soon as \(\cD
(\beta_{v+1})=\alpha_{v+1}\), we have that~\(\chi\) is initialized and the
\(U_j\)\nbd outcome of \(\rho\) is reset to GREEN, a contradiction.

This completes the proof.
\end{proof}

The following three lemmas justify the restoration part of the controller
strategy~(\ref{it:2U controller 3}):

\begin{lemma}\label{lem:induction step}
Suppose \(\cE^i(\beta)=\cE^{i+1}(\alpha)\otimes_s \cE^{i+1}(\beta)\)
(Definition~\ref{def:2U oplus data}), where \(\alpha\subsetneq \beta\) and
\(\alpha,\beta\) are \(R_d, R_c\)\nbd nodes for some \(d\le c\in \Ji(\cL)\),
respectively.
\begin{enumerate}
\item
For each \(\xi\in \cE^{i+1}(\alpha)\),
\(\Cond^i(\xi)(t)=\diff(U_i,y_\gamma,s,t)\) (where~\(y_\gamma\) is the
\(U_i\)\nbd reference length for \(\xi\)) and
\(\bigsame(\hat{D},y_{\beta^*},s,t)\) for each \(\hat{d}\ngeq d\) implies
that~\(\xi\) is \((U,D)\)\nbd restorable (Definition~\ref{def:U
restorable}) at \(t>s\).
\item
For each \(\xi\in \cE^{i+1}(\beta)\), if we have
\(\Cond^i(\xi)(t)=\same(U_i,y_\xi,s,t)\) and
\(\bigsame(\hat{C},y_\beta,s,t)\) for each \(\hat{c}\ngeq c\), then~\(\xi\)
is \((U,C)\)\nbd restorable.
\end{enumerate}
\end{lemma}

\begin{proof}
If~\(\cE^i(\beta)\) is weak \(U_i\)\nbd data (Definition~\ref{def:2U weak
data}), then the lemma follows as in
Lemma~\ref{lem:U restorable}. If~\(\cE^i(\beta)\) is strong \(U_i\)\nbd data
(Definition~\ref{def:2U strong data}), then the lemma follows as in
Lemma~\ref{lem:U restorable strong}.
\end{proof}

\begin{lemma}
Let~\(\beta\), an \(R_c\)\nbd node for some \(c\in \Ji(\cL)\), be a
controller with \(\cE^{\ctr}(\beta)\) and \(\cS^{\ctr}(\beta)\). Suppose
\(\cD_s(\beta)=\xi\) where \(\xi=h(\sigma)\) for some leaf \(\sigma\in
\cS^{\ctr}(\beta)\). If~\(\xi\) is an \(R_c\)\nbd node, then~\(\xi\) is
\((U_i,C)\)\nbd restorable for each \(i<\abs{\sigma}\) (in this case, we say
that~\(\xi\) is \emph{restorable}). If~\(\xi\) is an \(R_d\)\nbd node for
some \(d<c\), then~\(\xi\) is \((U_i,D)\)\nbd restorable for each
\(i<\abs{\sigma}\) (in this case, we say that~\(\xi\) is \emph{weakly
restorable}).
\end{lemma}

\begin{proof}
Induction on \(\cS^{\ctr}(\beta)\), using
Lemma~\ref{lem:induction step}.
\end{proof}

Stated another way, we have the following

\begin{lemma}\label{lem:2U main lemma}
Let~\(\beta\) be a controller where~\(\beta\) is an \(R_c\)\nbd node for some
\(c\in \Ji(\cL)\). Let \(\xi\in \cE^{\ctr}(\beta)\).
\begin{enumerate}
\item
Suppose that~\(\xi\) is an \(R_c\)\nbd node. If \(\cD_s(\beta)=\xi\), then
we can restore~\(y_\xi\) at the beginning of stage~\(s\) and activate the
\(d\)\nbd outcome of~\(\xi\) and~\(y_\xi\) remains restored at each
substage of stage~\(s\).
\item
Suppose that~\(\xi\) is an \(R_d\)\nbd node with \(d<c\). So~\(\xi\) is a
critical \(U_i^b\)\nbd problem for some~\(i\) and some \(b>0\). If
\(\cD_s(\beta)=\xi\), then we can restore~\(y_\xi\) at the beginning of
stage~\(s\) and turn the GREEN \(U_i\)\nbd outcome of \(\xi^{\sharp i}\)
RED. Furthermore,~\(y_\xi\) remains restored at each substage of
stage~\(s\). \qed
\end{enumerate}
\end{lemma}

With Lemma~\ref{lem:2Utree}, Lemma~\ref{lem:2U block size}, Lemma~\ref
{lem:2U decision} and Lemma~\ref{lem:2U main lemma}, the following four
lemmas have essentially the same proofs as Lemma~\ref{lem:1U finite
initialization lemma}, Lemma~\ref{lem:1U R is satisfied}, Lemma~\ref{lem:1U S
is satisfied}, and Lemma~\ref{lem:1U G is satisfied}, respectively.

\begin{lemma}[Finite Initialization Lemma]
Let~\(p\) be the true path. Each node \(\alpha\in p\) is initialized finitely
often and~\(p\) is infinite. \qed
\end{lemma}

\begin{lemma}
The \(R^e\)\nbd requirement is satisfied for each \(e\in \omega\). \qed
\end{lemma}

\begin{lemma}
The \(S(U_i)\)\nbd requirement is satisfied for each \(i\in \omega\). \qed
\end{lemma}

\begin{lemma}
The \(G\)\nbd requirement is satisfied. \qed
\end{lemma}

\section{Final remarks}
Lemma~\ref{lem:2U block size} remains true when our \(U\)\nbd sets are
\(\omega\)\nbd c.e.\ sets. In fact, we can list all \(\omega\)\nbd c.e.\ sets
(up to Turing degree) as a subsequence of the sets \(\{U_e\}_{e\in \omega}\)
where \(U_e(x)=\lim_ {y\to \infty} \Phi_e(x,y)\). \(U\) is a valid \(\omega\)
\nbd c.e.\ set if \(\abs{\{y\mid \Phi_e(x,y-1)\neq \Phi_e(x,y)\}}\le x\); and
invalid otherwise. Our construction is almost unaffected --- whenever we
detect for some~\(x\) that \(\abs{\{y\mid \Phi_e(x,y-1)\neq \Phi_e(x,y)\}} >
x\), then~\(S(U_e)\) wins by a finite outcome (we add this additional outcome
to an \(S\)\nbd node), claiming that it is an invalid requirement and does
not need to be satisfied. This yields

\begin{theorem}
If~\(\cL\) is a finite distributive lattice, then~\(\cL\) can be embedded
into \(\omega\)\nbd c.e.\ degrees as a final segment. \qed
\end{theorem}

\end{document}